\documentclass[12pt]{amsart}

\usepackage{amsfonts}
\usepackage{amsthm}
\usepackage{latexsym}
\usepackage{amsmath}
\usepackage{amssymb}
\usepackage{amscd}
\usepackage{amsmath}
\usepackage{mathrsfs}
\usepackage{graphicx}
\usepackage{hyperref}
\usepackage{a4wide}
\usepackage{color}
\usepackage{tikz-cd}
\usepackage{comment}

\makeatletter
\CheckCommand*\refstepcounter[1]{\stepcounter{#1}%
      \protected@edef\@currentlabel
       {\csname p@#1\endcsname\csname the#1\endcsname}%
  }
\renewcommand*\refstepcounter[1]{\stepcounter{#1}%
    \protected@edef\@currentlabel
      {\csname p@#1\expandafter\endcsname\csname the#1\endcsname}%
  }
\def\labelformat#1{\expandafter\def\csname p@#1\endcsname##1}
\DeclareRobustCommand\Ref[1]{\protected@edef\@tempa{\ref{#1}}%
     \expandafter\MakeUppercase\@tempa
  }
\makeatother

\makeatletter
  \newcommand{\numberlike}[2]{%
     \expandafter\def\csname c@#1\endcsname{%
         \expandafter\csname c@#2\endcsname}
  }
\makeatother

\def\DefaultNumberTheoremWithin{section}

\theoremstyle{plain}
\newtheorem{lemma}{Lemma}
     \numberwithin{lemma}{\DefaultNumberTheoremWithin}
     \labelformat{lemma}{Lemma~#1}

     \numberwithin{claim}{\DefaultNumberTheoremWithin}
     \numberlike{claim}{lemma}
     \labelformat{claim}{Claim~#1}
\newtheorem{theorem}{Theorem}
     \numberwithin{theorem}{\DefaultNumberTheoremWithin}
     \numberlike{theorem}{lemma}
     \labelformat{theorem}{Theorem~#1}
\newtheorem{corollary}{Corollary}
     \numberwithin{corollary}{\DefaultNumberTheoremWithin}
     \numberlike{corollary}{lemma}
     \labelformat{corollary}{Corollary~#1}
\newtheorem{proposition}{Proposition}
     \numberwithin{proposition}{\DefaultNumberTheoremWithin}
     \numberlike{proposition}{lemma}
     \labelformat{proposition}{Proposition~#1}
\newtheorem{conjecture}{Conjecture}
     \numberwithin{conjecture}{\DefaultNumberTheoremWithin}
     \numberlike{conjecture}{lemma}
     \labelformat{conjecture}{Conjecture~#1}

\theoremstyle{definition}
\newtheorem{definition}{Definition}
     \numberwithin{definition}{\DefaultNumberTheoremWithin}
     \numberlike{definition}{lemma}
     \labelformat{definition}{Definition~#1}

\theoremstyle{definition}

     \numberwithin{question}{\DefaultNumberTheoremWithin}
     \numberlike{question}{lemma}
     \labelformat{question}{Question~#1}

\theoremstyle{definition}

     \numberwithin{problem}{\DefaultNumberTheoremWithin}
     \numberlike{problem}{lemma}
     \labelformat{problem}{Problem~#1}

\theoremstyle{remark}

     \numberwithin{remark}{\DefaultNumberTheoremWithin}
     \numberlike{remark}{lemma}
     \labelformat{remark}{Remark~#1}
\theoremstyle{remark}

\newtheorem{example}{Example}
     \numberwithin{example}{\DefaultNumberTheoremWithin}
     \numberlike{example}{lemma}
     \labelformat{example}{Example~#1}

     \labelformat{case}{Case~#1}
     \numberwithin{case}{lemma}

     \labelformat{step}{Step~#1}
     \numberwithin{step}{lemma}
     
\theoremstyle{plain}
\newtheorem{THEO}{Theorem}
     
     \labelformat{THEO}{Theorem~#1}

\newtheorem{PROP}{Proposition}
     \numberlike{PROP}{THEO}
     \labelformat{PROP}{Proposition~#1}

\newtheorem{COR}{Corollary}
     \numberlike{COR}{THEO}
     \labelformat{COR}{Corollary~#1}

     \numberlike{LEM}{THEO}
     \labelformat{LEM}{Lemma~#1}

\numberwithin{equation}{section}

\labelformat{equation}{(#1)}
\labelformat{figure}{Figure~#1}
\labelformat{chapter}{Chapter~#1}
\labelformat{appendix}{Appendix~#1}
\labelformat{section}{\S~#1}
\labelformat{subsection}{\S~#1}
\labelformat{subsubsection}{\S~#1}

\def\A{\mathbb A}
\def\C{\mathbb C}

\def\N{\mathbb N}
\def\P{\mathbb P}

\def\R{\mathbb R}
\def\Z{\mathbb Z}
\def\ZZ{\mathbb Z}

\def\d{\partial}

\def\e{\epsilon}
\def\b{\beta}
\def\a{\alpha}

\def\cP{{\mathcal P}}
\def\cB{{\mathcal B}}
\def\D{{\mathcal D}}

\def\sI{{\mathsf I}}
\def\sM{{\mathsf M}}
\def\sR{{\mathsf R}}

\def\bfc{{\mathbf c}}
\def\Om{{\mathbf\Omega}}

\def\om{\omega}

\def\im{{\mathsf{im}}}
\def\ker{{\mathsf{ker}}}
\def\Sym{{\mathrm{Sym}}}
\def\Conf{{\mathrm{Conf}}}
\def\disc{{\mathrm{disc}}}

\begin{document}

\title[Spaces of polynomials with constrained real divisors, II.] {Spaces of  polynomials with constrained real divisors, II.  (Co)homology \& stabilization}    

\author[G.~Katz]{Gabriel Katz}
\address{5 Bridle Path Circle, Framingham, MA 01701, USA}
\email {gabkatz@gmail.com }

\author[B.~Shapiro]{Boris Shapiro}
\address{Stockholm University, Department of Mathematics, SE-106 91
Stockholm, Sweden and Institute for Information Transmission Problems, Moscow, Russia}
\email {shapiro@math.su.se }

\author[V.~Welker]{Volkmar Welker}
\address{Philipps-Universit\"at Marburg, Fachbereich Mathematik und Informatik, 35032 Marburg, Germany}
\email {welker@mathematik.uni-marburg.de}

\begin{abstract} 
  In the late 80s, V.~Arnold and V.~Vassiliev initiated the topological study 
   of the space of real univariate polynomials of a given 
  degree which  have no real roots of multiplicity exceeding a given positive 
  integer. Expanding their studies, we consider the spaces
  $\cP^{\bfc \Theta}_d$ of real monic univariate  polynomials of degree 
  $d$ whose real divisors avoid given sequences of root multiplicities. These forbidden sequences are taken from an arbitrary poset $\Theta$ of compositions that is closed under certain natural combinatorial operations. We reduce the computation of the homology $H_\ast(\cP^{\bfc \Theta}_d)$ to the computation of the homology of a differential complex, defined purely combinatorially in terms of the given closed poset $\Theta$. 
 We also obtain stabilization results for  $H^\ast(\cP^{\bfc \Theta}_d)$, as $d \to \infty$.
  
 These results are deduced from our description of the homology of spaces 
 $\cB^{\bfc \Theta}_d$ whose points are binary real homogeneous forms, considered  
 up to projective equivalence, with similarly $\Theta$-constrained real divisors. 
 Namely, we exhibit differential complexes that calculate the homology
 of these spaces and obtain some stabilization results for $H^\ast(\cB^{\bfc \Theta}_d)$, as $d \to \infty$. In particular, we compute the homology of the discriminants of projectivized binary real forms 
 and of their complements in $\cB_d \cong \R\P^d$. 
\end{abstract}

\date{\today}

\maketitle

\tableofcontents

\setcounter{page}{1}


\section{Introduction} \label{sec:intro}

In this paper, we conduct  parallel investigations of the topology of spaces 
of real univariate polynomials and of projectivized spaces of real binary forms with some
\emph{constraints} on their real divisors. The constraints we impose on the  
real divisors are of a quite general combinatorial nature. Denoting the 
collection of such constraints by $\Theta$, we are able to describe 
differential complexes that compute the (co)homology of these spaces of
polynomials/forms entirely in terms of the combinatorics of $\Theta$.
These differential complexes, which mimic possible bifurcations of real roots, 
seem to be of independent interest. The rich combinatorics, arising as a 
byproduct of our constructions, calls for computer experiments. The results of 
such experiments are assembled in several tables, spread through the text. \smallskip

In \cite{KSW1}, part I of the present paper, we have obtained results about 
the fundamental groups of spaces of real polynomials with the 
$\Theta$-restricted root patterns. These results are in the spirit of 
\ref{th:Arnold3}, stated below.  
\smallskip

In what follows, theorems, conjectures etc., labelled by letters, are 
borrowed from the existing literature, while those, labelled by numbers, 
are hopefully new.\smallskip

\subsection{Preliminary results}
In \cite{Ar}, V.~Arnold proved  Theorems A and B below. Later, these results of Arnold
were generalized by V.~Vassiliev. In particular, he described the 
multiplication in the cohomology of the relevant spaces, see \cite{Va, Va0}. 
All those papers study the topology of spaces of smooth 
functions/polynomials, which Arnold calls functions/polynomials with 
``\emph{moderate singularities}.'' These works of Arnold and Vassiliev are 
the main sources of inspiration for our study of functions/polynomials 
with ``\emph{less moderate singularities}.''  \smallskip

In the formulation of Arnold's theorems, we keep the original notation of 
\cite{Ar}, which we will abandon later on. 
For $1 \leq k \leq d$, let $G_k^d$ be the space  of 
real monic polynomials $x^d + a_{d-1}x^{d-1} + \cdots + a_0 \in \R[x]$
with no real roots of multiplicity $\geq k$.  

\begin{THEO}\label{th:Arnold3}  
  If $k \leq d < 2k-1$, then $G_k^d$ is diffeomorphic to the product of a 
   sphere $S^{k-2}$ by an Euclidean space. In particular,  
  $$\pi_i(G_k^d)\cong \pi_i(S^{k-2}) \textrm{~for all } i.$$
\end{THEO}

An analogous result holds for the space of polynomials whose sum of roots 
vanishes, i.e., the polynomials with the vanishing coefficient $a_{d-1}$.
	
\begin{THEO}\label{th:Arnold4}
  Let $k \geq 3$. The homology groups
  with integer coefficients of the space $G_k^d$
  are nonzero only in the dimensions $r(k-2)$ with $0 \leq r \leq [d/k]$,
  and for these $r$ we have
  $$H_{r(k-2)}(G_k^d) \cong \Z.$$
\end{THEO}


Besides the studies of V.~Arnold \cite{Ar} and V.~Vassiliev \cite{Va, Va0}, 
the second major motivation for this paper comes from results of the 
first author, connecting the cohomology of spaces of real polynomials, that 
avoid the real root patterns from $\Theta$, with the theory of traversing 
flows on manifolds with boundary in \cite{Ka1}, \cite{Ka2}, and \cite{Ka3}. 
For traversing flows (see \cite{Ka1},\cite{Ka2} for the definition) on compact 
manifolds $X$  that avoid a given collection of 
tangency patterns $\Theta$ of their trajectories to the boundary $\d X$, the 
spaces  of polynomials, avoiding the root patterns from $\Theta$, play the 
fundamental role of classifying spaces. This role is quite similar to the one 
played by Gra{\ss}mannians in the category of vector bundles 
\cite{Ka4}, \cite{Ka5}.\smallskip

\subsection{ Our set-up}

To make this paper independent of \cite{KSW1},  we repeat below some basic definitions, notations, and results from \cite{KSW1}.

Let $\cP_d$ denote the space of real monic univariate polynomials of degree 
$d$. Given a polynomial $P(x) = x^d + a_{d-1}x^{d-1} + \cdots + a_0 \in \cP_d$, 
we define its \emph{real divisor} $D_\R(P)$ as the multiset 
$$\underbrace{x_1 = \cdots = x_{i_1}}_{\omega_1 = i_1} < 
\underbrace{x_{i_1+1} = \cdots = x_{i_2}}_{\omega_2 = i_2-i_1} <
\cdots < \underbrace{x_{i_{\ell-1} +1} = \cdots = x_{i_\ell}}_{\omega_\ell
= i_\ell-i_{\ell-1}}$$
of the real roots of $P(x)$. The ordered $\ell$-tuple 
$\om = (\omega_1 ,\ldots, \omega_\ell)$ of natural numbers describes the {\it 
real root multiplicity pattern} of $P(x)$. 
Let $\mathring{\sR}^\omega_d$ be  the set of all polynomials in $\cP_d$ with real root multiplicity pattern $\omega$, and let
$\sR_d^\omega$ be the closure of $\mathring{\sR}^\omega_d$ in $\cP_d$. 

For a given collection $\Theta$ of root multiplicity patterns, we 
consider the union $\cP_d^\Theta \subset \cP_d$ of the subspaces $\mathring{\sR}_d^{\omega}$, taken 
over all $\omega \in \Theta$. We denote by $\cP_d^{\bfc\Theta} := \cP_d \setminus \cP^\Theta_d$ 
its complement. Our studies are restricted to the case when 
$\cP_d^\Theta$ is \emph{closed} in $\cP_d$ and we call such collections $\Theta$ \emph{closed}.

As has been observed in \cite[Lemma 1.2]{KSW1}, for any collection $\Theta$ 
of root multiplicity patterns which contains the pattern $(d)$, the space 
$\cP_d^\Theta$ is \emph{contractible}. 
Thus a topologically interesting situation occurs if one considers
the \emph{one-point compactification} $\bar{\cP}_d^\Theta$ of $\cP_d^\Theta$. 
For a closed $\Theta$, the latter is the union of the one-point 
compactifications $\bar{\sR}_d^\omega$ of the spaces $\sR_d^\omega$, taken 
over $\omega \in \Theta$ with points at infinity identified. By 
Alexander duality in $\bar{\cP}_d \cong S^{d}$, we get the relation  
$$\tilde H^j(\cP_d^{\bfc\Theta}; \Z) \cong \tilde H_{d - j -1}(\bar{\cP}_d^{\Theta}; \Z)$$
in reduced (co)homology, which  implies that the spaces $\cP_d^{\bfc\Theta}$ and $\bar{\cP}_d^{\Theta}$ 
carry equivalent (co)ho\-mo\-lo\-gi\-cal information.

In particular, for $\Theta$ comprising all $\om$s such that 
at least one entry of $\om$ is greater than or equal to $k$, we have 
that $\Theta$ is closed and 
$\cP_d^{\bfc\Theta} \cong G_k^d$ (see \cite[Example 1.2]{KSW1}).

\subsection{Outline of the main results} 
\smallskip 
In  \ref{sec:cellstructure}, we study the cellular structure of $\cP_d$, given by the cells $\mathring{\sR}_d^\omega$. In parallel, we investigate a cellular structure on the space $\cB_d$ of classes of binary homogeneous forms of degree $d$ with real coefficients up to projective equivalence. The cells of $\cB_d$ are
indexed by pairs consisting of a root pattern $\om$ and a non-negative integer $\kappa$. The integer $\kappa$ represents the root multiplicity at $\infty$. 
As in the polynomial case we consider for a set $\Theta$ of pairs 
$(\om,\kappa)$ the space $\cB_d^\Theta$ of forms satisfying one of the
root multiplicity patterns from $\Theta$ and its complement $\cB_d^{\bfc \Theta}
= \cB_d \setminus \cB_d^\Theta$. 

In \ref{sec:combcomplex}, we turn to (co)homology of the compact spaces $\bar{\cP}_d^{\Theta}$, $\cB_d^{\Theta}$, and their complements $\cP_d^{\bfc\Theta}$, $\cB_d^{\bfc\Theta}$. In \ref{th.3.3}, we construct a combinatorial differential 
complex of free $\Z$-modules that, for a given closed $\Theta$,  calculates the homology of  $\cB_d^\Theta$. 

Further, we can identify the univariate polynomial 
$a_0+\cdots+a_{d-1}x^ {d-1}+x^ d$ of degree $d$ 
with the class of the  homogeneous binary form 
$a_0y^ d+\cdots+a_{d-1}x^ {d-1}y+x^ d$ 
of degree $d$ with no (real) roots at infinity. 
Hence we can identify $\bar{\cP}_d$ with a quotient of $\cB_d$ by the closed 
subspace $\cB_{d-1}$. Based on this identification, for a given closed 
$\Theta$, \ref{cor13.4} introduces the combinatorial complex that calculates 
the homology of the space $\bar{\cP}_d^\Theta$.  The latter complex is
the quotient of the combinatorial complex from \ref{th.3.3} (which
computes the homology of $\cB_d^\Theta$) by the subcomplex of cells with a root at
infinity; its basis consists of the cells corresponding to
univariate degree $d$ polynomials. Via Alexander duality, this complex 
also computes the cohomology groups of $\cP_d^{\bfc \Theta}$.

In \ref{sec:stab}, we discuss what happens to the homology of 
$\cB_d^{\mathbf c\Theta}$ and $\cP_d^{\mathbf c\Theta}$ when $\Theta$ 
is fixed and $d\to \infty$. 
Our main stabilization results are \ref{simple_stabilization},
\ref{pol_simple_stabilization}, \ref{simple_stabilizationc} and 
\ref{pol_simple_stabilizationc}.
For a closed equal parity $\Theta$ and $d$ of the parity of $\Theta$, the first two give
isomorphisms $H_j(\cB_d^{\Theta};\Z)\cong H_{j+2}(\cB_{d+2}^{\Theta};\Z)$ and
$H_j(\bar{\cP}_d^{\Theta};\Z)\cong H_{j+2}(\bar{\cP}_{d+2}^{\Theta};\Z)$ for all $j$ bounded
from below by a function of $d$. The last two claim that the inclusions
$\cB_d^{\mathbf c\Theta}\subset\cB_{d+2}^{\mathbf c\Theta}$ and
$\cP_d^{\mathbf c\Theta}\subset\cP_{d+2}^{\mathbf c\Theta}$ induce isomorphisms of the
$k$\textsuperscript{th} homology groups (with twisted coefficients $\Z^t$ in the projective
case, which are trivial for odd $d$) for all $k$ bounded from above by a function of $d$.
For so-called profinite $\Theta$, the homology of the complements stabilizes in every fixed
degree $k$ for $d$ large enough.

\smallskip
Finally, in \ref{sec:computer}, we present the results of a variety of 
computer experiments calculating the homology of $\bar{\cP}_d^\Theta$ 
for various $\Theta$ and $d$. 

\smallskip
\noindent
\emph{Acknowledgements.} During the work on this paper, we received very 
valuable advice. Dan Petersen and Greg Friedman helped us to prove  
\ref{lem.diagram}. Conversations with Alex Postnikov, Alex Martsinkovsky, 
and Alex Suciu shaped our general perspective and presentation. Our deep 
gratitude goes to all of them.\smallskip

The first author is grateful to the Department of Mathematics of 
Massachusetts Institute of Technology for many years of hospitality.
The second author acknowledges financial support by the Swedish 
Research council through the grant 2016-04416. 
The third author thanks the Department of Mathematics of Stockholm
University for its hospitality in May 2018.
He also was partially supported by an NSF grant DMS 0932078, administered by 
the Mathematical  Sciences  Research  Institute  while  the  author  was  in  
residence  at  MSRI during the complementary program 2018/19.

\section[Combinatorics and geometry of the cell structure]{Combinatorics and geometry of the cell structure on spaces of real polynomials and real binary forms}
\label{sec:cellstructure}




\smallskip 
Let $\cB_d$ denote the space of non-zero real bivariate homogeneous polynomials of the form
$$a_0x^0y^d + \cdots +a_dx^dy^0,$$  being considered  
up to a non-zero real scalar factor. 

We have already associated to a polynomial $P(x) \in \R[x]$ its 
real root multiplicity pattern $(\omega_1,\ldots, \omega_\ell)$. 
Next, we define the combinatorial structure which will govern the cell decomposition of 
the space $\cB_d$. Topologically $\cP_d$ is a real $d$-dimensional affine space, while $\cB_d$ is the real $d$-dimensional projective space $\R\P^d$. 
For  $d' \leq d$, the map $\wp$, given by
\begin{eqnarray}\label{eq.polys-forms}
	\genfrac{}{}{0pt}{}{a_0 + \cdots +a_{d'-1}x^{d'-1}+x^{d'}}{\in \cP_{d'}} 
	\xrightarrow{\wp} \genfrac{}{}{0pt}{}{a_0 x^0y^d+ \cdots +a_{d'-1}x^{d'-1}y^{d-d'+1}+x^{d'}y^{d-d'}}{\in \cB_d,}
\end{eqnarray}
 is a homeomorphism of $\cP_{d'}$ with a 
subspace of $\cB_d$. 
The images for different $d'$ are disjoint.
We obtain the standard decomposition $\cB_d=\cP_d\sqcup \cP_{d-1}\sqcup \cP_{d-2}\sqcup \dots \sqcup \cP_0$. Here $\cP_0$ can be identified with the one-point space of constant non-zero polynomials up to a non-zero factor. 
Let us define the combinatorial structures which capture and refine this natural stratification. 

The space $\cP_d$  may be also identified  with the space of positive degree $d$ divisors on the complex plane $\mathbb{C}$,  invariant under the complex conjugation $J: z \to \bar z$. At the same time, $\cB_d$ may be identified with the space of positive degree $d$ divisors  on the Riemann sphere $\mathbb{CP}^1$  invariant under the complex conjugation $\hat J: [z: w] \to [\bar z: \bar w]$ where $[z:w]$ are homogeneous coordinates on $\mathbb{CP}^1$.

\begin{definition}
  An arbitrary sequence $\om=(\om_1,\dots, \om_\ell)$ of positive integers 
  is called a {\it composition}.
  We say that $\om$ is a composition of
  the number $|\om| := \om_1 + \cdots 
  +\om_\ell$. We also allow the \emph{empty composition} $\om = ()$
  of the number $ |()| = 0$. 
  We call the $\om_i$, $i=1,\ldots, \ell$, the parts of the composition $\om$.

  Similarly, a pair $(\omega,\kappa)$ is called a {\it marked composition} if $\omega$ is a
  composition and $\kappa$ is a non-negative integer. 
  For $\om=(\om_1,\dots, \om_\ell)$, we say that 
  $(\om,\kappa)$ is a marked composition of the number $|(\om,\kappa)| := |\om|+\kappa$.  

  For a composition $\omega$, (resp., a marked composition $(\om,\kappa)$,) we call 
  $|\omega|$ (resp., $|(\om,\kappa)|$) its \emph{norm} and  $|\om|' := |\omega|-\ell$ (resp.,  $|(\om,\kappa)|' =
  |\om|'+\kappa$) its \emph{reduced norm}.

  If $\kappa= 0$, then we often identify the marked composition $(\om,\kappa)$ with $\om$ and speak of the norm $|\om| = |(\om,0)|$ and the reduced norm $|\om|' = |(\om,0)|'$ of
  $\om$.  \hfill $\diamondsuit$
\end{definition}

For $f = a_0x^0y^d + \cdots +a_dx^dy^0 \in \cB_d$, we define its 
{\it real degree} $d'$ as $\max \{ i~|~a_i \neq 0\}$. 
As in \ref{eq.polys-forms}, $f$ can be identified with a polynomial 
$f^\dagger$ in $\cP_{d'}$ such that $\wp(f^\dagger) = f$.
 
Let $[x:y] \in \R\P^1$ be a real root of $f$. If $y \neq 0$, then we can choose $y=1$ and $x$ as a real root of $f^\dagger$. If $y = 0$, then we can choose $x=1$ and identify $(1:0)$ with the point at infinity in $\R\P^1$.  
Thus we can describe the root multiplicities of $f$ by a marked composition $(\omega,\kappa)$, where $\omega$ is a composition with 
$|\omega| \leq d' = d - \kappa$ and $\kappa$ is the multiplicity of the root at infinity.
Any polynomial in $\cP_0\subseteq \cB_d$ is associated with the marked composition $((),d)$. 


\smallskip
Let $(\om,\kappa)$ be a marked composition for which  $d-|(\om,\kappa)|$  is even and non-negative. 
Using $\wp$ from \ref{eq.polys-forms}, we set $\mathring{\sR}^{(\omega,\kappa)}_d := \mathring{\sR}^\omega_{d'} \subseteq \cP_{d'} \subseteq \cB_d$
where $d' = d-\kappa$. 
Analogously  to the univariate case, we write $\sR_d^{(\omega,\kappa)}$ for the closure of $\mathring{\sR}^{(\omega,\kappa)}_d$ in $\cB_{d}$. 
Evidently, for a given composition $\om$, the stratum $\sR^\omega_d$ is empty 
if and only if either $|\om | > d$, or $|\om | \leq d$ and 
$d-|\om |$ is odd. 
Thus if $(\om,\kappa)$ is a marked composition, then 
$\sR^{(\omega,\kappa)}_d$ is empty 
if and only if either $|(\om,\kappa)| > d$ or $|(\om,\kappa)| \leq d$ and $d-|(\om,\kappa)|$ is odd.
Note that for a composition $\om$, we have $\mathring{\sR}_d^\om = \mathring{\sR}_d^{(\om,0)}$. However,  in general,  
$\sR_d^\om \neq  \sR_d^{(\om,0)}$ since the closures are taken in different spaces. 
\smallskip

For a given collection $\Theta$ of marked compositions,  
consider the union $\cB_d^\Theta$ of the subspaces 
$\mathring{\sR}_d^{(\omega,\kappa)}$
over all $(\omega,\kappa) \in \Theta$. 
We denote its complement by $\cB_d^{\bfc\Theta} := \cB_d \setminus \cB^\Theta_d$. Similarly, for a given collection  of compositions $\Theta$, we write
$\cP_d^\Theta$ for the union of the subspaces $\mathring{\sR}_d^\om$ over all 
$\om \in \Theta$. 
Further, set $\cP_d^{\bfc\Theta} := \cP_d \setminus \cP^\Theta_d$.  

\smallskip
We restrict our studies to the case when either 
$\cB_d^\Theta$ is closed in $\cB_d$ or when $\cP_d^\Theta$ is closed in $\cP_d$.
In such situations, we  call the respective collection $\Theta$ \emph{closed}. \smallskip

As observed in \cite[Lemma 1.2]{KSW1}, for any collection $\Theta$ of compositions containing $(d)$,  the space
$\cP_d^\Theta$ is contractible. Since $(d)$ is obtained from every $\om \in \Om_{\langle d]}$ by merges, $\cP_d^\Theta$ is contractible for any closed $\Theta$ with $\cP_d^\Theta \neq \emptyset$.  As a consequence, we concentrate on the {\it one-point compactification} $\bar{\cP}^\Theta_d = \cP_d^\Theta \sqcup \infty$ of $\cP_d^\Theta$. It  has a non-trivial topology related to $\cP_d^{\bfc\Theta}$. 
Note that if $\cP_d^\Theta$ is closed in $\cP_d$, then usually neither
$\cP_d^\Theta$ nor  $\bar{\cP}_d^\Theta$ are closed in $\cB_d$. 
\smallskip

  We denote by $\Om^\infty$ the set of all marked compositions $(\om, \kappa)$.  
  For a given positive integer $d$, we denote by $\Om_{\langle d]}^\infty$ (resp., $\Om_{d}^\infty$) the 
  set of all marked compositions $(\om,\kappa)$, such that $|(\om,\kappa)| \leq d$ and $|(\om,\kappa)| \equiv d \mod 2$
  (resp., $|(\omega,\kappa)| = d$). 
  Analogously we write $\Om$ for the set of all compositions,
  $\Om_{\langle d]}$ for the set of all compositions $\om$ for which 
  $|\om| \leq d$ and $|\om| \equiv d \mod 2$, and $\Om_{d}$ for
  the set of all compositions $\om$ with $|\om|=d$.

  \smallskip


%

Let us define two kinds of operations on $\Om$ and on $\Om^\infty$, the \emph{merge} and the \emph{insert} operations, which will be instrumental in what follows.

\smallskip
For a composition $\om = (\om_1,\ldots , \om_\ell) \in \Om$ and an integer $1 \leq j \leq \ell-1$,  
we define  $\sM_j(\om)$ as 
$$\sM_j(\om) = (\sM_j(\om)_1, \ldots, \sM_j(\om)_{\ell-1}),$$
with
\begin{eqnarray}\label{eq2.1a}
\sM_j(\omega)_i & = & \omega_i \; \textrm{ if }\, i < j,\\ \nonumber
\sM_j(\omega)_j  & = & \omega_j + \omega_{j+1},\\ \nonumber
\sM_j(\omega)_i  & = & \omega_{i + 1} \; \textrm{ if }\, j < i \leq  \ell-1.
\end{eqnarray}
For a marked composition $(\om,\kappa)$, we set 
$$\sM_j^\infty((\om,\kappa)) = (\sM_j(\om),\kappa) \text{ for } 1 \leq j \leq \ell-1,$$ 
$$\sM_0^\infty((\om,\kappa)) = ((\om_2,\ldots,\om_\ell),\kappa+\om_1),
\text{ and }
\sM_\ell^\infty((\om,\kappa)) = ((\om_1,\ldots,\om_{\ell-1}),\kappa+\om_\ell).$$

We call $\sM_j, \sM_j^\infty$ the \emph{merge operations} on compositions and
marked compositions. 

\smallskip
In parallel, we define the \emph{insert operations} $\sI_j$.  For 
a composition $\om = (\om_1,\ldots, \om_\ell)$ and a number
$0 \leq j \leq \ell$ we set
$\sI_j(\om) = (\sI_j(\om)_1,
\ldots, \sI_j(\om)_{\ell+1}),$ where
\begin{eqnarray}\label{eq2.2a}
\sI_j(\omega)_i & = & \omega_i\; \textrm{ if }\, i \leq j, \\ \nonumber
\sI_j(\omega)_{j+1} & = & 2,\\ \nonumber
\sI_j(\omega)_i & = & \omega_{i - 1} \; \textrm{ if }\, j+1 < i \leq \ell+1.
\end{eqnarray}
That is, $\sI_j(\om) = (\om_1,\ldots,\om_j,2,\om_{j+1},\ldots,\om_\ell)$.
We extend $\sI_j$ to marked compositions by defining $\sI_j^\infty((\om,\kappa)) = (\sI_j(\om),\kappa)$. 


\begin{example} For $(\om, \kappa) = ((12243), 2)$ we have
	$$\sM_0^\infty((\om,\kappa)) = ((2243), {\color{red}3}),\quad \sM_1^\infty((\om,\kappa))= (({\color{red}3}243), 2), \quad \sM_2^\infty((\om,\kappa))= ((1{\color{red}4}43), 2),$$
	$$\sM_3^\infty((\om,\kappa)) = ((12{\color{red}6}3), 2),\quad \sM_4^\infty((\om,\kappa))= ((122{\color{red}7}), 2), \quad \sM_5^\infty((\om,\kappa))= ((1224), {\color{red}5}),$$

	$$\sI_0^\infty((\om,\kappa)) =(({\color{red}2}12243), 2), \quad \sI_1^\infty((\om,\kappa)) =((1{\color{red}2}2243), 2),\quad \sI_2^\infty((\om,\kappa)) =((12{\color{red}2}243), 2),$$ 
	$$\sI_3^\infty((\om,\kappa)) =((122{\color{red}2}43), 2), \quad \sI_4^\infty((\om,\kappa)) =((1224{\color{red}2}3), 2),  \quad \sI_5^\infty((\om,\kappa)) =((12243{\color{red}2}), 2).$$

	Note that for example for $d =14$ we have 
	$|\sI_j^\infty((\om,\kappa))| = 16$ and hence
	$\mathring{\sR}_d^{\sI_j^\infty((\om,\kappa))} = \emptyset$ for all 
	$0 \leq j \leq \ell$, while for all $0 \leq j \leq \ell$ we have
	$\mathring{\sR}_d^{\sM_j^\infty((\om,\kappa))} \neq \emptyset$.
\end{example}
\smallskip 

The next proposition collects some basic properties of the cells
$\mathring{\sR}^\omega_d$ and of their closures $\sR^\omega_d$ in $\cP_d$, see 
\cite[Theorem 4.1]{Ka1} for details. In \ref{lm1}, we will write 
$\mathbb H$ to denote the upper half-plane of the complex numbers with positive imaginary part,
$\Sym^m(X)$ for the $m$-fold symmetric product of a space $X$, and
$\Pi_m$ for the open cone $\{ (x_1,\ldots, x_m) \in \R^m~|~x_1 < \cdots < x_m\}$.
For later use, we also introduce $\Conf^m(X)$ as a notation for the configuration space of $m$ 
distinct unordered points in a space $X$. Note that $\Pi_m \cong \Conf^m(\R)$.

\begin{PROP}\label{lm1}
  Take $d \geq 1$ and $\om = (\om_1,\ldots, \om_\ell) \in \Om$ such 
  that $|\om | \le d$ and $|\om | \equiv d \mod 2$.
  Then $\mathring{\sR}^\omega_d\subset \cP_d$ is homeomorphic to 
  $\Pi_\ell \times \Sym^{\frac{d-|\om|}{2}}\mathbb H$. In particular, it is 
  an open cell of codimension $|\om|'$. 
  Moreover, 
	$\sR^\omega_d$ is the union of  the cells 
	$\mathring{\sR}_d^{\omega'}$, 
  taken over all $\omega'$ that are obtained from 
  $\omega$ by a sequence of merge and insert operations.
  In particular,
  \begin{itemize}
	  \item[(a)] The cell $\mathring{\sR}^\om_d$ has the maximal dimension $d$ if and only
       if, for 
       $0 \leq \ell \leq d$ and $\ell \equiv d \mod 2$, $\om=(\underbrace{1,1,\dots, 1}_{\ell})$. 
       
     \item[(b)] The cell $\mathring{\sR}^\om_d$ has the minimal possible dimension $1$ if
	     and only if $\om=(d)$. In this case, $\mathring{\sR}^{(d)} = 
	     \sR^{(d)} = \{ (x-a)^d~|~a \in \R\}$. 
  \end{itemize}
\end{PROP}

From a geometric perspective, if a point in $\mathring{\sR}^\omega_d$ approaches the boundary $\sR_d^\om \setminus \mathring{\sR}^\omega_d$, 
then either there exists at least one value of the index $j$ such that the distance between the
$j$\textsuperscript{th} and $(j+1)$\textsuperscript{st} distinct real roots
goes to $0$, or there is a  value of $j$ such that two complex-conjugate
non-real roots converge to a real point which becomes the $(j+1)$\textsuperscript{st} smallest distinct real root.
The first situation corresponds to the application of the
merge operation $\sM_j$ to $\omega$, and the second one to the application
of the insert operation $\sI_j$ (further degenerations are obtained by iterating these operations).

For $d' \leq d$, under the embedding $\wp: \cP_{d'} \hookrightarrow \cB_d$, 
the cell $\mathring{\sR}_{d'}^\om$ is sent to 
$\mathring{\sR}_d^{(\om,\, d-d')}$. The image  $\wp(\cP_{d'})$ lies in the 
closed subspace $\cB_{d'} \cong \R\P^{d'}$ of $\cB_d$ of codimension $d-d'$. 
As described in \ref{lm1}, in addition to the cells 
$\mathring{\sR}_{d}^{(\om',\,d-d')}$, which are the $\wp$-images of the cells  
$\mathring{\sR}_{d'}^{\om'}$, the boundary of $\mathring{\sR}_{d'}^\om$   
additionally contains only the cells which lie in the $\wp$-image of 
$\cP_{d''}$ for some $d'' < d'$. The latter arise when the smallest or the 
largest real root, or a pair of complex-conjugate roots, of some $f\in \cP_{d'}$ approaches infinity. Since the 
latter situation is described by the merge operations $\sM_0^\infty$ and 
$\sM_\ell^\infty$ (possibly preceded by an insert operation) 
acting on marked compositions $((\om_1,\ldots,\om_\ell),\kappa)$, we obtain 
the following result.

\begin{COR}
  \label{cor:bd}
  Let $d \geq 1$,  $\om = (\om_1,\ldots, \om_\ell) \in \Om$, and let
  $(\om,\kappa) \in \Om^\infty$ be such 
  that $|(\om,\kappa) | \le d$ and $|(\om,\kappa) | \equiv d \mod 2$.
  Then $\mathring{\sR}^{(\om,\kappa)}_d\subset \cB_d$ is homeomorphic to 
  $\Pi_\ell \times \Sym^{\frac{d-|(\om,\kappa)|}{2}}\mathbb H$. In particular, 
  $\mathring{\sR}^{(\om,\kappa)}_d$ is an open cell of codimension $|(\om,\kappa)|'$. 
  
  Moreover, 
	$\sR^{(\omega,\kappa)}_d$ is the union of cells 
	$\mathring{\sR}_d^{(\omega',\kappa')}$, 
  taken over all $(\omega',\kappa')$ that are obtained from 
  $(\om,\kappa)$ by a sequence of merge and insert operations.
  In particular,
  \begin{itemize}
	  \item[(a)] the cell $\mathring{\sR}^{(\om,\kappa)}_d$ has the maximal dimension $d$ if and only
		  if 
		  $\om=(\underbrace{1,1,\dots, 1}_{\ell})$ 
       for 
       some $0 \leq \ell \leq d$ with $\ell \equiv d \mod 2$, and  $\kappa=0$.
       
     \item[(b)] the cell $\mathring{\sR}^{(\om,\kappa)}_d$ has the minimal possible  dimension $0$ 
		     if and only if $\om=()$ and $\kappa=d$. In this case, $\mathring{\sR}^{((),d)}_d = 
	     \cP_0 = \{ a~|~a \in \R \setminus\{0\}\}/\sim$. 
  \end{itemize}
\end{COR}
 
 \smallskip

Note that the norms $|\om|$ and $|(\om, \kappa)|$ are preserved under the merge operations for cells both in $\cP_d$ and in $\cB_d$. In both cases, the insert operations increase the norm by $2$ and thus preserve its parity.  \smallskip
  
The merge and insert operations can be used to define a natural 
\emph{partial order} ``$\prec$''  on the sets $\Om$ and $\Om^\infty$. 
It reflects the cellular structure, described in \ref{lm1} and \ref{cor:bd}.

\smallskip
\begin{definition}\label{def2.1} 
  For $\omega, \omega' \in \Om$, we set $\omega' \prec \omega$, if $\omega'$ can be 
  obtained from $\omega$ by a sequence of merge operations, $\sM_j$, $j \geq 1$, and  
  insert operations $\sI_j$, $j \geq 0$. 
  
  Analogously, for $(\om,\kappa), (\om',\kappa') \in \Om^\infty$, we set 
  $(\om',\kappa') \prec (\om,\kappa)$, if $(\om',\kappa')$ can be 
  obtained from $(\om,\kappa)$ by a sequence of merge operations, $\sM_j^\infty$, $j \geq 0$, and  insert operations $\sI_j^\infty$, $j \geq 0$. 

  If $\om' \prec \om$ or $(\om',\kappa') \prec (\om,\kappa)$, then we say that $\om'$ is smaller than $\om$ or that $(\om',\kappa')$ is smaller than $(\om,\kappa)$. \hfill $\diamondsuit$
\end{definition}

One can  easily check  that  ``$\prec$'' defines partial orders on $\Om$ and 
$\Om^\infty$. 
From now on, we will consider a subset $\Theta \subseteq \Om$ or 
$\Theta \subseteq \Om^\infty$ as a partially ordered set, \emph{poset} for 
short, ordered by ``$\prec$.''
\ref{lm1} and \ref{cor:bd} imply instantly the following two statements.

\begin{COR} \label{cor:cw-structure}
  For a subset $\Theta \subseteq \Om_{\langle d]}$,
  \begin{itemize}
    \item[(i)] $\cP^\Theta_d$ is closed in $\cP_d$ if and only if,
     for any $\omega \in \Theta$ and any $\omega' \in \Om_{\langle d]}$ with $\omega' \prec \omega$, we have
     $\omega' \in \Theta$,\smallskip
     
    \item[(ii)] if $\cP^\Theta_d$ is closed in $\cP_d$, then 
	    $\bar{\cP}^\Theta_d$ carries the structure of a compact CW-complex 
		  with open cells $\mathring{\sR}_d^\omega$, labeled by
		  $\omega \in \Theta$, and the unique $0$-cell, represented by the  point $\infty$. 
\end{itemize}
\end{COR}

\begin{COR} \label{cor:cw-structurebd}
  For a subset $\Theta \subseteq \Om^\infty_{\langle d]}$, 
  \begin{itemize}
    \item[(i)] $\cB^\Theta_d$ is closed in $\cB_d$ if and only if,
     for any $(\omega, \kappa) \in \Theta$ and any $(\omega',\kappa') \in \Om^\infty_{\langle d]}$ with $(\omega',\kappa') \prec (\omega, \kappa)$, we have
     $(\omega',\kappa') \in \Theta$, \smallskip
     
    \item[(ii)] if $\cB^\Theta_d$ is closed in $\cB_d$, then 
	    $\cB^\Theta_d$ carries the structure of a compact CW-complex 
		  with open cells $\mathring{\sR}_d^{(\om,\kappa)}$, labeled by
		  $(\om,\kappa) \in \Theta$.
  \end{itemize}
\end{COR}

\ref{cor:cw-structure} and \ref{cor:cw-structurebd} motivate the following definition.

\begin{definition}
 A subposet $\Theta \subseteq \Om$  is  called 
  \emph{closed} if, for any $\om' \prec \om$ and $\om \in \Theta$, we have $\om' \in \Theta$. 
  
 Similarly, a subposet $\Theta \subseteq \Om^\infty$ is  called 
  \emph{closed} if, for any $(\om', \kappa') \prec (\om, \kappa)$ and $(\om, \kappa) \in \Theta$, we have $(\om', \kappa') \in \Theta$. 

 For $\Theta \subseteq \Om_{\langle d]}$ (resp., $\Theta \subseteq \Om^\infty_{\langle d]}$), we call $\Theta$ \emph{closed} if the above condition holds for all $\om' \in \Om_{\langle d]}$ (resp., all $(\om',\kappa') \in \Om^\infty_{\langle d]}$); equivalently, $\Theta = \Theta'\cap\Om_{\langle d]}$ (resp., $\Theta = \Theta'\cap\Om^\infty_{\langle d]}$) for a closed $\Theta'\subseteq\Om$ (resp., $\Theta' \subseteq \Om^\infty$).
\end{definition}

Revisiting the beginning of \ref{sec:intro}, the closed posets $\Theta$ yield 
exactly the spaces $\bar{\cP}^\Theta_d$, $\cP^{\bfc \Theta}_d$ and 
$\cB^\Theta_d$, $\cB^{\bfc \Theta}_d$ we intend to study.
  
In  \ref{sec:combcomplex} below,  given  \emph{any} closed poset 
$\Theta \subset \Om_{\langle d]}$ or $\Theta \subseteq \Om_{\langle d]}^\infty$,
we introduce a {\it combinatorial model} for the cellular chain complex that 
calculates the reduced homology of $\bar{\cP}_d^{\Theta}$, respectively the homology of $\cB_d^\Theta$.
\smallskip

However, for a {\it general} closed poset $\Theta \subset \Om_{\langle d]}$, it is  
impossible to obtain a {\it closed formula} description of the homology 
$\tilde H_*(\bar{\cP}_d^{\Theta};\Z)$!   
  To justify this ``metaphysical'' claim, consider the closed subposet 
$\Om_d \subseteq \Om_{\langle d ]}$ of all 
$\om \in \Om_{\langle d ]}$ with $|\om| = d$.
The  space $\cP^{\Om_d}_{d}$ is the space of all real monic polynomials of 
degree $d$, having only real roots. 
We can identify $\cP^{\Om_d}_d$ with $\{ (x_1,\ldots, x_d)\in \R^d ~|~
x_1 \leq \cdots \leq x_d\}$, the closure of $\Pi_d$. 

As a result, the cellular structure of $\cP^{\Om_d}_d$ is the facial structure of the 
product of a $(d-1)$-dimensional cone $\mathcal C\Delta^{d-2}$ over a closed simplicial base $\Delta^{d-2}$ with a copy of the real line $\R$.   
The compositions in $\Om_d$ are in an order-preserving bijection with the 
open faces of the simplicial cone $\mathcal C\Delta^{d-2}$, so that $\mathring{\sR}^\om_d$ is the product of the face corresponding to $\om$ with $\R$.
Since the face poset of $\mathcal C\Delta^{d-2}$ is the Boolean 
lattice of subsets of $\{1,\ldots,d-1\}$, it follows that one can identify 
simplicial complexes $K$ with vertex set contained in $\{1,\ldots, d-1\}$ and 
closed subposets of $\Om_d$. 
Then, for a closed subposet $\Theta \subseteq \Om_d$, the one-point 
compactification 
$\bar{\cP}^{\Theta}_{d}$ is a double suspension of the corresponding simplicial 
complex $K$ whose vertices are among $\{1,\ldots, d-1\}$. This situation is 
considered in detail in an earlier paper of the second and third authors 
\cite{SW}. 

In particular, for arbitrary closed subposets 
$\Theta \subseteq \Om_d$, the spaces $\bar{\cP}_d^{\Theta}$ can, up to a shift 
by $2$ in the (co)homological dimension, 
carry any homology that a simplicial complex on $d-1$ vertices can carry.
Since $S^d \cong \bar{\cP}_d$, Alexander duality shows that the cohomology of 
$\cP_d^{\bfc \Theta}$ can have a similarly arbitrary complex structure! In 
particular, arbitrary torsion may occur.   
 
 \smallskip

However, let us stress that in this example, the dimension of the simplicial complex, 
corresponding to $\Theta$, and the degree $d$ of the polynomials under 
consideration are closely linked. When it is possible to loosen this link, 
then, as we shall see in \ref {sec:stab}, the (co)homological stabilization 
takes place, and quite ``tame'' answers for the limiting homology 
$\tilde H_\ast (\bar{\cP}_d^{\Theta},\Z)$ or 
$\tilde H_\ast ({\cB}_d^{\Theta},\Z)$ emerge as $d\to \infty$. 
\medskip

Before we further pursue our topological investigation, we would like to 
present a few combinatorial facts about the number of cells in $\bar{\cP}_d$ 
and $\cB_d$. Let $p_j^d$ be the number of cells in $\bar{\cP}_d$ of 
dimension $j$, and let $b_j^d$ be the number of cells in $\cB_d$ of dimension 
$j$.
By \ref{lm1} and \ref{cor:bd}, we know that, for $d \geq 1$, $p_0^d = 1$ and, for $j\ge 1$,  
$p_j^d$ counts the compositions $\om \in \Om_{\langle d]}$ with 
$d-|\om|' = j$.  Similarly, for $j \geq 0$, $b_j^d$ counts the marked 
compositions $(\om,\kappa) \in \Om_{\langle d]}^\infty$ with 
$d-|(\om,\kappa)|' = j$.

\smallskip
We introduce the generating $t$-polynomials 
$$G(\bar{\cP}_d;t)=\sum_{j=0}^d p_j^d\,  t^j \quad \text{and}\quad G(\cB_d;t)=\sum_{j=0}^d b_j^d\,  t^j.$$
Interpreting the value of the generating functions at $t=-1$ as the Euler characteristics of $\bar{\cP}_d \cong S^d$ and $\cB_d \cong \R\P^d$, we get that 
$G(\bar{\cP}_d;-1)=1+(-1)^d$ and $ G(\cB_d;-1)=\frac{1}{2} (1+(-1)^d)$.
\ref{lm1}, (a) implies that $p_d^d=\lfloor\frac{d}{2}\rfloor+1$ for $d \geq 1$.

Since $\cB_d = \cP_d \sqcup \cB_{d-1}$, we get for $d \geq 1$ 
the recurrence relation
$$G(\cB_d;t)=G(\cB_{d-1};t)+(G(\bar{\cP}_d;t)-1).$$ 
This recurrence  implies   the following relation for the $(t, x)$-generating functions
$$\sum_{d \geq 0} G(\cB_d;t) x^d = 1 + x \sum_{d \geq 0} G(\cB_{d};t) \,x^d + \sum_{d \geq 0} G(\bar{\cP}_d;t)x^d - 2 - \frac{x}{1-x}.$$
Note that the $2$ in the recursion comes from $\bar{\cP}_0$ consisting of
$2$ points. 
In particular, $$\sum_{d \geq 0} G(\cB_d;t) x^d = \frac{1}{1-x} \Big( -1+ \sum_{d \geq 0} G(\bar{\cP}_d;t)x^d -\frac{x}{1-x}\Big).$$ 
Now, for $j \geq 1$, $p_j^d$ equals  the number of compositions of a number $d' \leq d$ with even $d-d'$  and $j-(d-d')$ parts. Thus  for $j \geq 1$,  
$$p_j^d = \sum_{\genfrac{}{}{0pt}{}{d-j\, \leq\, d' \,\leq \,  d}{d-d' \text{ even}}} \binom{d'-1}{j-d+d'-1} = \sum_{k=0}^{\min\{\lfloor\frac{j-1}{2}\rfloor,\, \lfloor \frac{d-1}{2} \rfloor\}}
\binom{d-1-2k}{j-1-2k} \; + \; \epsilon_j^d,$$ 
where, in the first sum, the term $d'=0$ (the empty composition) is interpreted as $\binom{-1}{-1}=1$, and where 
$\epsilon_j^d = 1$ if $j=d$ is even and $\epsilon_j^d=0$ otherwise. \smallskip
   The latter observations imply via some standard calculations of 
   generating functions the following result. 

\smallskip
\begin{lemma}\label{lm:count} The number of cells in $\bar{\cP}_d$ and $\cB_d$ is given by the following generating functions:
  \begin{itemize}
    \item[(i)] 
	    $$ \sum_{d=0}^\infty G(\bar{\cP}_d;t)x^d =
	    {\frac {{t}^{3}{x}^{3}+{t}^{2}{x}^{3}-{x}^{2}{t}^{2}-tx+{x}^{2}-3\,x+
2}{ \left( 1-t^2 x^2 \right)    \left( 1-x-tx \right)
 \left( 1-x \right) }}
 $$

    \item[(ii)] 
	    $$\sum_{d=0}^\infty G(\cB_d;t) x^d = 
	    {\frac {1}{  \left( 1-t^2 x^2 \right) \left( 1-x-tx \right)   }}$$
       
    \item[(iii)] For any positive odd $d$, $G(\bar{\cP}_d;t)-1=t \, G(\cB_{d-1};t)$, while for any positive even $d$, $G(\bar{\cP}_d;t)-1=t\, G(\cB_{d-1};t)+t^d$.  \hfill $\diamondsuit$
  \end{itemize}
\end{lemma}

\noindent
\begin{example} $G(\cB_0;t)=1$, \; $G(\cB_1;t)=1+t$,  \;$G(\cB_2;t)=1+2t+2t^2$, \; $G(\cB_3;t)=1+3t+4t^2+2t^3$,  \; \;$G(\cB_4;t)=1+4t+7t^2+6t^3+3t^4$,  \hfill\break  $G(\cB_5;t)=1+5t + 11t^2 + 13t^3 + 9t^4 + 3t^5$. 

 \smallskip
\noindent
	$G(\bar{\cP}_1;t)=1+t$, \; $G(\bar{\cP}_2;t)=1+t+2t^2$, \; $G(\bar{\cP}_3;t)=1+t+2t^2+2t^3$,   \hfill\break
	$G(\bar{\cP}_4;t)=1+t+3t^2+4t^3+3t^4$, \; $G(\bar{\cP}_5;t)=1+t+4t^2+7t^3+6t^4+3t^5$,  \hfill\break
	$G(\bar{\cP}_6;t)=1+t+5t^2 + 11t^3 + 13t^4 + 9t^5 + 4t^6$. \hfill $\diamondsuit$
 \end{example}

  The coefficient sequences of $G(\bar{\cP}_d;t)$ and $G(\cB_d;t)$, 
  up to a shift in indexing and removals of trailing or leading $1$s, 
  have also appeared in other contexts (see A035317 in \cite{Sl}).

\section[Combinatorial differential complexes]{Combinatorial differential complexes for  
\texorpdfstring{$H_\ast(\cB_d^{\Theta}; \Z)$ and $\tilde H_\ast(\bar{\cP}_d^{\Theta}; \Z)$}{Lg}}
\label{sec:combcomplex}


In this section, we use the natural CW-structure on 
$\cB_d$ and $\bar{\cP}_d$, described in \ref{cor:bd} and \ref{lm1}, to 
construct {\it combinatorial differential complexes} that calculate 
the homology of $\cB_d^\Theta$ and $\bar{\cP}^\Theta_d$ 
for any given closed subposet $\Theta \subset \Om_{\langle d]}^\infty$ or
$\Theta \subset \Om_{\langle d]}$. 
\smallskip

Recall that the cells $\mathring{\sR}^{(\om,\kappa)}_d$ of the CW-complex  
$\cB_d$ are indexed by marked compositions $(\om,\kappa) 
\in \Om_{\langle d]}^\infty$ and that 
the dimension of the cell $\mathring{\sR}_d^{(\om,\kappa)}$ equals $d-|(\om,\kappa)|' = d - |\omega|'-\kappa$. 
\smallskip

\medskip
Next, we provide an explicit description of the cellular chain complex of the CW-complex $\cB_d^\Theta$.
From now on, for $\Theta \subseteq \Om_{\langle d]}$, 
or $\Theta \subseteq \Om_{\langle d]}^ \infty$, 
we write $\Z[\Theta]$ for the abelian group or equivalently $\Z$-module, freely generated 
by the elements of $\Theta$. For $\Theta \subseteq \Om^\infty$ we define $\Theta_{|\sim|' \leq k}$ (resp., 
$\Theta_{|\sim|' =  k}$) as the set of all (marked) compositions $(\om,\kappa) \in \Theta$ with $|(\om,\kappa)|' \leq k$ (resp., $|(\om,\kappa)|' = k$). 
We use the analogous conventions for $\Theta \subseteq \Om$.
 \smallskip


For a composition $\om = (\om_1,\ldots,\om_\ell)$,  we denote by $s_\om$ the 
number  of its parts, i.e. $s_\om=\ell$. Using the merge operators $\sM$ and the insert 
operators $\sI$, introduced in  \ref{sec:cellstructure}, we define two 
homomorphisms which act on $\Z[\Theta]$.  They are given by 
$$\d_{\sM}^\infty ((\omega,\kappa)) := - \sum_{k=0}^{s_\omega} (-1)^k \sM_k^\infty((\omega,\kappa)) \, \text{  and   }\, \d_{\sI}^\infty ((\omega,\kappa)) :=  \sum_{k=0}^{s_\omega} (-1)^k \sI_k^\infty((\omega,\kappa)).$$
Here the terms $\sI_k^\infty((\om,\kappa))$ with $|(\om,\kappa)| = d$, which do not lie in $\Om^\infty_{\langle d]}$, are omitted, and the merge sum is empty if $s_\om = 0$.

Next, we define a homomorphism 
$$\d^\infty  = \d_{\sM}^\infty +\d_{\sI}^\infty : \Z[\Theta] \to \Z[\Theta]$$ by the formula
\begin{equation} \label{eq13.7} \d^\infty((\omega,\kappa)):= 
\begin{cases}
-\sum_{k = 0}^{s_\omega}(-1)^k\, \sM_k^\infty((\omega,\kappa)) + \sum_{k = 0}^{s_\omega}(-1)^k\, \sI_k^\infty((\omega,\kappa)),\;\; \text{for}\; |(\omega,\kappa)| < d,   \\
-\sum_{k = 0}^{s_\omega}(-1)^k\, \sM_k^\infty((\omega,\kappa)),\;\; \text{for}\; |(\omega,\kappa)| = d. 
\end{cases}
\end{equation}
\smallskip


\medskip

\begin{lemma}\label{lem13.8} 
       For any closed $\Theta \subseteq \Om_{\langle d]}^ \infty$, 
       the homomorphisms 
       $\d_{\sM}^\infty, \d_{\sI}^\infty: \Z[\Theta] \to \Z[\Theta]$ are 
       anticommuting differentials, i.e.  $(\d_{\sI}^\infty)^2=(\d_{\sM}^\infty)^2 = \d_{\sI}^\infty \d_{\sM}^\infty+ \d_{\sM}^\infty \d_{\sI}^\infty =0$.
       
      Thus, $\d^\infty = \d_{\sM}^\infty + \d_{\sI}^\infty$ is a differential as well, and $(\Z[\Theta], \d^\infty)$ is a graded differential complex,
		whose $j$\textsuperscript{th} graded part is
		$\ZZ[\Theta_{|\sim|'=d-j}]$.   


\end{lemma}


\begin{proof} 
 Let us first show that $(\d_{\sI}^\infty)^2 = (\d_{\sM}^\infty)^2 = \d_{\sI}^\infty \d_{\sM}^\infty + \d_{\sM}^\infty \d_{\sI}^\infty= 0$. 
   Consider a marked composition $(\omega,\kappa)$ with $\om = (\om_1,\ldots, \om_\ell)$. Then for $\ell \geq 3$,  the expression for $\d_{\sM}^\infty((\omega,\kappa))$ will, in particular,  include the terms of the form
   $$(-1)^k ((\om_1,\ldots,\om_{k-1},\om_k+\om_{k+1},\om_{k+2},\ldots, \om_\ell),\kappa)\; +$$ 
   $$(-1)^{k+1} ((\om_1,\ldots, \om_k, \om_{k+1}+\om_{k+2},\om_{k+3},\ldots,\om_\ell),\kappa).$$  
   Thus $(\d_{\sM}^\infty)^2((\omega,\kappa))$ will contain the vanishing sum   
   $$(-1)^{2k} ((\om_1,\ldots, \om_{k-1}, \om_k+\om_{k+1}+\om_{k+2},\om_{k+3},\ldots,\om_\ell),\kappa) +$$ $$ (-1)^{2k+1}  ((\om_1,\ldots, \om_{k-1}, \om_k+\om_{k+1}+\om_{k+2},\om_{k+3},\ldots, \om_\ell),\kappa)=0.$$
   For $\ell\geq 2$, the homomorphism $\d_{\sM}^\infty((\omega,\kappa))$  will also contain  the terms \hfill\break $((\om_2,\ldots, \om_\ell),\om_1+\kappa) - ((\om_1+\om_2,\om_3,\ldots, \om_{\ell}),\kappa)$,  which yields 
   $$((\om_3,\ldots, \om_\ell),\om_1+\om_2+\kappa) - ((\om_3,\ldots, \om_{\ell}),\om_1+\om_2+\kappa) = 0$$ in $(\d_{\sM}^\infty)^2((\omega,\kappa))$.
For $\ell=1$,   $\d^\infty_{\sM} ((\om_1),\kappa) = 0$, since $\sM_0^\infty$ and $\sM_1^\infty$ both give $((),\om_1+\kappa)$ with opposite signs. \smallskip

   To show that $(\d_{\sI}^\infty)^2 = 0$, write $$\d_{\sI}^\infty(\omega) =\; \dots + (-1)^k ((\om_1,\ldots, \om_{k}, 2, \om_{k+1},\ldots,\om_\ell),\kappa) + \dots\, .$$ 
   Thus $(\d_{\sI}^\infty)^2((\omega,\kappa))$ consists of vanishing sums such as  
   $$(-1)^{2k}((\om_1,\ldots, \om_k, 2, 2, \om_{k+1},\ldots, \om_\ell),\kappa) + (-1)^{2k+1} ((\om_1,\ldots, \om_k, 2, 2, \om_{k+1},\ldots, \om_\ell),\kappa) = 0;$$
   the terms in which the two $2$'s are not adjacent cancel in pairs in the same way. 

   Finally, let us compute $(\d_{\sI}^\infty \d_{\sM}^\infty + \d_{\sM}^\infty \d_{\sI}^\infty)((\omega,\kappa))$. Observe that  $\d_{\sI}^\infty((\omega,\kappa))$ consists of the sums of the form  
   $$(-1)^k ((\om_1,\ldots, \om_k,2,\om_{k+1} ,\ldots, \om_{\ell}),\kappa)\; + \;$$ 
   $$(-1)^{k+1} ((\om_1,\ldots, \om_{k+1} , 2, \om_{k+2} , \ldots,\om_{\ell}),\kappa) +(-1)^{k+2} ((\om_1,\ldots,  \om_{k+2} ,2, \om_{k+3},\ldots, \om_{\ell}),\kappa).$$ 
   
   \smallskip
   Cancelling terms in $\d_{\sM}^\infty \d_{\sI}^\infty((\omega,\kappa))$, we are
   left with expressions $$(-1)^{2k+2}((\om_1,\ldots,\om_k, 2, \om_{k+1}+\om_{k+2}, \om_{k+3}, \ldots, \om_{\ell}),\kappa) \; +$$ 
   $$(-1)^{2k+3} ((\om_1,\ldots, \om_k, \om_{k+1}+\om_{k+2} , 2, \om_{k+3},\ldots, \om_\ell),\kappa).$$  
   In a similar computation of $\d_{\sI}^\infty\d_{\sM}^\infty((\omega,\kappa))$, we will obtain the contributions $$(-1)^{2k+1} ((\om_1,\ldots, \om_k, 2, \om_{k+1}+\om_{k+2},\om_{k+3},\ldots, \om_\ell),\kappa) \;+$$ 
   $$(-1)^{2k+2} ((\om_1,\ldots, \om_k ,\om_{k+1}+\om_{k+2},2, \om_{k+3},\ldots, \om_\ell),\kappa).$$ 
  Adding them together, and noting that all remaining terms (merges away from the inserted $2$, merges of the inserted $2$ with a neighbour, and the terms involving the root at infinity) cancel in pairs in the same way, we get  that $(\d_{\sI}^\infty \d_{\sM}^\infty + \d_{\sM}^\infty \d_{\sI}^\infty)((\omega,\kappa)) = 0$.
  
  \smallskip
   The claim that $\Z[\Theta]$ is a graded differential complex now follows from the fact that $\d^\infty$ is a differential and that 
   $\d_{\sM}^\infty$ and $\d_{\sI}^\infty$ respect the grading.
\end{proof}

\begin{proposition}\label{prop.3.2}
  For any closed subposet $\Theta\subseteq \Om_{\langle d]}^\infty$, the graded differential complex $(\Z[\Theta], \d^\infty)$ coincides with the cellular chain complex of
  $\cB_d^\Theta$. 
  In particular, 
   $(\Z[\Theta], \d^\infty)$ calculates the homology of 
 $\cB_d^{\Theta}$.
\end{proposition}

\begin{proof} 
  By \ref{cor:bd}, the \emph{topological} boundary $\d\sR^{(\omega,\kappa)}_d$ of $\sR^{(\omega,\kappa)}_d$ 
  coincides with 

\begin{eqnarray}\label{eq:boundary}
        \Big(\,\bigcup_{k = 0}^{s_\omega} \sR^{\sM_k^ \infty((\omega,\kappa))}_d\, \Big)\; \bigcup \; \Big(\,\bigcup_{k = 0}^{s_\omega} \sR^{\sI_k^ \infty((\omega,\kappa))}_d\,\Big).
  \end{eqnarray}

  Therefore, the algebraic boundary of $\sR^{(\omega,\kappa)}_d$ in the
  cellular chain complex of $\cB_d^{\Theta}$
  is given by a sum of the form 
  \begin{eqnarray}\label{eq13.7a}
    \sum_{k = 0}^{s_\omega}  a_k\, \sR^{\sM_k^\infty((\omega,\kappa))}_d 
     + \sum_{k = 0}^{s_\omega}  b_k\, \sR^{\sI_k^\infty((\omega,\kappa))}_d, 
  \end{eqnarray}
  where $a_k$ and $b_k$ are some integer coefficients. 
\smallskip

  In order to determine the coefficients $a_k$ and $b_k$, we proceed as follows.
  By \ref{eq:boundary}, if a path $\gamma= \{P_t\}_{t \in [0, 1]}$ is such that $\gamma \setminus \gamma(1) \subset \mathring{\sR}^{(\om,\kappa)}_d$ and $\gamma(1)$ belongs to the boundary of $\sR^{(\om,\kappa)}_d$, 
	then at least either two consecutive real roots of the  polynomials $P_t$ approach each other, or 
  two complex-conjugate roots approach each other at a point of $\R$, or the smallest or the largest real root tends to $\infty \in \R\P^1$, as $t \to 1$. These
  situations are encoded in the merge and insert operations.
  Thus in order to calculate $a_k$ and $b_k$, we need to understand what happens along paths in $\mathring{\sR}^{(\omega,\kappa)}_d$ that approach the boundary cells
  $\mathring{\sR}^{\sM_k^\infty((\omega,\kappa))}_d$ and  
  $\mathring{\sR}^{\sI_k^\infty((\omega,\kappa))}_d$, respectively. Note that we can face a situation when, for $k \neq k'$, 
  $\sM_k^\infty((\omega,\kappa)) = \sM_{k'}^\infty((\omega,\kappa))$ or 
  $\sI_k^ \infty((\omega,\kappa)) =\sI_{k'}^\infty((\omega,\kappa))$. 
  Nevertheless, we will show that each operation $\sM_k$ or $\sI_k$ 
  corresponds to a unique local homotopy type of paths $\gamma$ in the vicinity of $\gamma(1)$ and hence yields path-independent contributions $a_k$ and $b_k$, respectively. 
 

  The preferred orientation of $\sR^{(\omega,\kappa)}_d$ is 
  induced by the canonical orientation of the open cell 
  $\mathring{\sR}^{(\omega,\kappa)}_d$. Recall that we can identify 
  a class of binary forms from $\mathring{\sR}^{(\omega,\kappa)}_d$ with a univariate polynomial of degree $d-\kappa$. 
  By \ref{cor:bd}, we know that 
  $\mathring{\sR}^{(\omega,\kappa)}_d \cong \Pi_{s_\omega} \times 
  \Sym^{m_\omega}(\mathbb H)$ for $\om = (\om_1,\ldots, \om_{s_{\omega}})$ and 
  $m_\om :=  \frac{d-|(\om,\kappa)|}{2}$.

 %

  The orientation of the open cell 
  $\Sym^{m_\omega}(\mathbb H)$ is canonically induced by its complex structure, 
  while the orientation of $\Pi_{s_\omega}$ is induced from 
  $\R^{s_\omega}$.
  In other words, the orientation of 
  $\mathring{\sR}^{(\om,\kappa)}_d \cong \mathring{\sR}^\omega_{d- \kappa}$ is  given by the volume form 
  $$\rho_\omega :=(dx_1 \wedge  \dots \wedge dx_{s_\omega}) \wedge 
  \Big(\frac{i}{2}\Big)^{m_\omega}(dz_1 \wedge d\bar{z}_1\, \wedge \dots \,
  \wedge dz_{m_\omega} \wedge d\bar{z}_{m_\omega}),$$ considered on the 
  product $\Pi_{s_\omega} \times \Sym^{m_\omega}(\mathbb H)$. 

  \smallskip

  \noindent {\sf Claim 1:} In the formula \ref{eq13.7a} for the boundary operator 
   we have $a_k = (-1)^{k+1}$. 
\smallskip

  Note that, if for some $k$ and $l$, we have 
  $\sM_k^\infty((\omega,\kappa)) = \sM_l^\infty((\omega,\kappa))$, then 
  either $k = l$ or $\{k,l\} = \{0,s_\om\}$ and all the entries of $\om$ are
  \emph{identical}. In the latter case, $\sM_l^\infty((\omega,\kappa)) = \sM_0^\infty((\omega,\kappa))$ and the corresponding cell $\mathring{\sR}^{(\omega,\kappa)}_d$ is adjacent to the cell $\mathring{\sR}^{\sM_0^\infty((\omega,\kappa))}_d$ ``on both sides.''

  \smallskip

  \noindent {\sf Case:} $k \neq s_\omega, 0$ 

  In this situation $\sM_k^\infty((\om,\kappa)) \neq \sM_l^\infty((\om,\kappa))$ for $l \neq k$.
  Let $P \in \mathring{\sR}^{\sM_k^\infty((\omega,\kappa))}_d$ be a  polynomial of degree $d-\kappa$. Consider a one-parameter family $\{P_t \subset \mathring{\sR}^{(\omega,\kappa)}_d\}_{t \geq 0}$ 
  such that $\lim_{t \to 0} P_t = P$. Then, for all sufficiently small $t > 0$, the $k$\textsuperscript{th} and $(k+1)$\textsuperscript{st} smallest real roots
  $x_k(P_t)$ and $x_{k +1}(P_t)$ of $P_t$ converge to the $k$\textsuperscript{th} smallest real root $x_k(P)$ of $P$. Moreover, for any two paths $P_t, Q_t \subset 
  \mathring{\sR}^{(\omega,\kappa)}_d$ such that  $\lim_{t \to 0} P_t  = \lim_{t \to 0} Q_t = P$, their germs at $P$ can be deformed into one another 
  by a \emph{small} homotopy in $\mathring{\sR}^{(\om,\kappa)}$. Let us explain this claim.

We say that $\e > 0$ is \emph {small enough for $P$} if the $\e$-disks around distinct roots of $P$ are disjoint in $\C$ and the $\e$-disks around the non-real roots do not intersect the real line.  Let $U_\e(P)$ denote the union of such $\e$-disks. 
  
 For any  $\e >0$ which is small enough for $P$,  there exists  $t_\e>0$ such that for $0< t < t_\e$,  each root of 
 $P_t$ and of $Q_t$ resides in $U_\e(P)$. This implies that, with the exception of the two merging roots $x_k(P_t)$ and  
 $x_{k +1}(P_t)$, which share the same $\e$-ball centered at the $k$\textsuperscript{th} root of $P$, any other real 
 root of $P_t$ resides in a single $\e$-ball contained in $U_\e$. The same conclusion is valid for $Q_t$, $t < t_\e$. The  non-real roots of $P_t$ and $Q_t$ require more careful treatment, since more than two non-real roots of $P_t$ or $Q_t$ may belong to the same $\e$-disk $D_\e(z)$ centered on a non-real root $z$ of multiplicity $m(z) \geq 2$ of the limiting polynomial $P$. As a result, there is no natural correspondence between the roots of $P_t$ and $Q_t$ that reside in $D_\e(z)$.
 
Now,  we can deform $P_t$ into $Q_t$ in $\mathring{\sR}^{(\omega,\kappa)}_d$ so that, in the process of deformation, the roots of the intermediate polynomials do not exit $U_\e(P)$. This may be accomplished by the linear homotopy of the corresponding real roots of $P_t$ and $Q_t$, being combined with the following deformation of the non-real roots of $P_t$ and $Q_t$ that reside in each disk $D_\e(z)$. We first contract radially all the roots of $P_t$ in $D_\e(z)$ to $z$, counted with the multiplicity $m(z)$, and then expand $m(z) z$ radially into the divisor of $Q_t$ that is supported in $D_\e(z)$. Thus the desired deformation of $P_t$ into $Q_t$ takes place not only in the open cell $\mathring{\sR}^{(\omega,\kappa)}_d$, but also in the vicinity of $P$. 
  Therefore, the germ of $\mathring{\sR}^{(\omega,\kappa)}_d$ at $P \in \mathring{\sR}^{\sM_k^\infty((\omega,\kappa))}_d$ has a single connected component contained in $\mathring{\sR}^{(\omega,\kappa)}_d$.  As a result, the incidence index $a_k$ of the cell $\mathring{\sR}^{(\omega,\kappa)}_d$ with 
  the cell $\mathring{\sR}^{\sM_k^\infty(\omega,\kappa)}_d$ equals $\pm 1$.
\medskip

In order to determine the sign, it suffices to consider a family of polynomials $P_t \subset \mathring{\sR}^\omega_d$
  such that:
  \begin{itemize}
    \item $\lim_{t \to 0} P(t) = P \in \mathring{\sR}^{\sM_k^\infty((\omega,\kappa))}_d$,
    \item the polynomials $P(t)$ share the roots with $P$, except for the roots $x_k(P_t), x_{k+1}(P_t)$ and $x_k(P)$,
    \item $x_k(P_t) = x_k(P)  - t$, $x_{k+1}(P_t) = x_k(P) + t$.
  \end{itemize}

  Since the complex roots of $P_t$ and $P$ coincide, we can restrict our attention to
  $\Pi_{\omega}$.
  The tangent vector to the path $P(t)$ at $t = 0$ is given by
  $w = (0, \dots, 0, -1, 1, 0, \dots)$. This vector is the inward normal to the face
  $\{x_k = x_{k+1}\}$ of the polyhedron $\Pi_{\omega} \subset \R^{s_\omega}$.

  The boundary orientation is obtained by contracting the volume form with the outward vector $-w$.
  The interior product ``$\rfloor$''
  of $-w$ with the volume form in $\R^{s_\omega}$ is given by
  \begin{multline*}
    (-w)\, \rfloor (dx_1 \wedge  \dots \wedge dx_{s_\omega})
   = (-1)^{k + 1}\big(dx_1 \wedge  \dots \wedge dx_{k-1} \wedge dx_{k+1}\wedge \dots \wedge dx_{s_\omega} \\
   + dx_1 \wedge  \dots \wedge dx_k \wedge dx_{k+2}\wedge \dots \wedge dx_{s_\omega}\big),
  \end{multline*}
  whose restriction to the face $\{x_k = x_{k+1}\}$ is $2(-1)^{k+1}$ times its preferred volume form.

 Therefore, the orientation of the boundary $\d\Pi_{\omega}$, induced by the orientation of its interior $\Pi_{\omega}^\circ$, differs from the preferred orientation of
  $\Pi_{\sM_k^\infty(\omega)}^\circ$ exactly by the factor $(-1)^{k +1}$.
\smallskip

  \noindent {\sf Case:} $k = s_\omega := |\om|- |\om|'$ 

 If a real polynomial $P$ of degree $d -\kappa$ has its largest real root $x_k(P)$ of multiplicity $\om_k$, then  $P$ produces, with the help 
 of the map $\wp$ from \ref{eq.polys-forms}, a binary form $\wp(P)$ of degree $d$ in the homogeneous variables $[z:w]$. In fact, $\wp(P) = (z- x_\ell \cdot w)^{\om_k}\,w^\kappa \times Q(z, w)$, 
 where $Q(z, w)$ is a real binary form of degree $d-\om_k -\kappa$ such that $Q(z, 1) = P(z)\cdot (z- x_k)^{-\om_k}$.

\smallskip
 Let us examine the result of the merge operation $\sM_\ell^\infty((\om,\kappa))$ on the real zero divisors of the binary form $\wp(P)$.
\smallskip
                                                                                                                                                                                                            
 First  we choose an open interval $(\a, \b) \subset \R$ that contains all the real roots $x_1(P), \ldots,$ $ x_k(P)$ of $P$. Then we choose an affine chart  
 $U_\a := \R\P^1 \setminus \{[\a : 1]\}$  and a projective transformation $\mathcal A \in \mathsf{PGL}(2, \R)$ which maps  the ordered sequence $[x_1:1], \ldots , [x_k: 1]$ in the chart 
 $U_\infty:= \{[z: 1]\}$ to an ordered sequence $\mathcal A([x_1:1]), \ldots , \mathcal A([x_k: 1])$ in the chart $U_\a$. In particular,  $\mathcal A$ maps  $\infty = [1:0]$ to a point $[a : 1]$, where $a > \mathcal A([\b : 1])$. 
 In the original affine chart $U_\infty$, the projective transformation $\mathcal A$ is given by the rational function $F(z) = a - \frac{b}{z- \a}$, where $F(\b)  < a$. Now, as $x_k \to \infty$, $F(x_k) \to a$.                                                                                                                                                                  
 The real zero locus $\om_1\cdot x_1 +\ldots + \om_k \cdot x_k + \kappa \cdot\infty$ of the binary form $\wp(P)$ is mapped by $\mathcal A$ to the real zero locus  
 $\om_1 \cdot \mathcal A(x_1) +\ldots + \om_k \cdot \mathcal A(x_k) + \kappa \cdot a$ of the binary form  $\mathcal A^\ast (\wp(P))$. Under the transformation $\mathcal A$, 
 the combinatorial pattern $(\om, \kappa)$ of $\wp(P)$ is transformed into the combinatorial pattern $((\om_1, \ldots, \om_k, \kappa), 0)$ of 
 $\mathcal A^\ast (\wp(P))$ in $U_\a$. Now set $\om^\kappa := (\om_1, \ldots, \om_k, \kappa)$. Notice that under $\mathcal A^\ast$, the merge operation 
 $\sM_k^\infty((\omega,\kappa))$ is transformed into the operation $\sM_k(\omega^\kappa) = \sM_k^\infty(\omega^\kappa, 0)$. The latter case does not involve $\infty$ and  has been analyzed above. 
\smallskip                                                                                                                                                                                                                 


  \smallskip

  \noindent {\sf Case:} $k= 0$ 
  
  This case can be treated analogously to the case $k = s_\omega$.

  \medskip

  \noindent {\sf Claim 2:} In the formula \ref{eq13.7a} for the boundary operator 
  we have $b_k = (-1)^{k}$. 

  \smallskip

  This is a more delicate situation since multiple sheets 
  of $\mathring{\sR}^{(\omega,\kappa)}_d$ may appear in the neighborhood of $P \in \mathring{\sR}^{\sI_k(\omega)}_d$.  
  The simplest example of such phenomenon occurs when $P$ has two consecutive real roots, $x_i$ and $x_{i+1}$, each of  multiplicity $2$. 
  Denote by $P_1$ and $P_2$ two polynomials obtained by resolving the first and, respectively, the second double root into a pair of simple complex-conjugate roots. 
  Then the combinatorial patterns of the real root multiplicities of $P_1$ and  $P_2$ are identical. 
 Note that, although $P_1$ and $P_2$ are close to each other in $\sR^{(\omega,\kappa)}_d$, there is no short path in 
  $\mathring{\sR}^{(\omega,\kappa)}_d$  connecting them in a neighborhood of $P$.

  Let $x_1(P) < \cdots < x_{s_\omega +1}(P)$ be the distinct real roots of 
  $P  \in \mathring{\sR}^{\mathsf I_k^\infty((\omega,\kappa))}_d$, and let 
  $\{(z_l^\star, \bar{z}_l^\star)\}_l$ be the unordered collection of non-real 
  roots of $P$. By our assumptions on $P$, we have that $x_{k+1}(P)$ is a root of multiplicity $2$. 
  Let $P_t \subset  \mathring{\sR}^{(\omega,\kappa)}_d$ be a path in 
  $\mathring{\sR}^{(\omega,\kappa)}_d$ such that $\lim_{t \rightarrow 0} P_t = P$. 
  We may assume that 
\begin{itemize}
  \item $x_j(P_t) = x_j(P)$ for all $j \leq k$ and 
	$x_j(P_t) = x_{j + 1}(P)$ for all $j \geq k+1$;
  \item 
        $(x_{k+1}(P) + \mathsf i\, t,  
	x_{k+1}(P) - \mathsf i\,t)$ is a pair of simple complex-conjugate roots for 
        $P_t$ and all the other non-real roots of $P_t$ and $P$ coincide. 
\end{itemize}

With this convention and choosing a fixed ordering of the complex roots with positive imaginary parts, we may consider $P_t$ as a  curve in the ordered ``root space'' $\R^{s_\omega} \times \C^{m_\omega}$. Here
the vector $w = (0, \dots , 0, \mathsf i, 0, \dots,  0)$ is   tangent to the curve 
$P_t$ at $P$. 

\smallskip
The volume form $\rho_\omega$ can be written as 
$\rho_\omega^{\R} \wedge  \rho_\omega^{\C}$, where 
$$\rho_\omega^{\R} := dx_1 \wedge \dots \wedge dx_{s_\omega}$$ and $$\rho_\omega^{\C} := \left(\frac{\mathsf i}{2}\right)^{m_\omega}(dz_1 \wedge d\bar{z}_1) \wedge \dots \wedge (dz_{m_\omega} \wedge d\bar{z}_{m_\omega}).$$ 

Since $\mathsf i\, \rfloor \big(\frac{\mathsf i}{2} dz \wedge d\bar{z}\big) = (0,1)\, \rfloor (dx \wedge dy) = -dx$, we get \hfill\break
\smallskip
\[
w\,  \rfloor (\rho_\omega^{\R} \wedge \rho_\omega^{\C}) = (-1)^{s_\omega}\, \rho_\omega^{\R} \wedge (w\, \rfloor  \rho_\omega^{\C}) = -(-1)^{s_\omega}\, \rho_\omega^{\R} \wedge dx_\ast \wedge \rho_{\sI_k(\omega)}^{\C}
\]
\[ = (-1)^{2s_\omega + k+1}\, \rho_{\sI_k(\omega)}^{\R} \wedge \rho_{\sI_k(\omega)}^{\C} = (-1)^{k+1} \,\rho_{\sI_k(\omega)},
\]
where $x_\ast$ denotes the coordinate of the new double root. Since $w$ points into the cell, the boundary orientation is obtained by contracting with the outward vector $-w$, which gives $(-1)^k\rho_{\sI_k(\omega)}$.


This calculation implies that the part $\mathring{\sR}^{\sI_k^\infty((\omega,\kappa))}_d$ of the boundary
of $\sR^{(\omega,\kappa)}_d$, being approached via the path $P_{t, k} := P_t \subset
\mathring{\sR}^{(\omega,\kappa)}_d$ as above, acquires an orientation that differs from its 
$\rho_{\sI_k(\omega)}$-induced orientation by the factor $b_k = (-1)^k$. 


\medskip

Thus we have shown that the cellular chain complex of the CW-complex $\cB_d^\Theta$ coincides with the
graded differential complex $\d^\infty:  \Z[\Theta] \to \Z[\Theta]$,
defined by the formula \ref{eq13.7}.
\end{proof}


As an immediate consequence of \ref{prop.3.2} and of the previous discussions, we obtain one of our main results --- the combinatorial differential complex that calculates the homology of $\cB_d^{\Theta}$.

\begin{theorem}\label{th.3.3} 
  Let $\Theta \subset \Om_{\langle d]}^\infty$ be a closed subposet. 
  For all $j \geq 0$, the homology of $\cB_d^{\Theta}$ is given by
 \begin{eqnarray}
     H_{j}(\cB_d^{\Theta}; \Z) \cong \; 
     \frac{\ker\big\{\d^\infty:  
     \Z[\Theta_{|\sim|'=d-j}] \rightarrow 
     \Z[\Theta_{|\sim|'=d - j + 1}]\big\}}
     {\im\big\{\d^\infty:  
     \Z[\Theta_{|\sim|'=d - j - 1}] \rightarrow  
	\Z[\Theta_{|\sim|'=d - j}]\big\}},
 \end{eqnarray}
 where the differential $\d^\infty$ is given by formula \ref{eq13.7}. \hfill $\diamondsuit$
\end{theorem}

For later use, let us also consider  a relative version of this differential complex. Namely, given  a closed subposet $\Theta \subseteq \Om_{\langle d]}^\infty$,  
consider the exact sequence of differential complexes
\begin{equation}\label{eq.short_exact}
0 \to (\Z[\Theta], \d^\infty) \to (\Z[\Om^\infty_{\langle d]}], \d^\infty) \to (\Z[\Om^\infty_{\langle d]} \setminus \Theta], \d^\#) \to 0,
\end{equation}
where the first homomorphism is the obvious inclusion and the last term in the exact sequence is, by definition, the  differential quotient complex. Its differential $\d^\#$ is still given by formulas 
\ref{eq13.7} in which all terms $\sM_k^\infty((\om,\kappa))$ (resp. $\sI_k^\infty((\om,\kappa))$) with $\sM_k^\infty((\om,\kappa)) \in \Theta$ (resp. $\sI_k^\infty((\om,\kappa)) \in \Theta$) are replaced by $0$.
\smallskip

Since all $\Z$-modules in \ref{eq.short_exact} are free, we get a short exact sequence of dual complexes:
\begin{equation}\label{eq.dual_short_exact}
0 \to (\Z[\Om^\infty_{\langle d]} \setminus \Theta]^\ast, (\d^\#)^\ast) \to (\Z[\Om^\infty_{\langle d]}]^\ast, (\d^\infty)^\ast) \to (\Z[\Theta]^\ast, (\d^\infty)^\ast) \to 0,
\end{equation}
where $\Z[\sim]^\ast$ denotes $\mathsf{Hom}_\Z(\Z[\sim], \Z)$, and $(\d^\#)^\ast, (\d^\infty)^\ast$ are the dual differentials. 

\smallskip
For any closed subposet $\Theta \subseteq \Om_{\langle d]}^\infty$, the following claim is an immediate consequence of \ref{th.3.3} and of the fact that $\cB_d^\Theta$ is a subcomplex of the CW-complex $\cB_d$.

\begin{corollary}\label{lem.homology_of_quotient} 
  Let $\Theta \subset \Om_{\langle d]}^\infty$ be a closed subposet. 
  Then for all $j \geq 0$,  the reduced 
  homology of the quotient $\cB_d / \cB_d^{\Theta}$ is given by
 \begin{eqnarray}
     \tilde H_{j}(\cB_d / \cB_d^{\Theta}; \Z) \cong \; 
     \frac{\ker\big\{\d^\#:  
     \Z[(\Om^\infty_{\langle d]} \setminus \Theta)_{|\sim|'=d-j}] \longrightarrow 
     \Z[(\Om^\infty_{\langle d]} \setminus \Theta)_{|\sim|'=d - j + 1}]\big\}}
     {\im\big\{\d^\#:  
     \Z[(\Om^\infty_{\langle d]} \setminus \Theta)_{|\sim|'=d - j - 1}] \longrightarrow  
	\Z[(\Om^\infty_{\langle d]} \setminus \Theta)_{|\sim|'=d - j}]\big\}}. 
 \end{eqnarray}

 The reduced cohomology of the quotient space $\cB_d\big/ \cB_{d}^{\Theta}$ 
 is isomorphic to the homology of the dual 
 differential complex $(\Z[\Om^\infty_{\langle d]} \setminus \Theta]^\ast, (\d^\#)^\ast)$ introduced in \ref{eq.dual_short_exact}. \hfill $\diamondsuit$
\end{corollary}

%
 
Now let us consider the one-point compactification $\bar{\cP}_d^\Theta$ for a closed subposet $\Theta \subseteq \Om_{\langle d]}$. 
Analogously to the case of a closed subposet $\Theta \subseteq \Om_{\langle d]}^\infty$, we introduce two homomorphisms on $\Z[\Theta]$ by the formula

$$\d_{\sM} (\omega) := - \sum_{k=1}^{s_\omega-1} (-1)^k \sM_k(\omega) \, \text{  and   }\, \d_{\sI} (\omega) :=  \sum_{k=0}^{s_\omega} (-1)^k \sI_k(\omega).$$


Next, we define a homomorphism $$\d  = \d_{\sM} +\d_{\sI} : \Z[\Theta] \to \Z[\Theta]$$ by 
\begin{eqnarray}\label{eq13.7uni}
\quad \quad\d(\omega) := \left\{
\begin{array}{ll}
-\sum_{k = 1}^{s_\omega-1}(-1)^k\, \sM_k(\omega) + \sum_{k = 0}^{s_\omega}(-1)^k\, \sI_k(\omega),\;\; \text{for}\; |\omega| < d,   \\
 \\
-\sum_{k = 1}^{s_\omega-1}(-1)^k\, \sM_k(\omega),\;\; \text{for}\; |\omega| = d.
\end{array}\right.
\end{eqnarray}
(In particular, $\d((d)) = 0$. The $0$-cell $\infty$ is not a basis element: we use the reduced cellular complex.)

\begin{corollary}\label{cor13.4} 
  Let $\Theta \subset \Om_{\langle d]}$ be a closed subposet. 
  Then for all $j \geq 0$, the reduced homology of the one-point compactification $\bar{\cP}_d^{\Theta}$ of $\cP_d^{\Theta}$ is given by
 \begin{eqnarray}
     \tilde{H}_{j}(\bar{\cP}_d^{\Theta}; \Z) \cong \; 
     \frac{\ker\big\{\d:  
     \Z[\Theta_{|\sim|'=d-j}] \rightarrow 
     \Z[\Theta_{|\sim|'=d - j + 1}]\big\}}
     {\im\big\{\d:  
     \Z[\Theta_{|\sim|'=d - j - 1}] \rightarrow  
	\Z[\Theta_{|\sim|'=d - j}]\big\}},
 \end{eqnarray}
 where the differential $\d$ is given by formula \ref{eq13.7uni}.
\end{corollary}
\begin{proof}
  Consider the decomposition $\cB_{d} = \cP_d \sqcup \cdots \sqcup \cP_0 =\cP_d   \sqcup \cB_{d-1}$. 
  For a closed $\Theta \subseteq \Om_{\langle d]}$, consider the smallest closed subposet $\Theta^\infty 
  \subseteq \Om_{\langle d]}^\infty$ that contains $\Theta \times \{0\}$. 
  Let $\Theta_\ell$ be the subset of all $\om$ such that $(\om,d-\ell) \in \Theta^\infty$. 
   (In particular, $\Theta= \Theta_d$.) 
  Then the above decomposition  induces  decomposition $\cB_d^{\Theta^\infty}
  =  \cP_d^{\Theta_d}  \sqcup \cdots \sqcup \cP_0^{\Theta_0} $. 
 
  Since $\Theta_{d-1} \times \{1\} \cup \cdots \cup \Theta_0 \times \{d\} $ is
  closed, we get that 
  $Y =\cP_{d-1}^{\Theta_{d-1}} \sqcup \cdots \sqcup  \cP_0^{\Theta_0} $ is a closed CW-subcomplex of 
  $\cB_d$. Thus $\bar{\cP}_d^\Theta$ and $\cB_d^{\Theta^\infty}/Y$ are isomorphic as CW-complexes.
  In particular, the cellular differentials of $\bar{\cP}_d^\Theta$ and $\cB_d^{\Theta^\infty}/Y$ coincide. 
  Now the assertion follows from \ref{prop.3.2}, applied to the CW-pair $(\cB_d^{\Theta^\infty}, Y)$ as in \ref{lem.homology_of_quotient}, and our definition of the differential $\d^\#$. 
\end{proof}
  
 \ref{cor13.4} allows us to compute the (co)homology of $\cP_d^{\bfc \Theta} := \cP_d\setminus \cP_d^{\Theta} = S^d \setminus \bar{\cP}_d^{\Theta}$ using the Alexander duality.  

\begin{corollary}\label{cor13.4b} 
  Let $\Theta \subset \Om_{\langle d]}$ be a closed subposet. 
  Then for all $j \geq 0$, the reduced cohomology of $\cP_d^{\bfc \Theta}$ is given by 
  \begin{equation}\label{eq13.8} 
    \tilde H^j(\cP_d^{\bfc\Theta}; \Z) \cong H_{d- j - 1}(\d:  \Z[\Theta] \to \Z[\Theta]) 
      := \frac{\ker\{\d:  \Z[\Theta_{|\sim|^{'}=j+1}] \to \Z[\Theta_{|\sim|^{'}=j+ 2}]\}}
      {\im\{\d:  \Z[\Theta_{|\sim|^{'}=j}] \to  \Z[\Theta_{|\sim|^{'}=j+1}]\}}. \hfill \diamondsuit
  \end{equation}
\end{corollary}

With the combinatorial complex that calculates the cohomology of $\cP_d^{\bfc \Theta}$ in place, it is natural to look  for a similar complex that would 
calculate the cohomology of the complement $\cB_d^{\bfc \Theta} := \cB_d \setminus \cB_d^\Theta$ for any closed $\Theta \subseteq \Om_{\langle d]}^\infty$.
Since ${\cB}_d\cong \R\P^d$,  we need to consider the Alexander duality for the real projective spaces. 
With this goal in mind, let us remind the reader of a few standard constructions and notions regarding the Poincar\'{e} duality on 
non-simply-connected manifolds (see \cite{W}). 

\smallskip
Let $X$ be a $d$-dimensional compact connected manifold or, more generally, a Poincar\'{e} CW-complex.  
Let $\Lambda$ be the group ring $\Z[\pi_1(X)]$ of the fundamental group $\pi_1(X)$. An element of $\Lambda$ is a finite combination $\sum_g n_g g$, where $g \in \pi_1(X)$ and $n_g \in \Z$.

We denote by $\tilde X$ the universal cover of $X$ and by $C_\ast(\tilde X,\Z)$ the cellular  
chain complex of $\tilde X$, viewed as a right $\Lambda$-module under the free $\pi_1(X)$-action on $\tilde X$.  We define the homology and cohomology of $X$ with coefficients in a right $\Lambda$-module $\A$ (a local system of coefficients) by
$$H^\ast(X; \mathbb A) :=\; H(\mathsf{Hom}_\Lambda(C_\ast(\tilde X,\Z),\, \A)), \quad
H^t_\ast(X; \A) :=\; H(C_\ast(\tilde X,\Z)\otimes_\Lambda \A^t).
$$
The operation $\otimes_\Lambda$ relies on converting $\A$ into a left $\Lambda$-module $\A^t$ via the formula $\lambda a := a \bar\lambda$, where $a \in \A$ and $\lambda \in \Lambda$.  Here, for any $\lambda = \sum n_g \, g$, we denote by
$\bar\lambda := \sum w(g) n_g\, g^{-1}$, where the homomorphism $w: \pi_1(X) \to \Z_2$ is defined by the first Stiefel-Whitney class, an element of $H^1(X; \Z_2)$. The special case $\A = \Z$ is central for us.\smallskip

Similar definitions of $H^\ast(X, \d X; \A)$ and $H^t_\ast(X, \d X; \A)$ are available for a manifold $X$ with boundary $\d X$. They also make sense for any pair $X \supset K$, where $K$ is a compact subcomplex of a Poincar\'{e} complex $X$ (see \cite{W}). 
In such a case, the cap product with the fundamental cycle $[X] \in H_d(X; \Z^t)$ delivers the Poincar\'{e} duality isomorphism \[{}[X] \cap: H^j(X, K; \A ) \stackrel{\cong}{\rightarrow}  H_{d-j}(X \setminus K; \A^t),\]
where $\A$ is a right $\Z[\pi_1(X \setminus K)]$-module. 
As before,  in order to  convert $\A$ into a left $\Z[\pi]$-module $\A^t$,  we use the Stiefel-Whitney map $$w: \pi_1(X \setminus K) \to \pi_1(X)  \to \Z_2,$$  that describes whether the local orientation of $X$ is preserved or reversed along a loop in $X \setminus K$. 

Now consider the natural homomorphism $H^j(K; \A) \stackrel{\delta^\ast}{\rightarrow} H^{j+1}(X, K; \A)$ from the long exact sequence of the pair $(X, K)$ and compose it with the Poincar\'{e} duality $\D_{X \setminus K} = [X] \cap \sim$. This composition produces the Alexander homomorphism
\begin{eqnarray}
\mathcal A\ell^j:\; H^j(K; \A) \stackrel{\delta^\ast}{\rightarrow} H^{j+1}(X, K; \A) \stackrel{\cong\,\D_{X \setminus K}}{\longrightarrow} H_{d-j -1}(X \setminus K; \A^t).
\end{eqnarray}
Obviously, if $\delta^\ast$ is an isomorphism (monomorphism/epimorphism), then so is $\mathcal A\ell^j$. Similarly, we get a homomorphism
%
\begin{eqnarray}\label{eq.Al_D_general}
\mathcal A\ell_j:\; H^{d-j -1}(X \setminus K; \A) \stackrel{\cong\,\D_{X \setminus K}^{\ast}}{\longrightarrow} H_{j+1}(X, K; \A^t) \stackrel{\d_\ast}{\rightarrow} H_j(K; \A^t).
\end{eqnarray} 
Again if $\d_\ast$ is an isomorphism (monomorphism/epimorphism), then so is $\mathcal A\ell_j$.
\smallskip

We will apply $\mathcal A\ell$ in the case when $X= \R\P^{d}$, $K =  \cB_{d}^\Theta$,   $\A = \Z$, and $\A^t = \Z^t$ is the local coefficient system, defined by the homomorphism
$w: \pi_1(\R\P^{d} \setminus \cB_{d}^\Theta) \to \pi_1(\R\P^{d}) \cong \Z_2$. 

 \begin{theorem}\label{th.main_quadratic_formA}   
  Given  a closed  subposet $\Theta \subset \Om^\infty_{\langle d]}$, set $\bfc\Theta:= \Om_{\langle d]}^\infty \setminus \Theta$. 
   Then for any $j \geq 0$, we get isomorphisms: 
\begin{eqnarray}\label{eq.ADB} 
\qquad \;\;    H^j(\cB_d^{\bfc\Theta}; \Z^t) \cong \;
       \frac{\ker\big\{\d^\#:\,  \Z[\bfc\Theta_{|\sim|^{'}=j}] \longrightarrow \Z[\bfc\Theta_{|\sim|^{'}=j+1}]\big\}}
      {\im\big\{\d^\#:\,  \Z[\bfc\Theta_{|\sim|'=j-1}] \longrightarrow  \Z[\bfc\Theta_{|\sim|'=j}]\big\}},
\end{eqnarray}
\begin{eqnarray}\label{eq.ADB_dual} 
\qquad \;\;    H_j(\cB_d^{\bfc\Theta}; \Z^t) \cong \;
       \frac{\ker\big\{(\d^\#)^\ast:\,  \Z[\bfc\Theta_{|\sim|'=j}]^\ast \longrightarrow \Z[\bfc\Theta_{|\sim|'=j-1}]^\ast\big\}}
      {\im\big\{(\d^\#)^\ast:\,  \Z[\bfc\Theta_{|\sim|'=j+1}]^\ast \longrightarrow  \Z[\bfc\Theta_{|\sim|'=j}]^\ast\big\}}.
\end{eqnarray}
\end{theorem}

\begin{proof} The above isomorphisms  follow from \ref{th.3.3} and \ref{lem.homology_of_quotient}, being combined with the Poincar\'{e} dualities:
 $$\mathcal{PD}: H^j(\cB_d^{\bfc\Theta}; \Z^t) \cong H_{d-j}(\cB_d, \cB_{d}^{\Theta}; \Z) \cong \tilde H_{d-j}(\cB_d\big/ \cB_{d}^{\Theta}; \Z),$$
$$\mathcal {PD}^{-1}: H_j(\cB_d^{\bfc\Theta}; \Z^t) \cong H^{d-j}(\cB_d, \cB_{d}^{\Theta}; \Z) \cong \tilde H^{d-j}(\cB_d\big/ \cB_{d}^{\Theta}; \Z),
$$
and the fact that the $(d-j)$-dimensional cells of $\cB_d$ are indexed by the marked compositions of reduced norm $j$. For odd $d$, the local system $\Z^t$ is trivial.
\end{proof}

\noindent
{\bf Classical example: the discriminant of real binary forms.} 
To the best of our knowledge, only very few of the spaces $\cB_d^{\Theta}$ or 
$\cB_d^{\bfc\Theta}$ have been considered 
in the literature before. The most significant result, albeit in a different setting, is again due to Vassiliev.
He considered Arnold's situation from \ref{th:Arnold4} for \emph{real binary forms with $k$-moderate singularities} \cite{Va1, Va2}. 
Following Vassiliev's notation, let $\mathcal{HP}_d \cong \R^{d+1}$ denote the space of real binary homogeneous 
forms $Q(x,y)$ of degree $d$. Thus $\cB_d$ is the projective space of $\mathcal{HP}_d$. 
For $k \geq 2$ let $\Sigma_k\subset \mathcal{HP}_d$ be the subset consisting of all 
forms $Q$ vanishing with multiplicity at least $k$ on some line in $\R^2$. The main theorem of \cite{Va2}, formulated below, describes  the reduced cohomology  $\bar H^*(\mathcal{HP}_d\setminus \Sigma_k; \Z)$ for any $2\le k \le d$.

\begin{THEO}\label{th:Va2} Fix a number $k \in [2 , d]$.
	\begin{itemize}
		\item[(i)] For $k$ even, the reduced cohomology group $\bar H^*(\mathcal{HP}_d\setminus \Sigma_k; \Z)$ is a free abelian group of rank $2[d/k]+1$ with generators in degrees $$k-2, k-1,\; 2(k-2), 2(k-2)+1,\; \dots \; , \; [d/k](k-2), [d/k](k-2)+1, \text{ and } d-2[d/k].$$ 
               \item[(ii)] For $k$ odd and $d$ not divisible by $k$, $\bar H^*(\mathcal{HP}_d\setminus \Sigma_k; \Z)$  is a direct product of the following groups:
		       \begin{itemize}
                         \item[(a)] for any $p=1,2,\dots, [d/k]$ such that $d-pk$ is odd,  $\Z$ in dimension $p(k-2)$ and $\Z$ in dimension $p(k-2)+1$;
                         \item[(b)] for any $p=1,2,\dots, [d/k]$ such that $d-pk$ is even,  $\Z/2\Z$  in dimension \hfill \break $p(k-2)+1$;
                         \item[(c)] $\Z$  in dimension $d-2[d/k]$. 
		       \end{itemize}
              \item[(iii)] For $k$ odd and $d$ divisible by $k$, the answer is almost the same as in case (ii), but the summand $\Z/2\Z$ in dimension $d-2(d/k)+1$ disappears.  \hfill $\diamondsuit$
	\end{itemize}
\end{THEO}


 As an example of  application of \ref{th.main_quadratic_formA} and \ref{th.3.3}, let us now compute the reduced homology of the {\sf discriminant variety} and the homotopy type of its complement in our projective setting.  \ref{cor:homdiscr} should be compared with \ref{th:Va2} for $k=2$, in which case, the latter simplifies greatly.
 \smallskip 

Let $\D_d \subset \cB_d$ be the discriminant variety that consists of all classes of binary degree $d$ forms with at 
least one real root of multiplicity $\geq 2$.  Note that $\D_d = \cB_d^{\Theta_\disc}$ for 
$$\Theta_\disc = \Big\{ ((\om_1,\ldots, \om_\ell),\kappa) \in \Om_{\langle d]}^\infty ~\Big| \; \big(\exists\;\; 1 \leq i \leq \ell :\; \om_i \geq 2\big) \text{ or } \kappa \geq 2\Big\}.$$

The following result is contained in  Example 1 and Lemma 1 of \cite {Va2}. 

\begin{proposition}\label{lem.complement to discriminant} 
  The connected components of $\cB_d^{\bfc\Theta_\disc} = \cB_d \setminus \D_d$ are labeled by marked compositions 
	$((\underbrace{1,\ldots, 1}_\ell),0)$, where $\ell \geq 0$ and $d-\ell \geq 0$ is even. If $d$ is odd, then each connected component of  $\cB_d^{\bfc\Theta_\disc} = \cB_d \setminus \D_d$ is homotopy equivalent to a circle;  
  if $d$ is even, then each connected component of  $\cB_d^{\bfc\Theta_\disc} = \cB_d \setminus \D_d$ is homotopy equivalent to a circle, except for the component labeled by $((),0)$, which is contractible.
\end{proposition}
\begin{proof} 
  Let $\ell$ be such that $d-\ell$ is even. 
  Then, for $\ell \neq 0$ and 
	$\hat\om_1 := ((\underbrace{1,\ldots, 1}_\ell),0)$, $\hat\om_2 := \sM_\ell^\infty(\hat\om_1) = \sM_0^\infty(\hat\om_1) = ((\underbrace{{1,\ldots, 1}}_{\ell-1}),1)$,  the union
  $\mathcal C_{\ell,d} = \mathring{\sR}_d^{\hat\om_1} \cup \mathring{\sR}_d^{\hat\om_2}$ is a connected component of 
  $\cB_d \setminus \D_d$. If $d$ is even and $\ell = 0$, then
  $\mathcal C_{0,d} = \mathring{\sR}_d^{((),0)}$ is a connected component. 
  Since every cell of $\cB_d$ that is not contained in $\D_d$ is contained in some 
  $\mathcal C_{\ell,d}$, these are all the connected components.  

  The space $\mathcal C_{0,d} = \mathring{\sR}_d^{((),0)}$ is 
  an open cell and hence contractible.  
  Further, for $\ell \neq 0$ and even $d-\ell$, the polynomials in the 
  corresponding connected components have 
  $\ell$ simple real roots (including the root at infinity)  
  and $\frac{d-\ell}{2}$ complex-conjugate pairs of roots in $\C\P^1\setminus \R\P^1$, counted with their multiplicities. 
  Then $\mathcal C_{\ell,d}$ is the 
  Cartesian product of the configuration space $\textup{Conf}^\ell(\R\P^1)$ and $\Sym^{\frac{d-\ell}{2}}(\mathbb H)$.
 Since the upper half-plane $\mathbb H$ is contractible, $\Sym^{\frac{d-\ell}{2}}(\mathbb H)$ is contractible as well.
  It is well-known that $\textup{Conf}^\ell(\R\P^1) = \textup{Conf}^\ell(S^1)$ is homotopy equivalent to a circle \cite{Mort}.  
 As in \cite{Mort}, there is a fibration  $\textup{Conf}^\ell(S^1) \to S^1$ whose fiber is the open simplicial cone $\Pi_{\ell-1}$.
  Since the latter is contractible, we conclude that $\textup{Conf}^\ell(S^1)$ is homotopy  equivalent to a circle, which completes the proof. 
\end{proof}

The following lemma is the crucial step for the determination of the homology of $\D_d$. 

\begin{lemma} \label{lem:crucial} ~
   \begin{itemize}
	   \item[(i)] $H_i(\cB_d, \mathcal D_d; \Z) = 0$ for $i \leq d-2$. 
	   \item[(ii)] If $d$ is odd, then
		   $H_d(\cB_d, \mathcal D_d; \Z)\cong \Z^{(d+1)/2}$ and 
 $H_{d-1}(\cB_d, \mathcal D_d; \Z)\cong \Z^{(d+1)/2}$.
	   \item[(iii)] If $d$ is even, then $H_d(\cB_d, \mathcal D_d; \Z)\cong \Z$ and $H_{d-1}(\cB_d, \mathcal D_d; \Z)\cong (\Z/2\Z)^{d/2}$.
   \end{itemize}
\end{lemma}

We provide \emph{two} proofs of the lemma in order to demonstrate  the applicability of both \ref{th.main_quadratic_formA} and of \ref{th.3.3}.

\begin{proof}[Proof of \ref{lem:crucial}, based on \ref{th.main_quadratic_formA}]
  By the Poincar\'{e} duality and \ref{th.main_quadratic_formA}, 
   $$H_j(\cB_d, \mathcal D_d; \Z)\cong H^{d-j}(\cB_d\setminus \mathcal D_d; \Z^t),$$
  where, for $d$ odd, the local system of coefficients $\Z^t$ is constant, and, for $d$ even, the local coefficient system $\Z^t$ is twisted by the monodromy along the generator of $\pi_1(\cB_d) \cong \Z/2\Z$. 

  For an odd $d$, by \ref{lem.complement to discriminant}, the space $\cB_d\setminus \mathcal D_d$ is homotopy equivalent 
  to a disjoint union of $(d+1)/2$ circles, which implies that $H^0(\cB_d\setminus \mathcal D_d; \Z)\cong \Z^{(d+1)/2}$, 
  and $H^1(\cB_d\setminus \mathcal D_d; \Z)\cong \Z^{(d+1)/2}$, and all other cohomology groups vanish. As a result, 
  for $d$ odd, we get $H_d(\cB_d, \mathcal D_d; \Z)\cong \Z^{(d+1)/2}$, and 
  $H_{d-1}(\cB_d, \mathcal D_d; \Z)\cong \Z^{(d+1)/2}$, and all other homology groups vanish. 

  For an even $d$, by \ref{lem.complement to discriminant}, the space $\cB_d\setminus \mathcal D_d$ is homotopy equivalent to a disjoint union of a point and  $d/2$ circles. The generator of the fundamental group of each circle component maps to the generator of $\pi_1(\cB_d)$: rotating $\ell$ equally spaced real roots by $\pi/\ell$, while keeping the factor $(x^2+y^2)^{(d-\ell)/2}$, takes a form $f$ to $-f$. Hence $\Z^t$ restricts to the non-trivial local system on each circle, and we get $H^0(\cB_d\setminus \mathcal D_d; \Z^t) \cong \Z$ and  
  $H^1(\cB_d\setminus \mathcal D_d; \Z^t)\cong (\Z/2\Z)^{d/2}$. 
  All other cohomology groups $H^\ast(\cB_d\setminus \mathcal D_d; \Z^t)$ vanish. Thus, for $d$ even, 
  again by the Poincar\'{e} duality and \ref{th.main_quadratic_formA}, we get $H_d(\cB_d, \mathcal D_d; \Z)\cong \Z$, 
  $H_{d-1}(\cB_d, \mathcal D_d; \Z)\cong (\Z/2\Z)^{d/2}$, and all other homology groups vanish.
\end{proof}

\begin{proof}[Proof of \ref{lem:crucial}, based on \ref{th.3.3}]
	Since in our cellulation the cells in $\cB_d \setminus \D_d$ are of dimension $d$ and $d-1$, the cellular
	chain complex of the pair $(\cB_d,\D_d)$, arising from \ref{lem.homology_of_quotient}, has trivial chain groups in dimensions
	$\neq d,d-1$. This instantly implies (i).\smallskip

	From the preceding arguments it also follows that the  group of cycles in dimension $d-1$
	of the cellular chain complex of $(\cB_d,\D_d)$ is freely generated by the cells of dimension 
	$d-1$ in $\cB_d \setminus \D_d$. 

	For $d$ odd, the $d$-cells in $\cB_d \setminus \D_d$ are labeled by 
	$((\underbrace{1,\ldots,1}_{2j+1}),0)$, where $j \in [0,\frac{d-1}{2}]$, and the $(d-1)$-dimensional cells by  $((\underbrace{1,\ldots,1}_{2j}),1)$, where
	$j\in [0,\frac{d-1}{2}]$.
	By \ref{th.3.3}, the differential of the cells of dimension $d$ is fully contained in $\D_d$. Hence
	$H_d(\cB_d,\D_d;\Z)$ is freely generated by the cells of dimension $d$ and
	$H_{d-1}(\cB_d,\D_d;\Z)$ is freely generated by the cells of dimension $d-1$.
	A simple counting argument then yields (ii).

	For $d$ even, the $d$-cells in $\cB_d \setminus \D_d$ are labeled by $((\underbrace{1,\ldots,1}_{2j}),0)$, where $j \in [0,
	\frac{d}{2}]$, and the $(d-1)$-dimensional cells by  $((\underbrace{1,\ldots,1}_{2j-1}),1)$, where $j \in [1,\frac{d}{2}]$.
	By \ref{th.3.3}, the differential of the cell, labeled by
	$((\underbrace{1,\ldots,1}_{2j}),0)$, consists only of cells from $\D_d$ if $j = 0$, and 
	twice the cell, labeled by $((\underbrace{1,\ldots,1}_{2j-1}),1)$, plus some terms which correspond to cells from $\D_d$ when $j\in [1,\frac{d}{2}]$. 
	By simple counting and linear algebra, we then have $H_d(\cB_d,\D_d;\Z) = \Z$ and
	$H_{d-1}(\cB_d, \mathcal D_d; \Z)\cong (\Z/2\Z)^{d/2}$. This validates (iii).
\end{proof}
%
 %
%
%
Finally, we are in position to compute $\bar H_\ast(\mathcal D_d;\Z)$.

\begin{proposition}\label{cor:homdiscr} 
  Let $d \geq 2$. The non-zero reduced {\it integral} homology groups $\bar H_i(\mathcal D_d;\Z)$ of the discriminant $\D_d\subset \cB_d$ 
  have the following description: 

		\begin{itemize}
			\item[(i)] 
				if $d$ is odd, then $$\bar H_{d-1}(\D_d;\Z)\cong \Z^{(d-1)/2},\;  \bar H_{d-2}(\mathcal D_d;\Z)\cong \Z^{(d+1)/2}, \text{ and }$$ 
$$\bar H_{d-4}(\D_d;\Z) \cong \bar H_{d-6}(\D_d;\Z) \cong \cdots \cong \bar H_1(\D_d;\ZZ) \cong \Z/2\Z.$$ 
The rest of the homology vanishes.

\item[(ii)] if $d$ is even, then $$\bar H_{d-1}(\D_d;\Z)\cong \Z,\;  \bar H_{d-2}(\D_d;\Z)\cong (\Z/2\Z)^{(d/2)-1}, $$
 $$\text{ and }\; \bar H_{d-3}(\D_d;\Z)\cong \bar H_{d-5}(\D_d;\Z)\cong \cdots \cong \bar H_1(\D_d;\Z)\cong \Z/2\Z.$$
 The rest of the homology vanishes.
 \end{itemize}
\end{proposition} 

\begin{proof} 
	Since $\D_d$ has no cells in dimension $d$ by \ref{th.3.3} we have ${H}_d(\D_d;\Z) = 0$. 

  Consider the long exact homology sequence of the pair $\D_d\subset \cB_d$ with coefficients in $\Z$:

	\begin{figure}
	\begin{center}
		\begin{tikzcd}
			  0\cong H_d(\D_d;\Z) \arrow[r] 
			& H_d(\cB_d;\Z) \arrow[d,phantom,""{coordinate,name=Z}] \arrow[r] 
			& H_d(\cB_d, \D_d;\Z) \arrow[dll,rounded corners,to path={ --([xshift=2ex]\tikztostart.east)|- (Z)[near end]\tikztonodes-| ([xshift=-2ex]\tikztotarget.west)-- (\tikztotarget)}] \\	
		        H_{d-1}(\D_d;\Z) \arrow[r] 
			& H_{d-1}(\cB_d;\Z) \arrow[d,phantom,""{coordinate,name=Y}] \arrow[r] 
			& H_{d-1}(\cB_d, \D_d ;\Z) 
\arrow[dll,rounded corners,to path={ --([xshift=2ex]\tikztostart.east)|- (Y)[near end]\tikztonodes-| ([xshift=-2ex]\tikztotarget.west)-- (\tikztotarget)}] \\	
\cdots & \cdots 
			\arrow[d,phantom,""{coordinate,name=X}]
			& \cdots 
\arrow[dll,rounded corners,to path={ --([xshift=2ex]\tikztostart.east)|- (X)[near end]\tikztonodes-| ([xshift=-2ex]\tikztotarget.west)-- (\tikztotarget)}] \\	
			H_0(\D_d;\Z) \arrow[r] & H_0(\cB_d;\Z) \arrow[r] & H_0(\cB_d, \D_d;\Z) \cong 0.
		\end{tikzcd}
	\end{center}
		\caption{\small{Long exact homology sequence of the pair $(\cB_d,\D_d)$.}}
		\label{fig:long}
	\end{figure}

Basic algebraic topology \cite{Ha} tells us  that the homology of $\R\P^d \cong \cB_d$ is given by 
\begin{equation}
H_p(\R\P^d; \Z)=\begin{cases} \Z\; \text{ for } p=0 \text{ and,  when } d \equiv 1 \mod 2, \text{ for } p=d \\
                                     \Z/2\Z,\;  p \text{ odd, } 0<p<d\\
                                     0,\;\; \text{otherwise}.
 \end{cases}\label{eq:homodd}
 \end{equation}
 

\smallskip

	By \ref{lem:crucial}, (i) and the long exact sequence of the pair $(\cB_d,\D_d)$ from \ref{fig:long} we get that
	$\bar{H}_i(\D_d;\Z) \cong \bar{H}_i(\R\P^d;\Z)$ for $i < d-2$. 

	It remains to determine $\bar{H}_{i}(\D_d;\Z)$ for $i=d-1,d-2$.
	For $d$ odd, by \ref{lem:crucial} and the homology long exact sequence from \ref{fig:long}, 
	we obtain two short exact sequences. 

	\begin{align}
		\label{eq:seq1} 0 & \to & \Z           & \to \; \Z^{(d+1)/2}              & \to \; & H_{d-1}(\mathcal D_d; \Z) & \to & \; 0 & \text{\; and \;} \\
		\label{eq:seq2} 0 & \to & \Z^{(d+1)/2} & \to H_{d-2}(\mathcal D_d; \Z) & \to \; & \;\Z/2\Z & \to & \; 0.     & 
	\end{align}

	In \ref{eq:seq1}, $H_{d-1}(\mathcal{D}_d;\Z)$ is the homology of $\mathcal{D}_d$ in homological dimension 
	$\dim(\mathcal{D}_d)$ and hence torsion free. From the exactness of the sequence it then follows
	that $H_{d-1}(\mathcal{D}_d;\Z) \cong \Z^{(d+1)/2-1} = \Z^{(d-1)/2}$.

 Let $\gamma_\ell$ denote the boundary of the cell labeled $((\underbrace{1,\ldots,1}_{\ell}),1)$,
 $0 \leq \ell \leq d-1$ even. It lies completely in $\D_d$, so each $\gamma_\ell$ is a cycle in the
 cellular chain complex of $\D_d$. By the proof of \ref{lem:crucial}, the classes of these $\frac{d+1}{2}$
 cells freely generate $H_{d-1}(\cB_d,\D_d;\Z)$. Since $H_{d-1}(\R\P^d;\Z) = 0$, the connecting
 homomorphism in \ref{eq:seq2} is injective. Hence the classes $[\gamma_\ell]$ form a basis of its image
 $\ZZ^{\frac{d+1}{2}} \subseteq H_{d-2}(\D_d;\Z)$, which is the kernel of the homomorphism
 $H_{d-2}(\D_d;\Z) \to H_{d-2}(\R\P^d;\Z) \cong \Z/2\Z$.

 A simple calculation shows that in the sum
 $\gamma_0 + \gamma_2 + \cdots + \gamma_{d-1}$ all non-zero coefficients
 are $\pm 2$. Hence $\gamma_0 + \gamma_2 + \cdots + \gamma_{d-1} = 2\gamma'$ for a
 cycle $\gamma'$ in the cellular chain complex of $\D_d$. The chain
 $c = \sum_\ell ((\underbrace{1,\ldots,1}_{\ell}),1)$ is a cycle modulo $2$ in $\cB_d$, and its image in
 $H_{d-1}(\cB_d,\D_d;\Z/2\Z) \cong (\Z/2\Z)^{\frac{d+1}{2}}$ is non-zero; hence $[c] \neq 0$ in
 $H_{d-1}(\R\P^d;\Z/2\Z)$. Since $H_{d-1}(\R\P^d;\Z) = 0$ for odd $d$, the Bockstein homomorphism
 $H_{d-1}(\R\P^d;\Z/2\Z) \to H_{d-2}(\R\P^d;\Z)$ is injective, and it maps $[c]$ to the class of
 $\frac{1}{2}\d c = \gamma'$. Thus $[\gamma']$ maps to the generator of $H_{d-2}(\R\P^d;\Z) \cong \Z/2\Z$.
 If $H_{d-2}(\D_d;\Z)$ had a non-zero torsion element $\tau$, then
 $H_{d-2}(\D_d;\Z) = \ZZ^{\frac{d+1}{2}} \oplus \langle \tau \rangle$, and writing
 $[\gamma'] = i + \epsilon\, \tau$ with $i$ in the image of the connecting homomorphism, we would get
 $2i = \sum_\ell [\gamma_\ell]$, which is impossible since $\sum_\ell [\gamma_\ell]$ is not divisible by $2$
 in $\ZZ^{\frac{d+1}{2}}$. Hence $H_{d-2}(\D_d;\Z)$ is torsion free, and \ref{eq:seq2} gives
 $H_{d-2}( \D_d;\Z) \cong\Z^{\frac{d+1}{2}}$.

 
	For $d$ even, by \ref{lem:crucial} and the homology long exact sequence in \ref{fig:long}, 
	we obtain the exact sequence. 
	\begin{align}
		\label{eq:seq3}
0 \to \Z \to H_{d-1}(\mathcal D_d; \Z) \to \Z/2\Z \to (\Z/2\Z)^{d/2} \to H_{d-2}(\mathcal D_d; \Z) \to 0.
	\end{align}

      The copy of $\Z$ in \ref{eq:seq3} is the top homology of the pair $(\cB_d,\mathcal{D}_d)$. 
      By the proof of \ref{lem:crucial}, in our cellulation, the cell labeled $((),0)$ represents a class that
      generates this homology group.
      The image of this class, under the connecting homomorphism, is the class of its boundary, which consists  only of 
      the cell labeled by $((2),0)$. This is a top-dimensional cell of $\mathcal{D}_d$.
	Its boundary in the cellular chain complex of $\mathcal{D}_d$ is $0$. Since
	the cell, labeled by $((2),0)$, is a basis element of the top-dimensional cellular chain group of 
	$\mathcal{D}_d$, it follows that its class generates a $\Z$-summand of the homology group. Since this
	$\ZZ$-summand is the image of $H_d(\cB_d,\mathcal{D}_d;\Z)$ under the connecting homomorphism, and since
	$H_{d-1}(\mathcal{D}_d;\Z)$ is torsion free (it is the top-dimensional homology of $\mathcal{D}_d$), it follows
	from the exactness of \ref{eq:seq3} that $H_{d-1}(\mathcal{D}_d;\Z) \cong \Z$ and 
	\ref{eq:seq3} reduces to
	\begin{align*}
                   0 \to \Z/2\Z \to (\Z/2\Z)^{d/2} \to H_{d-2}(\mathcal D_d; \Z) \to 0.
	\end{align*}
	This exact sequence can be seen as an exact sequence of vector spaces. Hence the sequence splits and
        $H_{d-2}(\mathcal D_d; \Z) =  (\Z/2\Z)^{(d/2)-1}$.
%
%
\end{proof}

At the moment, 
we are unable to extend \ref{cor:homdiscr}  to the case of 
the discriminant of classes of binary forms in $\cB_d$ with at least one 
root of multiplicity  $k \geq 3$. One should mention that the latter case  
does not immediately follow from Vassiliev's \ref{th:Va2}, where the case 
of actual (non-projectivized) binary forms has been settled; this is related 
to a non-trivial action of $\Z/2\Z$ on the homology, when taking the 
projectivization. To illustrate this phenomenon, consider the table in 
\ref{fig:triple}.  It shows the results of our computer calculations for 
the projectivized discriminant in case 
$k = 3$, using the cellular chain complex from \ref{th.3.3}. Note the 
appearance of the higher torsion group $\ZZ/4\ZZ$ in the table, while in 
Vassiliev's case all the torsion groups are sums of copies of  $\ZZ/2\ZZ$.  

\begin{figure}
	{\tiny
$
\begin{array}{c|c|c|c|c|c|c|c|c|c|c|c|c}
	d\; \backslash \; i & 1 & 2 & 3 & 4 & 5 & 6 & 7 & 8 & 9 & 10 & 11 & 12 \\
\hline
 &          &          &           &            &          &              &     & & & & & \\  
	3  & \ZZ      &          &           &            &          &              &     & & & & & \\  
	   &          &          &           &            &          &              &     & & & & & \\  
	   \hline
	    &          &          &           &            &          &              &     & & & & & \\  
	4  & \ZZ^2    & \ZZ      &           &            &          &              &     & & & & & \\  
	   &          &          &           &            &          &              &     & & & & & \\  
	   \hline
	    &          &          &           &            &          &              &     & & & & & \\  
	5  & \begin{array}{c} \ZZ \\ \oplus \\ \ZZ/2\ZZ \end{array}      & \ZZ/2\ZZ          &  0         &            &          &              &     & & & & & \\  
		&          &          &           &            &          &              &     & & & & & \\  
		\hline
		 &          &          &           &            &          &              &     & & & & & \\  
	6  & \ZZ/2\ZZ & 0        & \ZZ^2     & \ZZ        &          &              &     & & & & & \\  
	   &          &          &           &            &          &              &     & & & & & \\  
	   \hline
	    &          &          &           &            &          &              &     & & & & & \\  
	7  & \ZZ/2\ZZ      & 0 & \begin{array}{c} \ZZ^2 \\ \oplus \\ \ZZ/2\ZZ \end{array}        &  \begin{array}{c} \ZZ \\ \oplus \\ \ZZ/2\ZZ \end{array}        & 0           &          &              &     & & & & \\  
		&          &          &           &            &          &              &     & & & & & \\  
		\hline
		 &          &          &           &            &          &              &     & & & & & \\  
	8  & \ZZ/2\ZZ & 0        & \begin{array}{c} \ZZ \\ \oplus \\ \ZZ/2\ZZ \end{array}  & \ZZ/4\ZZ        & \ZZ         & \ZZ              &     & & & & & \\  
		&          &          &           &            &          &              &     & & & & & \\  
		\hline
		 &          &          &           &            &          &              &     & & & & & \\  
	9  & \ZZ/2\ZZ      & 0 & \ZZ/2\ZZ & 0 & \begin{array}{c} \ZZ^2 \\ \oplus \\ \ZZ/2\ZZ \end{array}         &  \begin{array}{c} \ZZ \\ \oplus \\ \ZZ/2\ZZ \end{array}         & 0           &          &              &     & & \\  
		&          &          &           &            &          &              &     & & & & & \\  
		\hline
		 &          &          &           &            &          &              &     & & & & & \\  
	10  & \ZZ/2\ZZ & 0        & \ZZ/2\ZZ  & 0         & \begin{array}{c} \ZZ^2 \\ \oplus \\ \ZZ/2\ZZ \end{array}       & \begin{array}{c} \ZZ \\ \oplus \\ \ZZ/4\ZZ \end{array}  & \ZZ    & \ZZ & & & & \\  
		&          &          &           &            &          &              &     & & & & & \\  
		\hline
		 &          &          &           &            &          &              &     & & & & & \\  
	11 & \ZZ/2\ZZ      & 0 & \ZZ/2\ZZ & 0 & \begin{array}{c} \ZZ \\ \oplus \\ \ZZ/2\ZZ \end{array}        &  \ZZ/4\ZZ         & \begin{array}{c} \ZZ\\ \oplus \\ \ZZ/2\ZZ \end{array} & \begin{array}{c} \ZZ \\ \oplus \\ \ZZ/2\ZZ \end{array} &0 & & &\\  
		&          &          &           &            &          &              &     & & & & & \\  
		\hline
		 &          &          &           &            &          &              &     & & & & & \\  
	12  & \ZZ/2\ZZ & 0        & \ZZ/2\ZZ  & 0         & \ZZ/2\ZZ & 0 & \begin{array}{c} \ZZ^2 \\ \oplus \\ \ZZ/2\ZZ\end{array}        & \begin{array}{c} \ZZ \\ \oplus \\ \ZZ/4\ZZ \end{array}  & \ZZ    & \ZZ & & \\  
		&          &          &           &            &          &              &     & & & & & \\  
		\hline
		 &          &          &           &            &          &              &     & & & & & \\  
	13  & \ZZ/2\ZZ & 0        & \ZZ/2\ZZ  & 0         & \ZZ/2\ZZ & 0 & \begin{array}{c} \ZZ^2 \\ \oplus \\ \ZZ/2\ZZ\end{array}        & \begin{array}{c} \ZZ \\ \oplus \\ \ZZ/4\ZZ \end{array}  & \begin{array}{c} \ZZ \\ \oplus \\ \ZZ/2 \ZZ \end{array}  & \begin{array}{c} \ZZ \\ \oplus \\ \ZZ/2\ZZ \end{array} & 0 & \\  
	 &          &          &           &            &          &              &     & & & & & \\  
	\hline
\end{array}
$
	}
	\caption{\small{Reduced homology groups $\{\bar H_i(\cB_d^{\,\max \om\, \geq\, 3}; \Z)\}$ of the  projectivized  space of binary forms with at least one root (i.e., a line in the plane) of multiplicity $\geq 3$. The rows of the table correspond to the fixed values of $d$.}}
	\label{fig:triple}
\end{figure}

\section[(Co)homological stabilization]{(Co)homological stabilization in the univariate and bivariate situations}\label{sec:stab}



In what follows, we either consider closed subposets  $\Theta\subset \Om^\infty$ such that,   
for all $(\omega,\kappa) \in \Theta$,  
the norms $|(\om,\kappa)|$ have the same parity, or we consider closed subposets $\Theta\subset \Om$ such that,
for all $\omega \in \Theta$, 
the norms $|\om|$ have the same parity. 
We call these posets $\Theta$ \emph{closed equal parity subposets}. We say that a closed equal parity poset $\Theta$
is of parity $d$ or equivalently $d$ is of parity $\Theta$ 
if the norm of all elements of $\Theta$ has the same parity as $d$. 
Note that posets of mixed parity do not have natural geometric interpretation. 

\smallskip
We  say that $\Theta$ is  \emph{generated by  a subset $\Theta' \subseteq \Theta$}  if 
$\Theta$ is the smallest closed subposet of either $\Om^\infty$ or $\Om$ that contains $\Theta'$. If $\Theta'$ can be chosen to be finite, then 
we say that $\Theta$ is \emph{ finitely generated}. 
It is easily seen that, for any $\Theta' \subseteq \Om^\infty$ or $\Theta' \subseteq \Om$,  
there exists a unique smallest closed subposet $\Theta$ that contains $\Theta'$. Moreover, if $\Theta$ is generated by $\Theta'$, then one can 
always assume that $\Theta'$ consists of maximal elements of $\Theta$ with respect to ``$\prec$". 
In particular, being finitely generated is equivalent to  having only finitely many maximal elements.

If $\Theta$ is a closed equal parity poset, then we denote  
$\Theta_{\langle d]}:=\Theta\cap  \Om_{\langle d]}^\infty$ (resp.,
$\Theta_{\langle d]}:=\Theta\cap  \Om_{\langle d]}$).
We also introduce $\Theta_d := \Theta \cap \Om_d^\infty$ (resp.,  $\Theta_d := \Theta \cap \Om_d$). (These constructions can only lead to non-empty posets if $\Theta$ is of parity $d$.)

\begin{definition}\label{def.profinite}
We call a closed poset $\Theta\subseteq \Om^\infty$ (resp., $\Om$)  {\sf profinite} if, for all  integers $q \geq 0$, there exist only finitely many elements
$(\om,\kappa) \in \Theta$ (resp., $\om \in \Theta$) such that $|(\om,\kappa)|' \leq q$ (resp., $|\om|' \leq q$).   \hfill $\diamondsuit$
\end{definition}
Obviously, every finite $\Theta$ is profinite. In particular, for any closed $\Theta \subseteq \Om$ (resp., $\Theta \subseteq \Om^\infty$), we have that 
$\Theta_{\langle d]} = \Theta \cap \Om_{\langle d]}$ (resp., $\Theta_{\langle d]} = \Theta \cap \Om_{\langle d]}^\infty$) is finite and hence profinite. 

\smallskip

\begin{lemma}\label{lem.virtually_finite}
  Let $\Theta$ be a closed finitely generated  equal parity subposet of 
  $\Om^\infty$ (resp., of $\Om$). 
 Let $n$ be the largest norm of the maximal elements in $\Theta$.
  
 Then, for any $q \geq 0$, we have  inclusions $$\{ (\om,\kappa) \in \Theta~|~|(\om,\kappa)|'\leq q \} \subseteq  \Theta_{\langle n+2q]} \text{\; and \;} \{ \om \in \Theta~|~|\om|'\leq q \} \subseteq  \Theta_{\langle n+2q]}.$$ 
 \indent In particular, every closed finitely generated  equal parity subposet is profinite. 
\end{lemma}

\begin{proof}
  Consider the case $\Theta \subseteq \Om^\infty$ and assume that $(\omega,\kappa)$ is one of its maximal elements.
  Let $(\om',\kappa') \preceq (\om,\kappa)$ be such that $|(\om',\kappa')|' \leq q$.
  By definition any $(\omega',\kappa') \preceq (\omega,\kappa)$ can be 
  obtained from $(\om,\kappa)$ by a sequence of merge and insert operations. 
  Both operations increase $|\sim |'$ by one. In particular,  
	at most $q-|(\om,\kappa)|' \geq 0$ insertion operations
  can be applied to $(\omega,\kappa)$ without $|\sim |'$ exceeding $q$. Each merge operation preserves the norm and each insertion 
  operation increases the norm by $2$. 
  In particular,  if $(\om',\kappa') \preceq (\om,\kappa)$ and $|(\om',\kappa')|'  \leq  q$ then 
  $$|(\om',\kappa')|\leq |(\om,\kappa)|+ 2(q-|(\om,\kappa)|') = 
  |(\om,\kappa)|-2|(\om,\kappa)|'+2q \leq n+2q.$$
  The proof in case $\Theta \subseteq \Om$ is analogous. 

  The last assertion now follows from the fact that 
  $\Theta_{\langle n+2q]}$ is finite.
\end{proof}

Let $\Theta \subseteq \Om^\infty$ (resp., $\Theta \subseteq \Om$) be a closed equal parity subposet of parity $d$. 
Since each $\Theta_{\langle d]}$ is finite, it has 
only finitely many maximal elements. Let $n_1,\ldots,n_{\ell(d)}$ be the norms of the maximal elements in $\Theta_{\langle d]}$ and 
$c_1,\ldots, c_{\ell(d)}$ be  the corresponding reduced norms. 
Now we define for $d$ of parity $\Theta$ 
\begin{eqnarray}\label{eq.eta} 
\eta_\Theta(d)  : = \max_{i\,=\,1,\,\ldots \, , \,\ell(d)} (n_i -2c_i).
\end{eqnarray} 

Note that, if $\Theta$ is finitely generated, then $\eta_\Theta(d)$ is constant for all $d$ greater than or equal to the largest norm of a maximal element in $\Theta$.
In general, if $(\om,\kappa)$ is maximal in $\Theta_{\langle d]}$, then it
is also maximal in $\Theta_{\langle d']}$ for any $d' \geq d$ of parity $\Theta$, since the norm does not decrease under merge and insert operations. This implies that 
$\eta_\Theta$ is a \emph{weakly increasing} function on arguments of parity $\Theta$. 

Further, using \ref{eq.eta}, for any number $d$ of the parity $\Theta$, we define  the function
\begin{eqnarray}\label{eq.function_psi} 
\psi_\Theta(d) := \frac{1}{2}(d + \eta_\Theta(d)) =\frac{1}{2}( d+ \max_{i\,=\, 1,\,\ldots \, ,\,\ell(d)} \{n_i -2c_i\}).
\end{eqnarray}
 

Since the function $\eta_\Theta(d)$ is weakly increasing as a function for arguments $d$ of parity $\Theta$, 
it follows that $\psi_\Theta$ is strictly increasing. 
In particular, we get that $\lim_{d \rightarrow \infty} 
\psi_\Theta(d) = \infty$, the limit being taken over $d$ of parity $\Theta$.
Nevertheless, the growth of $\psi_\Theta$ depends on whether $\Theta$ is a finite, a profinite, or a general subposet.

\smallskip

\begin{lemma}
	\label{lem:profinite}
	Let $\Theta \subseteq \Om^\infty$ (resp., $\Om$) be a closed profinite equal parity poset. 
	Then $\lim_{d \to +\infty} (d-\psi_\Theta(d)) = +\infty$. 
\end{lemma}
\begin{proof}
	If $\Theta$ is finitely generated, then $\eta_\Theta(d) = c' \in \Z$ is constant for $d \gg 0$. 
	Elementary manipulations validate the claim.

	Assume $\Theta$ is not finitely generated and let $\{n_i\}_{i \in \N}, \{c_i\}_{i \in \N}$ be the sequences of 
	norms and the corresponding reduced norms of the maximal elements of $\Theta$ ordered such that the norms are weakly
	increasing. 
	Since $\Theta$ is profinite, for any $q \geq 0$ there are only finitely many $c_i$ with $c_i \leq q$. It follows
	that $\lim_{i \rightarrow \infty} c_i = \infty$.

	For every $d$ of parity $\Theta$, we have
	$\psi_\Theta(d) = \frac{1}{2} (d+n_i-2c_i)$ for some $i$ such that $n_i \leq d$. Note that $i$ may depend on $d$.

	It follows that $d-\psi_\Theta(d) = \frac{1}{2} (d - n_i) + c_i$.
	Given $q \geq 0$, let $N_q$ be the largest norm of the finitely many maximal elements of $\Theta$ with $c_i \leq q$.
	If $c_i > q$, then $d-\psi_\Theta(d) > q$; otherwise $n_i \leq N_q$ and $d-\psi_\Theta(d) \geq \frac{1}{2}(d-N_q)$.
	Hence $d-\psi_\Theta(d) \geq q$ for all $d \geq N_q+2q$ of parity $\Theta$, which proves the claim.
\end{proof}

\begin{example}
	\label{ex:profinite}
	Let $\ell \geq 1$ be a number and 
	$\Theta(\ell) \subseteq \Om^\infty$ be generated by 
	the multiplicity patterns 
	$\big\{\gamma_k = \big((\underbrace{3,\ldots,3}_{2\,k},\underbrace{1,\ldots,1}_{2\,\ell\,k}),0\big)\big\}_{k \geq 1}$.
	Then $|\gamma_k| = k\,(6+2\,\ell)$ and $|\gamma_k|' = 4\,k$. 
	It follows that $\Theta(\ell)$ 
	is profinite, but not finitely generated.
	If $d$ is even and $k\,(6+2\,\ell) \leq d < (k+1)\,(6+2\,\ell)$, then the maximal elements of $\Theta(\ell)_{\langle d]}$ are
	$\gamma_1,\ldots,\gamma_k$, and
	 $\eta_{\Theta(\ell)}(d) = \max_{k' \leq k} \{ 2\,k'\,(\ell-1)\} = 2\,k\,(\ell-1) \leq \frac{\ell-1}{\ell+3}\,d$.
	Hence $\psi_{\Theta(\ell)} (d) \leq \frac{\ell+1}{\ell+3}\, d$, and
	it follows that $d-\psi_{\Theta(\ell)}(d)
	\geq \frac{2}{\ell+3}\, d \rightarrow +\infty$. \hfill $\diamondsuit$
\end{example}

We now show that the function $\psi_\Theta$ plays a fundamental role in the homological stabilization of spaces $\cB_d^\Theta$ and $\bar{\cP}_d^\Theta$. \smallskip

%





Let $\Theta$ be a closed equal parity poset in either $\Om$ or $\Om^\infty$. 
For a number $d$ of the parity $\Theta$, we introduce the homomorphism
$$\mathsf{trunc} : \Z[\Theta_{\langle d+2]}] \to \Z[\Theta_{\langle d]}]$$ 
that is the identity on the elements of $\Theta_{\langle d+2]}$ of norm 
$\leq d$ and sends all elements of norm $d+2$ to $0$.
A simple calculation shows that 
$\mathsf{trunc}$ commutes with the differentials $\d^\infty$ (resp., $\d$) 
and thus defines a homomorphism of differential complexes 
$(\Z[\Theta_{\langle d+2]}], \d^\infty)$ and 
$(\Z[\Theta_{\langle d]}], \d^\infty)$ 
(resp., $(\Z[\Theta_{\langle d+2]}], \d)$ and $(\Z[\Theta_{\langle d]}], \d)$). 
It also lowers the homological degree by $2$. Note 
that in the univariate case (i.e. for $\Theta \subset \Om$), the strata that 
correspond to compositions in  $\Om_{\langle d+2]}$ with norm $d+2$ are 
exactly the cells that consist of  real polynomials whose roots are all real. 


\begin{proposition}\label{lem13.8x}
  Let $\Theta$ be a closed equal parity poset in $\Om$ or $\Om^\infty$
	and $d \geq 0$ a number of parity $\Theta$. 
  Then 
  \begin{itemize}
	  \item[(i)] the operation $\mathsf{trunc}$ generates the short exact sequences 
	    of differential complexes
$$0\to (\Z[\Theta_{d+2}], \d_\mathsf M) \to  (\Z[\Theta_{\langle d+2]}], \d) \stackrel{\mathsf{trunc}}{\longrightarrow} (\Z[\Theta_{\langle d]}], \d) \to 0,$$
\quad \quad \quad \quad \quad \quad \quad \quad \quad \quad \quad \quad \quad \quad  \quad or
$$ 0\to (\Z[\Theta_{d+2}], \d_\mathsf M^\infty) \to  (\Z[\Theta_{\langle d+2]}], \d^\infty) \stackrel{\mathsf{trunc}}{\longrightarrow} (\Z[\Theta_{\langle d]}], \d^\infty) \to 0,$$

\item[(ii)] if, for some $k$, the complex $(\Z[\Theta_{d+2}], \d_\mathsf M)$ or 
      the complex $(\Z[\Theta_{d+2}], \d_\mathsf M^\infty)$ is acyclic in 
      dimensions $\geq k$, then for all $j > k$, 
      the homomorphism $\mathsf{trunc}$ induces homological isomorphisms  
      $$H_{j+2}(\Z[\Theta_{\langle d+2]}], \d) \stackrel{\mathsf{trunc}_\ast}{\longrightarrow} H_j(\Z[\Theta_{\langle d]}], \d),\quad \text{or}\quad  H_{j+2}(\Z[\Theta_{\langle d+2]}], \d^\infty) \stackrel{\mathsf{trunc}_\ast}{\longrightarrow} H_j(\Z[\Theta_{\langle d]}], \d^\infty).$$

 \smallskip
\end{itemize}
\end{proposition}

\begin{proof} 
  First note that on $\ZZ[\Theta_{d+2}]$ the differential $\d$ (resp., $\d^\infty$) coincides with $\d_\sM$ (resp., $\d_{\sM}^\infty$). 
  Now the fact that the two three-term sequences of differential complexes 
	from (i) are exact follows instantly from definition of 
	$\mathsf{trunc}$.
	The long exact sequences in their homology, induced 
	by the short exact sequences of differential complexes, 
	then imply (ii).
\end{proof}

In \ref{simple_stabilization}, we will find out which posets $\Theta$ and 
integers $k$ satisfy the hypotheses of \ref{lem13.8x}(ii). 
Prior to that, we need to investigate the geometry that leads to  
the (co)homological stabilization for the spaces 
$\cB_d^\Theta$, $\cB_d^{\mathbf c\Theta}$ and 
$\bar{\cP}_d^\Theta$, $\cP_d^{\mathbf c\Theta}$ when $d$  increases. 

\smallskip

For each $(\om,\kappa) \in \Om^\infty_{\langle d]}$, consider the set ${\mathcal K}_d^{(\om,\kappa)} \subseteq 
\cB_d$ of classes of binary forms 
of degree $d$ with \emph{all zeros being real} with root multiplicity pattern 
$(\om',\kappa')\preceq (\om,\kappa)$. More precisely, we set
\begin{eqnarray*}\label{eq.K} 
  {\mathcal K}_{d}^{(\om,\kappa)} := \bigcup_{\genfrac{}{}{0pt}{}{(\om',\kappa')\, \preceq \,(\om,\kappa)}{|(\om',\kappa')| \, = \, d}} \sR_{d}^{(\om',\kappa')}. 
\end{eqnarray*}

Analogously, for a closed equal parity poset $\Theta \subseteq \Om^\infty_{\langle d]}$,  we define 
\begin{eqnarray*}\label{eq.KTheta} 
  {\mathcal K}_{d}^\Theta := \bigcup_{(\om,\kappa) \in \Theta} {\mathcal K}_d^{(\om,\kappa)}. 
\end{eqnarray*}

\begin{figure}[ht]
\begin{center}
\centerline{\includegraphics[height=1.7in,width=4.7in]{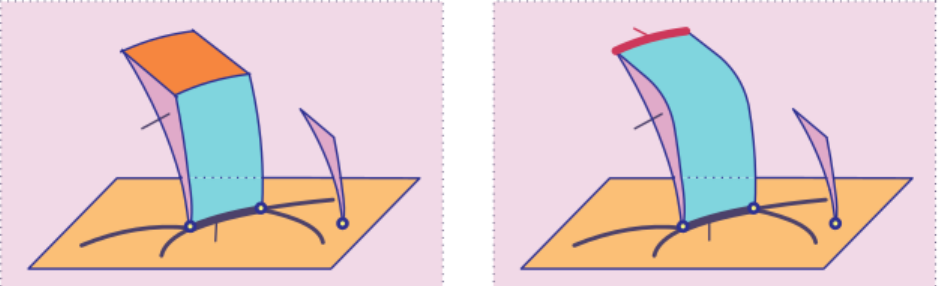}}
\end{center}
~\vskip-5.6cm
\setlength{\unitlength}{3947sp}%
\begingroup\makeatletter\ifx\SetFigFont\undefined%
\gdef\SetFigFont#1#2#3#4#5{%
  \reset@font\fontsize{#1}{#2pt}%
  \fontfamily{#3}\fontseries{#4}\fontshape{#5}%
  \selectfont}%
\fi\endgroup%
\begin{picture}(7962,2829)(1168,-3001)
\put(3200,-2600){\makebox(0,0)[lb]{\smash{{\SetFigFont{12}{14.4}{\rmdefault}{\mddefault}{\updefault}{\color[rgb]{0,0,0}(a)}%
}}}}
\put(6200,-2600){\makebox(0,0)[lb]{\smash{{\SetFigFont{12}{14.4}{\rmdefault}{\mddefault}{\updefault}{\color[rgb]{0,0,0}(b)}%
}}}}
\put(3350,-2200){\makebox(0,0)[lb]{\smash{{\SetFigFont{12}{14.4}{\rmdefault}{\mddefault}{\updefault}{\color[rgb]{0,0,0}$\sR_d^{(\om,\kappa)}$}%
}}}}
\put(6350,-2200){\makebox(0,0)[lb]{\smash{{\SetFigFont{12}{14.4}{\rmdefault}{\mddefault}{\updefault}{\color[rgb]{0,0,0}$\sR_d^{(\om,\kappa)}$}%
}}}}
\put(2840,-1800){\makebox(0,0)[lb]{\smash{{\SetFigFont{12}{14.4}{\rmdefault}{\mddefault}{\updefault}{\color[rgb]{0,0,0}$\R\P^d$}%
}}}}
\put(5800,-1800){\makebox(0,0)[lb]{\smash{{\SetFigFont{12}{14.4}{\rmdefault}{\mddefault}{\updefault}{\color[rgb]{0,0,0}$\R\P^d$}%
}}}}
\put(4200,-0600){\makebox(0,0)[lb]{\smash{{\SetFigFont{12}{14.4}{\rmdefault}{\mddefault}{\updefault}{\color[rgb]{0,0,0}$\R\P^{d+2}$}%
}}}}
\put(7200,-0600){\makebox(0,0)[lb]{\smash{{\SetFigFont{12}{14.4}{\rmdefault}{\mddefault}{\updefault}{\color[rgb]{0,0,0}$\R\P^{d+2}$}%
}}}}
\put(4100,-1100){\makebox(0,0)[lb]{\smash{{\SetFigFont{12}{14.4}{\rmdefault}{\mddefault}{\updefault}{\color[rgb]{0,0,0}$\nu^2$}%
}}}}
\put(7050,-1100){\makebox(0,0)[lb]{\smash{{\SetFigFont{12}{14.4}{\rmdefault}{\mddefault}{\updefault}{\color[rgb]{0,0,0}$\nu^2$}%
}}}}
\put(2570,-1300){\makebox(0,0)[lb]{\smash{{\SetFigFont{12}{14.4}{\rmdefault}{\mddefault}{\updefault}{\color[rgb]{0,0,0}$\sR_{d+2}^{(\om,\kappa)}$}%
}}}}
\put(5530,-1300){\makebox(0,0)[lb]{\smash{{\SetFigFont{12}{14.4}{\rmdefault}{\mddefault}{\updefault}{\color[rgb]{0,0,0}$\sR_{d+2}^{(\om,\kappa)}$}%
}}}}
	\put(3080,-0780){\makebox(0,0)[lb]{\smash{{\SetFigFont{12}{14.4}{\rmdefault}{\mddefault}{\updefault}{\color[rgb]{0,0,0}{\small $\mathcal{K}^{(\om,\kappa)}_{d+2}$}}%
}}}}
\put(5520,-0500){\makebox(0,0)[lb]{\smash{{\SetFigFont{12}{14.4}{\rmdefault}{\mddefault}{\updefault}{\color[rgb]{0,0,0}{\small $\mathcal{K}^{(\om,\kappa)}_{d+2}$}}%
}}}}
\end{picture}

\vskip-0.3cm
\caption{\small{A cell $\sR_{d}^{(\om,\kappa)}$, $|(\om,\kappa)| \leq d$, in $\cB_d \cong \R\P^d$ and the corresponding cell $\sR_{d+2}^{(\om,\kappa)}$ in $\cB_{d+2} \cong \R\P^{d+2}$. The set ${\mathcal K}_{d}^\Theta$ is shown as a graph in the plane $\R\P^d$. The normal $2$-bundle $\nu(\R\P^d, \R\P^{d+2})$ is denoted by $\nu^2$.
	In (a) the portion ${\mathcal K}_{d+2}^{(\om,\kappa)}$ of the boundary $\d\sR_{d+2}^{(\om,\kappa)}$ has dimension $\dim(\sR_{d+2}^{(\om,\kappa)}) -1$, while in (b) ${\mathcal K}_{d+2}^{(\om,\kappa)}$ has dimension smaller than $\dim(\sR_{d+2}^{(\om,\kappa)}) -1$.}} 
\label{fig.cells_d->d+2}
\end{figure}

The main idea of our arguments in the next three lemmas is captured by  \ref{fig.cells_d->d+2}. Here is their sketch. 
Any cell $\sR_d^{(\om,\kappa)}$ in $\mathcal B_d^\Theta$ produces a unique cell $\sR_{d+2}^{(\om,\kappa)}$ in 
$\mathcal B_{d+2}^\Theta$. The portion of the boundary $\d\sR_{d+2}^{(\om,\kappa)}$ that does not intersect 
$\mathcal B_d$ is exactly the locus ${\mathcal K}_{d+2}^{(\om,\kappa)}$. In the case of \ref{fig.cells_d->d+2}(a), that 
portion has the dimension $\dim(\sR_{d+2}^{(\om,\kappa)}) -1$ and contributes a term to the boundary operator 
of $\sR_{d+2}^{(\om,\kappa)}$ in the cellular chain complex. In case \ref{fig.cells_d->d+2} (b) 
the portion ${\mathcal K}_{d+2}^{(\om,\kappa)}$ has dimension strictly smaller than $\dim(\sR_{d+2}^{(\om,\kappa)}) -1$ and 
does not contribute a term. The function 
$\psi_\Theta(d+2)$ (see \ref{eq.function_psi}) gives the dimension of ${\mathcal K}_{d+2}^{\Theta}$ (an upper bound for $\dim{\mathcal K}_{d+2}^{(\om,\kappa)}$) and thus helps to discriminate 
between the ''bad case'' (a) and the ``good case'' (b). 
If $\dim(\sR_{d+2}^{(\om,\kappa)}) - 1  = \dim(\sR_{d}^{(\om,\kappa)}) + 1 >  \psi_\Theta(d+2)$, we are dealing with the good case (b) in 
which the algebraic boundary of $\sR_{d+2}^{(\om,\kappa)}$ is faithfully represented by the algebraic boundary of $\sR_d^{(\om,\kappa)}$. 
 
\begin{lemma}\label{lem.dim_K} ~ 
  \begin{itemize}
	  \item[(i)] For $(\om,\kappa) \in \Om^\infty_{\langle d]}$ or $\om \in \Om$, 
    we have 
    \begin{eqnarray*}
      \dim({\mathcal K}_{d}^{(\om,\kappa)}) =  \frac{1}{2}\Big(\,d + \big(\,|(\om,\kappa)| - 2|(\om,\kappa)|'\,\big)\,\Big)
     \end{eqnarray*}
		  or 
    \begin{eqnarray*}
      \dim({\mathcal K}_{d}^{\om}) =  \frac{1}{2}\Big(\,d + \big(\,|\om| - 2|\om|'\,\big)\,\Big).
     \end{eqnarray*}
			  
      \item[(ii)] For $d \geq 0$ and a closed equal $d$-parity subposet $\Theta \subseteq \Om^\infty_{\langle d]}$ or
		  $\Theta \subseteq \Om$, we have
    \begin{eqnarray*}\label{dim KTheta} 
      \dim({\mathcal K}_{d}^{\Theta}) =  \psi_\Theta(d).
     \end{eqnarray*}
  \end{itemize}
\end{lemma}
 
\begin{proof}
	It suffices to consider marked composition and $\Theta \subseteq \Om^\infty$.
  To settle item (i), we need to find the marked composition 
  $(\om',\kappa') \preceq (\om,\kappa)$ with the minimal reduced norm $|(\om',\kappa')|'$ 
  and the norm $|(\om',\kappa')| = d$. 
  Since $d-|(\om,\kappa)|$ is even, any sequence of 
  $\frac{1}{2}(d-|(\om,\kappa)|)$ insert operations, being applied to $(\om,\kappa)$, does the job. 
  The resulting $(\om',\kappa')$ has the reduced norm
  $|(\om',\kappa') |' = |(\om,\kappa)|'+ \frac{1}{2}(d-|(\om,\kappa)|)$. 
  Hence the dimension of $\sR_d^{(\om',\kappa')}$, and thus of $\mathcal K_d^{(\om,\kappa)}$, is given by  
  $$d-\frac{1}{2}(d - |(\om,\kappa)| + 2|(\om,\kappa)|') =   
	\frac{1}{2}(d + (|(\om,\kappa)| - 2|(\om,\kappa)|')).$$

  Item (ii) follows immediately from the definition of $\psi_\Theta(d)$ and (i).
\end{proof}

%
%


\smallskip

The next few lemmas 
prepare for  \ref{simple_stabilization} below, which, 
in the spirit of \ref{lem13.8x}, represents a geometrization 
of the $\mathsf{trunc}$ operator. 
Recall that for $\Theta \subseteq \Om^\infty$ and a number $\ell \geq 0$ we denote by
$\Theta_{|\sim|'=\ell}$ all $(\om,\kappa) \in \Theta$ such that $|(\om,\kappa)|' = \ell$.

\begin{lemma}
  \label{lem:chaincrucial}
	Let $\Theta \subseteq \Om^\infty$ (resp., $\Theta \subseteq \Om$) 
	be a closed equal parity poset, $d \geq 0$ of parity $\Theta$. For
  $j+1 \geq \psi_\Theta(d+2)$, we have an inclusion 
	$\Theta_{|\sim|' = d-j} \cap \Omega_{\langle d+2]}^\infty
	\subseteq \Theta_{\langle d]}$. 
	(resp., $\Theta_{|\sim|' = d-j} \cap \Omega_{\langle d+2]}
	\subseteq \Theta_{\langle d]}$). 
\end{lemma}
\begin{proof}
	It again suffices to consider the case $\Theta \subseteq \Om^\infty$.
	Consider $(\om,\kappa)$ such that $|(\om,\kappa)| \leq d+2$ and 
	$|(\om,\kappa)|' = d-j$, where $j+1 \geq \psi_\Theta(d+2)$. Contrary to the lemma's 
  claim, assume that $ |(\om,\kappa)| = d+2$. 
  Then 
  $$\dim(\sR_{d+2}^{(\om,\kappa)}) \, = \,
	d+2-(d-j) \, = \,  j+2 \overset{j+1 \,\geq \, \psi_\Theta(d+2)}{>} \psi_\Theta(d+2).$$ 
  By our assumptions, $\sR_{d+2}^{(\om,\kappa)} \subset 
  {\mathcal K}_{d+2}^\Theta$. From \ref{lem.dim_K}, (ii),
  it follows that 
  $$\dim(\sR_{d+2}^{(\om,\kappa)}) \leq 
	\dim(\mathcal K_{d+2}^\Theta) = \psi_\Theta(d+2).$$
  The latter yields a contradiction and thus we conclude that
	$(\om,\kappa) \in \Theta_{\langle d]}$.
\end{proof}

The next lemma claims that, in dimensions greater than $\dim({\mathcal K}_{d+2}^{\Theta})$, the boundary operators 
$\d^\infty$ from \ref{eq13.7} carry the same information in degrees $d$ and $d+2$.

\begin{lemma}
  \label{lem:iso}
	\begin{itemize}
\item[(i)]
  Let $\Theta \subseteq \Om^\infty$ be a closed equal parity poset, $d \geq 0$ of the parity $\Theta$.
  Then the following  diagram is commutative. Its horizontal arrows are the boundary
  operators $\d^\infty$ and its vertical arrows are isomorphisms 
	for $j \geq \psi_\Theta(d+2)$.
	\end{itemize}
	{\small
	\begin{center}
		\begin{tikzcd}[row sep=large, column sep=small]
			\ZZ[\Om^\infty_{\langle d+2]} \cap \Theta_{|\sim|'=d+2-(j+3)}]
			\arrow[d,"\mathsf{trunc}_\ast" ]
	\arrow[r] 
			&  
			\ZZ[\Om^\infty_{\langle d+2]} \cap \Theta_{|\sim|'=d+2-(j+2)}]
	\arrow[d,"\mathsf{trunc}_\ast" ]
	\arrow[r] 
			&  
			\ZZ[\Om^\infty_{\langle d+2]} \cap \Theta_{|\sim|'=d+2-(j+1)}]
	\arrow[d,"\mathsf{trunc}_\ast" ]
	\\
			\ZZ[\Om^\infty_{\langle d]} \cap \Theta_{|\sim|'=d-(j+1)}]
	\arrow[r] 
			&  
			\ZZ[\Om^\infty_{\langle d]} \cap \Theta_{|\sim|'=d-j}]
	\arrow[r] 
			&  
			\ZZ[\Om^\infty_{\langle d]} \cap \Theta_{|\sim|'=d-(j-1)}]
		\end{tikzcd}
	\end{center}
	}
	\begin{itemize}
\item[(ii)]
  Let $\Theta \subseteq \Om$ be a closed equal parity poset, $d \geq 0$ of the parity $\Theta$.
  Then the following  diagram is commutative. Its horizontal arrows are the boundary
  operators $\d$ and its vertical arrows are isomorphisms 
	for $j \geq \psi_\Theta(d+2)$.
	\end{itemize}
	{\small
	\begin{center}
		\begin{tikzcd}[row sep=large, column sep=small]
			\ZZ[\Om_{\langle d+2]} \cap \Theta_{|\sim|'=d+2-(j+3)}]
			\arrow[d,"\mathsf{trunc}_\ast" ]
	\arrow[r] 
			&  
			\ZZ[\Om_{\langle d+2]} \cap \Theta_{|\sim|'=d+2-(j+2)}]
	\arrow[d,"\mathsf{trunc}_\ast" ]
	\arrow[r] 
			&  
			\ZZ[\Om_{\langle d+2]} \cap \Theta_{|\sim|'=d+2-(j+1)}]
	\arrow[d,"\mathsf{trunc}_\ast" ]
	\\
			\ZZ[\Om_{\langle d]} \cap \Theta_{|\sim|'=d-(j+1)}]
	\arrow[r] 
			&  
			\ZZ[\Om_{\langle d]} \cap \Theta_{|\sim|'=d-j}]
	\arrow[r] 
			&  
			\ZZ[\Om_{\langle d]} \cap \Theta_{|\sim|'=d-(j-1)}]
		\end{tikzcd}
	\end{center}
	}
\end{lemma}

\medskip

\begin{proof}
	It suffices to prove (i).
  By \ref{lem:chaincrucial} and the inequality $j+2 = (j+1) + 1 > j \geq \psi_\Theta(d+2)$, we have
	$\Theta_{|\sim|' = d+2-(j+3)} = \Theta_{|\sim|' = d-(j+1)} \subseteq \Theta_{\langle d]}$.

  The same lemma shows that the inequality $j+1 > j \geq \psi_\Theta (d+2)$  
	implies that $\Theta_{|\sim|' = d+2-(j+2)} = \Theta_{|\sim|' = d-j} \subseteq \Theta_{\langle d]}$.

  Finally, this argument also shows that the inequality
  $(j-1)+1  = j \geq \psi_\Theta(d+2)$ implies that
	$\Theta_{|\sim|' = d+2-(j+1)}= \Theta_{|\sim|' = d-(j-1)} \subseteq \Theta_{\langle d]}$.

  Thus by definition, 
  the vertical maps $\mathsf{trunc}_*$ in the diagram are isomorphisms of $\Z$-modules.
 
  Again by \ref{lem:chaincrucial}, for any $(\om,\kappa)$ in
  $\Theta_{|\sim|' = d-j-1}$, 
  $\Theta_{|\sim|' = d-j}$, or
  $\Theta_{|\sim|' = d-j+1}$, 
  we conclude that $\sR_{d+2}^{(\om',\kappa')}$ 
  appears with coefficient $a_{(\om',\kappa')}$ in the expansion of the boundary
  $\partial \sR_{d+2}^{(\om,\kappa)}$ 
  if and only if 
  $\sR_{d}^{(\om',\kappa')}$ 
  appears with coefficient $a_{(\om',\kappa')}$ in the expansion of the boundary
  $\partial \sR_{d}^{(\om,\kappa)}$.
  Therefore the diagram is commutative. 
\end{proof}

Now we are in position to establish our main stabilization results.

\begin{theorem}[{\bf Short stabilization} $\mathbf{d \Rightarrow d+2}${\bf, the projective case}]
	\label{simple_stabilization} 
  Let $\Theta \subseteq \Om^\infty$ be a closed equal parity poset and $d$ 
  of parity $\Theta$.
  \begin{itemize}
    \item[(i)] For all  $j \geq \psi_\Theta(d+2)$, 
    there is
      an isomorphism 
      $$T_j: H_j(\cB_{d}^\Theta; \Z) \cong H_{j+2}(\cB_{d+2}^{\Theta}; \Z).$$ 
   \item[(ii)] For all $j \geq \psi_\Theta(d+2)+1$, 
      there is an isomorphism 
      $$T_j^\#: H_j(\cB_d, \cB_{d}^\Theta; \Z) \cong H_{j+2}(\cB_{d+2}, \cB_{d+2}^{\Theta}; \Z).$$
 \end{itemize}
 Both isomorphisms $T_j$ and $T_j^\#$ are delivered by the inverses $(\mathsf{trunc}_\ast)^{-1}$ of truncations $\mathsf{trunc}_\ast$.
\end{theorem}
\begin{proof} Using that $\ZZ[\Om^\infty_{\langle d+2]} \cap \Theta_{|\sim|'=d+2-j'}]$ is the
	chain group of cellular chain complex of $\cB_{d+2}^\Theta$ in dimension $j'$ and 
                         $\ZZ[\Om^\infty_{\langle d]} \cap \Theta_{|\sim|'=d+2-j'}]$ is the
	chain group of cellular chain complex of $\cB_{d}^\Theta$ in dimension $j'-2$ 
	assertion (i) follows directly from 
  \ref{lem:iso}. 

  \smallskip

  For (ii), consider the new poset $\Theta'$, consisting of all the elements 
  of $\Om^ \infty$ with the same parity as $d$. 
  Of course, for that poset, we have $\cB_d^{\Theta'_{\langle d]}} = 
  \cB_d$ and 
	$\cB_{d+2}^{\Theta'_{\langle d+2]}} = \cB_{d+2}$.
    The kernel of $\mathsf{trunc}$ in \ref{lem13.8x}(i) is now the cellular chain complex of the subspace of
  $\cB_{d+2}$ of forms with only real roots, which is homeomorphic to $\Sym^{d+2}(\R\P^1) \simeq S^1$ \cite{Mort}.
  Hence, by the long exact sequence of \ref{lem13.8x}(i),
  $\mathsf{trunc}_*: H_{j+2}(\cB_{d+2};\Z)\to H_j(\cB_d;\Z)$ is an isomorphism for $j\geq 1$.  
  Consider the long exact homology sequences of the two pairs 
  $(\cB_{d+2}, \cB_{d+2}^{\Theta})$ and $(\cB_{d}, \cB_{d}^{\Theta})$ and the 
  homomorphisms $\mathsf{trunc}_*$ connecting them.

%

	\begin{center}\begin{figure}
		\begin{tikzcd}[row sep=large]
	\cdots 
	\arrow[r] 
	\arrow[d,phantom, ""{coordinate, name=Z}]
			&  
	H_{j+2}(\cB_{d+2}^\Theta) 
	\arrow[d,"\mathsf{trunc}_\ast" near start]
	\arrow[r] 
			&  
	H_{j+2}(\cB_{d+2}) 
	\arrow[d,"\mathsf{trunc}_\ast" near start]
	\arrow[r] 
			&  
	H_{j+2}(\cB_{d+2},\cB_{d+2}^\Theta) 
	\arrow[d,"\mathsf{trunc}_\ast" near start]
	\arrow[ddlllr,rounded corners,to path={ --([xshift=2ex]\tikztostart.east)
			|- (Z)[near end]\tikztonodes|- 
			([xshift=-2ex]\tikztotarget.west)-- 
(\tikztotarget)}]
	\\
	\cdots 
	\arrow[r] 
			&  
	H_{j}(\cB_{d}^\Theta) 
	\arrow[r] 
			&  
	H_{j}(\cB_{d}) 
	\arrow[d,phantom,""{coordinate, name=Y}]
	\arrow[r] 
	                & 
        H_{j}(\cB_{d},\cB_{d}^\Theta)  
	\arrow[ddll,rounded corners,to path={ 
			--([xshift=2ex]\tikztostart.east) 
	|-(Y)[near end]\tikztonodes -|([xshift=-2ex]\tikztotarget.west) --(\tikztotarget)}]
			\\
        & 
	H_{j+1}(\cB_{d+2}^\Theta) 
	\arrow[d,"\mathsf{trunc}_\ast" ]
	\arrow[r] 
			& 
	H_{j+1}(\cB_{d+2}) 
	\arrow[d,"\mathsf{trunc}_\ast" ]
	\arrow[r] 
	                & 
        H_{j+1}(\cB_{d+2},\cB_{d+2}^\Theta) 
	\arrow[d,"\mathsf{trunc}_\ast" ]
			\\ 
        & 
	H_{j-1}(\cB_{d}^\Theta) 
	\arrow[r] 
			& 
	H_{j-1}(\cB_{d}) 
	\arrow[r] 
	                & 
        H_{j-1}(\cB_{d},\cB_{d}^\Theta) 
\end{tikzcd}
		\caption{\small{Connecting the long exact homology sequences of  
		$(\cB_d,\cB_d^\Theta)$
		and
		$(\cB_{d+2},\cB_{d+2}^\Theta)$}
		}
		\label{DIAGRAM_X} 
	\end{figure}
\end{center}

\noindent By (i), the leftmost homomorphisms $\mathsf{trunc}_\ast$ are isomorphisms in the first row for $j\ge\psi_\Theta(d+2)$ and in the second row for $j\geq \psi_\Theta(d+2)+1$; the middle ones are isomorphisms for $j\ge 1$ (see above). 
Thus, by the Five Lemma, the assertion (ii) follows. 
\end{proof}

If true, the conjecture below would deliver a better geometric understanding of the spaces that appear in the 
preceding proof.

\begin{conjecture} 
  For any closed $\Theta$, the quotient $\cB_{d+2}^\Theta\big/ \,\mathcal K_{d+2}^\Theta$ is the Thom space 
  of an orientable topological $2$-dimensional microbundle over $\cB_{d}^\Theta$. \hfill $\diamondsuit$
\end{conjecture}


The same proof as the proof of \ref{simple_stabilization} yields the following stabilization result for
spaces of polynomials.

\begin{theorem}[{\bf Short stabilization} $\mathbf{d \Rightarrow d+2}${\bf, the polynomial case}]\label{pol_simple_stabilization} 
  Let $\Theta \subseteq \Om$ be a closed equal parity poset and $d$ 
  of parity $\Theta$.
  \begin{itemize}
    \item[(i)] For all  $j \geq \psi_\Theta(d+2)$, 
    there is
      an isomorphism 
		  $$T_j: H_j(\bar{\cP}_{d}^\Theta; \Z) \cong H_{j+2}(\bar{\cP}_{d+2}^{\Theta}; \Z).$$ 
   \item[(ii)] For all $j \geq \psi_\Theta(d+2)+1$, 
      there is an isomorphism 
		  $$T_j^\#: H_j(\bar{\cP}_d, \bar{\cP}_{d}^\Theta; \Z) \cong H_{j+2}(\bar{\cP}_{d+2}, \bar{\cP}_{d+2}^{\Theta}; \Z).$$
 \end{itemize}
 Both isomorphisms $T_j$ and $T_j^\#$ are delivered by the inverses $(\mathsf{trunc}_\ast)^{-1}$ of truncations $\mathsf{trunc}_\ast$.
\end{theorem}

Next we give examples where the stabilization bound $\psi_\Theta$ 
from \ref{simple_stabilization} and \ref{pol_simple_stabilization} can be  
explicitly determined.

\begin{example} 
	For \ref{pol_simple_stabilization} we give in \ref{sec:compute} 
a long list of homology computations for $\Theta \subseteq \Om$ which are generated by a single composition. 
	In these cases, for $d \geq |\om|$, we have $\psi_\Theta(d) = \frac{1}{2}(d+|\om|-2 |\om|')$.
The reader is welcome to verify the
results from \ref{pol_simple_stabilization} in the tables. The bound in \ref{pol_simple_stabilization}(i) cannot be lowered:
for $\om = (n)$, $n \geq 3$, and $d = n$, we have $\psi_\Theta(n+2) = 2$, while
$\tilde H_1(\bar{\cP}_n^{\langle (n) \rangle};\Z) \cong \Z$ (since $\bar{\cP}_n^{\langle (n) \rangle} \cong S^1$) and
$\tilde H_3(\bar{\cP}_{n+2}^{\langle (n) \rangle};\Z) = 0$.
\end{example} 

Interesting examples, where the stabilization bound from \ref{simple_stabilization} is non-trivial can be derived from 
\ref{ex:profinite} or cases where $\Theta$ is \emph{finitely generated}. 
The next examples will not provide interesting applications of \ref{simple_stabilization}, but they show that
the stabilization bound in it is close to the dimension of $\cB_d^\Theta$. 

\begin{example}\label{ex.skeleton}
For $\ell = 0$ or $1$,  the two closed equal parity posets given by 
\begin{equation}\label{eq.skeleton_A}
	\Theta_{\ell,q} = \Om^\infty_{|\sim|' \geq q} := \Big\{(\om,\kappa) \in \Om^\infty  \,\Big|\, |(\om,\kappa)|' \geq q, \, |(\om,\kappa)| \equiv \ell \mod 2\,\Big\}; 
\end{equation}
%
are easily seen to be \emph{non}-profinite. The space $\cB_d^{\Theta_{\ell,q}}$ is the $(d-q)$-\emph{skeleton} of $\cB_d$.

By definition of $\eta_\Theta(d)$, we get $\eta_\Theta(d) = d-2q$.  Thus
	$\psi_{\Theta}(d) = d-q$ 
	and $\psi_{\Theta}(d+2) = d+2 -q$.
	Now \ref{simple_stabilization} shows that, if
	$j \geq d+2-q$, then 
	$H_j(\cB_{d}^{\Theta_{\ell,q}}; \Z) \cong H_{j+2}(\cB_{d+2}^{\Theta_{\ell,q}}; \Z).$ 
	For $q = 2$, the homology of $\cB_d^{\Theta_{\ell,q}}$ is listed in \ref{fig:skeleton} for 
	small $d$ and $\ell \equiv d \mod 2$. The list is the result of our computer-assisted calculations.
	
	Note that the stabilization from \ref{simple_stabilization} appears for
	$j \geq d$, which is two above the top non-vanishing homology group (for $j = d-1$ both groups vanish for dimension reasons).
	On the other hand, for $j = d-2$, \ref{fig:skeleton} testifies that there is no stabilization.
	 In fact, for $0 < j < d-2$,  the homology is stable, since here we are considering the skeleta of the real projective space.
	\hfill $\diamondsuit$
\begin{figure}
	{\tiny
	$
\begin{array}{c|c|c|c|c|c|c|c|c|c|c|c|c}
        d\; \backslash \; i & 1 & 2 & 3 & 4 & 5 & 6 & 7 & 8 & 9 \\
          &        &       &        &       &          &         &         &       & \\
        4 & \Z/2\Z & \ZZ^3 &        &       &          &         &         &       & \\  
          &        &       &        &       &          &         &         &       & \\
        5 & \Z/2\Z & 0     & \ZZ^7  &       &          &         &         &       & \\  
          &        &       &        &       &          &         &         &       & \\
        6 & \Z/2\Z & 0     & \Z/2\Z & \ZZ^8 &          &         &         &       & \\  
          &        &       &        &       &          &         &         &       & \\
        7 & \Z/2\Z & 0     & \Z/2\Z & 0     & \ZZ^{13} &         &         &       & \\  
          &        &       &        &       &          &         &         &       & \\
        8 & \Z/2\Z & 0     & \Z/2\Z & 0     & \Z/2\Z   & \Z^{15} &         &       & \\  
          &        &       &        &       &          &         &         &       & \\
        9 & \Z/2\Z & 0     & \Z/2\Z & 0     & \Z/2\Z   & 0       & \Z^{21} &       & \\  
          &        &       &        &       &          &         &         &       & \\
       10 & \Z/2\Z & 0     & \Z/2\Z & 0     & \Z/2\Z   & 0       & \Z/2\Z  & \Z^{24}& \\
          &        &       &        &       &          &         &         &       & \\
       11 & \Z/2\Z & 0     & \Z/2\Z & 0     & \Z/2\Z   & 0       & \Z/2\Z  & 0      & \Z^{31} \\ 
\end{array}
$
	}
	\caption{Homology of the space $\cB_d^{\Theta_{\ell, 2}}$, the $(d-2)$-skeleton of $\cB_d$.  The poset $\Theta_{\ell, 2}$ is defined in \ref{eq.skeleton_A}.
	}
	\label{fig:skeleton}
\end{figure}

\end{example}
\begin{example}\label{ex.arnold}
	Fix a parity and a number $k \geq 2$.
	Consider the closed equal parity poset $\Theta = \Theta^{\max \geq k}$ containing all
	$(\om,\kappa)$ of the chosen parity for which there is a part 
	of $\omega$ that is $\geq k$
	or $\kappa \geq k$. Thus $\cB^\Theta_d$ is the space of classes of binary forms 
	of degree $d$ with at least one $k$\textsuperscript{fold} root (see
	 \ref{th:Arnold4} and \ref{th:Va2} for some results on analogously defined spaces).
	 
	 The maximal elements of $\Theta$ are
	 the $(\omega,0)$ where $\omega$ has one part of size $k$ and all
	 other being $1$. It follows that $|(\om,0)|' = k-1$.
	 Hence $\eta_\Theta (d) = d-2(k-1)$. Therefore
	 $\psi_\Theta(d) = d-k+1$ and $\psi_\Theta(d+2) = d-k+3$. 
	Again, \ref{simple_stabilization} shows that, if 
	$j \geq d-k+3$,
	then $H_j(\cB_{d}^\Theta; \Z) \cong H_{j+2}(\cB_{d+2}^{\Theta}; \Z)$. 
	Since $\dim \cB_d^\Theta = d-k+1$, the stabilization bound is again two above the top dimension. Looking at \ref{fig:triple}, it seems that here a stabilization below that bound is present. \hfill $\diamondsuit$
\end{example}

Let us describe, in terms of divisors, the effect of increasing the degree $d \rightarrow d+2$ on the spaces of projectivized binary forms.

Let $\e : \cB_{d} \rightarrow \cB_{d+2}$ be the map which sends the class of a binary form 
$f(x, y)$ in $\cB_d$ to the class of the binary form $(x^2+y^2) f(x, y)$ in $\cB_{d+2}$. 
Since $x^2+y^2$ has no real roots in $\C\P^1$, for any $(\om,\kappa) \in \Om^\infty_{\langle d]}$  
the image of $\mathring{\sR}_d^{(\om,\kappa)}$ under $\e$ is a subset of $\mathring{\sR}_{d+2}^{(\om,\kappa)}$.

Consider the $2$-sphere $\C\P^1$ and the involution $J: \C\P^1 \to \C\P^1$, given by the complex conjugation $J[z : w] = [\bar z : \bar w]$. The fixed point set of $J$ is the circle $\bar \R = \R\P^1$. We can identify $\cB_{d}$ with the space of positive $J$-symmetric divisors on $\C\P^1$ of degree $d$. Then $\e$ translates into a map $\e^\bullet: [Sym^d(\C\P^1)]^J \to [Sym^{d+2}(\C\P^1)]^J$ that takes each divisor $D$ to the divisor $D + [\mathsf i: 1] + [\mathsf {-i} : 1]$. 
Thus the real divisors $\e^\bullet(D)_{\bar\R}$ and $D_{\bar\R}$ share the same combinatorial type $(\om,\kappa)$. 

\smallskip

Our next goal is to give alternative purely \emph{homological} definitions of the maps 
$(\mathsf{trunc})^\ast$, $(\mathsf{trunc})_\ast$. The results will allow us to study stabilization in cohomology of
the spaces $\cB_d^{\bfc \Theta}$. With this goal in mind, let us describe a list of hypotheses, needed for such a definition.

\smallskip
{\sf Hypotheses A.}  Let $M$ be a closed $(n+k)$-dimensional manifold, equipped with a structure
 of a finite CW-complex, and $N \subseteq M$ a closed $n$-dimensional submanifold whose normal bundle is orientable. Assume that the intersections of $N$ with the cells of $M$ induce a structure of a cellular complex on $N$. 

  Let $K \subset M$ be a compact CW-subcomplex of $M$. Put $L := K \cap N$
	and let \ $M \setminus K$ and $N \setminus L$ be connected.

Assume that, for each open cell $e$ of $M$, the intersection $e \cap N$ is 
either a single  cell such that $\dim(e) = \dim(e \cap N) +k$, or $e \cap N = \emptyset$; moreover, $\d e \cap N = \d(e \cap N)$, and the incidence numbers of the cells of $N$ (with the induced orientations) agree with those of the corresponding cells of $M$. 
We assume also that $M$ admits a triangulation $T$ that is consistent with its cellular structure: so that $K, L, N$ and their cellular skeletons are closed subcomplexes of this triangulation.
\smallskip  

Put $\Lambda := \Z[\pi_1(M\setminus K)]$. Let $\A$ be a right $\Lambda$-module. The embedding $\e: N \setminus L \subset M\setminus K$, with the help of the  homomorphism $\e_\ast: \pi_1(N \setminus L) \to \pi_1(M\setminus K)$,  converts the $\Lambda$-module $\A$ into a $\Z[\pi_1(N\setminus L)]$-module $\e^\ast(\A)$. 
\smallskip

Let $[M] \in H_{n+k}(M; \Z^t)$ and $[N] \in H_{n}(N; \Z^t)$ be the fundamental  classes of $M$ and $N$, respectively.

We denote by $\mathcal C_\bullet(\sim)$ and $\mathcal C^\bullet(\sim)$ the 
$\Z$-modules of cellular chains of a cellular complex.\hfill $\diamondsuit$

\begin{lemma}\label{lem.diagram} 
 Under {\sf Hypotheses A},  the following diagram is commutative:
\begin{eqnarray}\label{eq.DIAGRAM}
H^{j}(N, L;\, \e^\ast(\A)) & \stackrel{(\mathsf{trunc})^\ast}{\longrightarrow} & H^{j + k}(M, K;\, \A) \nonumber\\
\cong\, \downarrow [N] \cap &    & \cong\, \downarrow [M] \cap \nonumber \\
H_{n - j}(N \setminus L;\, \e^\ast(\A^t)) &  \stackrel{\e_\ast}{\longrightarrow} & H_{n - j}(M \setminus K;\, \A^t), 
\end{eqnarray}
where the vertical isomorphisms are the Poincar\'{e} duality maps, $\mathcal D_N = [N] \cap \sim$ and $\mathcal D_M = [M] \cap \sim$. The horizontal homomorphisms $(\mathsf{trunc})^\ast$ are induced by the duals of the chain homomorphisms 
$\mathsf{trunc}: \mathcal C_{\ast + k}(M, K;\, \A) \to \mathcal C_{\ast }(N, L;\, \e^\ast(\A)),$ the latter ones are produced by the intersection of cells $\{e\}$ in $M$ with $N$. 
The proof below contains the exact definition of the cohomological transfer $(\mathsf{trunc})^\ast$ from the diagram \ref{eq.DIAGRAM}.
\end{lemma}

\begin{proof} 
 We use the second barycentric subdivision of the triangulation $T$ to form regular neighborhoods of closed subcomplexes of $T$. 
\smallskip

Since, for each open cell $e$ of $M$, the intersection $e \cap N$ is either a single cell such that $\dim(e) = \dim(e \cap N) +k$, or $e \cap N = \emptyset$, the restriction of the cellular structure on $M$ to $N$ gives rise to a cellular structure on $N$ and produces the correspondence $\mathsf{trunc}: e \leadsto e \cap N$. Since $\nu(N, M)$ is orientable, its orientation helps to use the orientation of $e$ to orient $e \cap N$. 
The correspondence $\mathsf{trunc}$ leads to a well-defined homomorphisms on the \emph{cellular} chain level: 
$$\mathsf{trunc}: \mathcal C_{j+k}(K;\, \A)  \to  \mathcal C_{j}(L;\, \e^\ast(\A)), \; \mathsf{trunc}: \mathcal C_{j + k}(M, K;\, \A) \to  \mathcal C_{j}(N, L;\, \e^\ast(\A))$$
We  use the identity $[(\mathsf{trunc})^\ast(\a)](e) := \a(e \cap N)$ for any cochain $\a \in \mathcal C^{j }(N, L;\, \e^\ast(\A))$ and a $(j+k)$-dimensional cell $e \subset M \setminus K$ to introduce the dual operator $(\mathsf{trunc})^\ast$. 
\smallskip

Our next goal is to define the truncation $H^{j}(N, L;\, \e^\ast(\A)) \stackrel{(\mathsf{trunc})^\ast}{\longrightarrow} H^{j + k}(M, K;\, \A)$ in pure homological terms, since the cap products in \ref{eq.DIAGRAM} are not directly amenable to the cellular structure $\{e\}$ on $M$. 

Let $M_p, K_p$ denote the $p$-skeletons of $M, K$ in the cellular structure $\{e\}$,  and $N_p, L_p$ denote the $p$-skeletons of $N, L$ in the cellular structure $\{N \cap e\}$. 
To simplify the notations, put $$K_p^\sharp := K_p \cup L, \;\; M_p^\sharp := M_p \cup K, \;\; N_p^\sharp := N_p \cup L.$$ We omit the relevant coefficient systems $\A$, $\A^t$ and $\e^\ast(\A)$, $\e^\ast(\A^t)$ in our notations. 
\smallskip

Consider the following filtrations by closed subsets: 
\begin{eqnarray*}
N & := & N_n^\sharp \supset \ldots \supset N_p^\sharp \supset N_{p-1}^\sharp  \supset \ldots \supset N_0^\sharp, \\
M & := & M_{n+k}^\sharp \supset \ldots \supset M_p^\sharp \supset M_{p-1}^\sharp \supset \ldots \supset M_0^\sharp, 
\end{eqnarray*}
 
These filtrations give rise to two spectral sequences, whose $E_1$-pages  
are: 
\begin{eqnarray*}
 E^{p, q}_1(N, L) = H^{p+q}(N_p^\sharp, N_{p-1}^\sharp) \text{ converges to } \Rightarrow   H^\ast(N, L),  \\
 E^{p, q}_1(M, K) = H^{p+q}(M_p^\sharp, M_{p-1}^\sharp) \text{ converges to }\Rightarrow   H^\ast(M, K), 
 \end{eqnarray*}
Therefore, in terms of the given cellular structures $\{e \cap N\}$  in $N$ and $\{e\}$ in $M$, we get $E^{p, 0}_1(N, L) \cong \mathcal C^p(N, L)$, and $E^{p, q}_1(N, L) = 0$ for all $q > 0$ since the cohomology with compact support of each open $p$-cell arises only in dimension $p$.  
Similarly,  $E^{p, 0}_1(M, K) \cong \mathcal C^p(M, K) $, and $E^{p, q}_1(M, K) = 0$ for all $q > 0$. 




Note that $N_p^\sharp \setminus N_{p-1}^\sharp = N_p \setminus (N_{p-1} \cup L_p)$ and $M_{p+k}^\sharp \setminus M_{p+k-1}^\sharp = M_{p+k} \setminus (M_{p+k-1} \cup K_{p+k})$ are unions of open $p$-cells and $(p+k)$-cells, respectively. 
So we have a homomorphism $(\mathsf{trunc}^\ast)^{p, q}: E_1^{p,q}(N, L) \to E_1^{p+k,q}(M, K)$, defined by:
\begin{eqnarray}\label{eq.the_p-diagram}
 H^{p}(N_p^\sharp, N_{p-1}^\sharp)  \cong \mathcal C^p(N, L)& \stackrel{(\mathsf{trunc})^\ast}{\longrightarrow}  & \mathcal C^{p+k}(M, K)  \cong  H^{p+k}(M_{p+k}^\sharp, M_{p+k-1}^\sharp). 
\end{eqnarray}
It relies on the coherent orientations of the pairs $(e, e\cap N)$ of cells and amounts to attaching the value $\a(e \cap N)$ to each $(p+k)$-cell $e$ in $M$ and each $p$-cochain $\a$ that represents an element of $H^{p}(N_{p}^\sharp, N_{p-1}^\sharp)$.
%
%

%
It remains to verify that $\mathsf{trunc}^\ast: E_1^{p,q}(N, L) \to E_1^{p+k,q}(M, K)$ commutes with the differentials of the two spectral sequences. 
%
%
This verification is in line with the standard arguments that show the equivalence of singular and cellular (co)homology.  We incorporate the differentials from both spectral sequences in the commutative diagrams as on page 227 of \cite{Mu} (with homology being replaced with cohomology and the arrows being reversed). The top of the two diagrams is based on the relative cohomology of different pairs from the quadruple $$M_{p+k+1}^\sharp \supset M_{p+k}^\sharp \supset M_{p+k-1}^\sharp \supset M_{p+k-2}^\sharp,$$
and the bottom one on the relative cohomology of different pairs from the quadruple $$N_{p+1}^\sharp \supset N_{p}^\sharp \supset N_{p-1}^\sharp \supset N_{p-2}^\sharp.$$

The degree raising differentials $\delta^{p-1}_N: H^{p-1}(N_{p-1}^\sharp, N_{p-2}^\sharp) \to H^{p}(N_{p}^\sharp, N_{p-1}^\sharp)$ and $\delta^{p}_N: H^{p}(N_{p}^\sharp, N_{p-1}^\sharp) \to H^{p+1}(N_{p+1}^\sharp, N_{p}^\sharp)$ are taken from the long exact cohomology sequence of the triples $\{N_{p}^\sharp \supset N_{p-1}^\sharp \supset N_{p-2}^\sharp\}$ and $\{N_{p+1}^\sharp \supset N_{p}^\sharp \supset N_{p-1}^\sharp\}$, respectively. Similarly, we have the differentials $\delta^{p+k-1}_M: H^{p+k-1}(M_{p+k-1}^\sharp, M_{p+k-2}^\sharp) \to H^{p+k}(M_{p+k}^\sharp, M_{p+k-1}^\sharp)$ and $\delta^{p+k}_M: H^{p+k}(M_{p+k}^\sharp, M_{p+k-1}^\sharp) \to H^{p+k+1}(M_{p+k+1}^\sharp, M_{p+k}^\sharp)$.  By the lemma hypotheses ``$\d(e) \cap N = \d(e \cap N)$", these differentials coincide with the coboundary operators on the cochain level. By the diagrams as on page 227 of \cite{Mu}, 
the differentials $$\delta^{p-1}_N = \d^{p-1}: \mathcal C^{p-1}(N, L) \to \mathcal C^{p}(N, L), \quad \delta^{p}_N = \d^{p}: \mathcal C^{p}(N, L) \to \mathcal C^{p+1}(N, L).$$ 
 $$\delta^{p+k-1}_M = \d^{p+k-1}: \mathcal C^{p+k-1}(M, K) \to \mathcal C^{p+k}(M, K),$$    $$\delta^{p+k}_M = \d^{p+k}: \mathcal C^{p+k}(M, K) \to \mathcal C^{p+k+1}(M, K)$$
commute with the $\mathsf{trunc}^\ast$ maps: that is,  $\mathsf{trunc}^\ast \circ \delta^{p-1}_N = \delta^{p+k-1}_M \circ \mathsf{trunc}^\ast$ and $\mathsf{trunc}^\ast\circ\delta^{p}_N = \delta^{p+k}_M\circ \mathsf{trunc}^\ast$. Therefore (see \cite{Mu}, Theorems 39.4 and 39.5) the map from \ref{eq.the_p-diagram} gives rise to a homomorphism 
$$\mathsf{trunc}^\ast: H^p(N, L) = \frac{\mathsf{ker}(\delta^{p}_N)}{\mathsf{im}(\delta^{p-1}_N)}   \longrightarrow   \frac{\mathsf{ker}(\delta^{p+k}_M)}{\mathsf{im}(\delta^{p+k-1}_M)} = H^{p+k}(M, K),$$
which is now well-defined on the level of cohomology (and not only on the cellular level).
As a result, all the maps in diagram \ref{eq.DIAGRAM} are well-defined homologically.




\smallskip

Thus to validate the commutativity of \ref{eq.DIAGRAM}, we seek to verify that, for a given $\a\in H^{p}(N, L; \e^\ast(\A))$, we get $[M] \cap \mathsf{trunc}^\ast(\a) = \e_\ast([N] \cap \a).$ 

\smallskip

Let $U(N) \subset M$ be a regular closed neighborhood of $N$ in $M$ (defined via the second barycentric subdivision of $T$) and $\pi: U \to N$ its canonical retraction; thanks to the properties of the triangulation $T$,  $\pi^{-1}(L) = K \cap U$. Let $u \in H^k(U, \d U)$ be a relative cocycle, Poincar\'{e}-dual in $U$ to the fundamental cycle $\e_\ast([N]) \in H_n(U; \Z^t) \cong H_n(N; \Z^t)$. Then $[M] \cap u = \e_\ast([N])$ and $\pi^\ast(\a) \cup u$ is a relative cocycle of dimension $p+k$ in $(U, \d U)$. Since $\a$ vanishes on singular chains on $L$, $\pi^\ast(\a) \cup u$ vanishes on singular chains on $\pi^{-1}(L) \cup \d U =  (K \cap U) \cup \d U$. So $\pi^\ast(\a) \cup u$ extends to a relative cocycle on $(M, K)$ that vanishes on $K \cup (M \setminus \mathsf{int}(U))$. Abusing notations, we denote this extension by $\pi^\ast(\a) \cup u$ as well. \smallskip

For any $(p+k)$-cell $e$, we denote by $e_U$ the intersection $e \cap U$. Note that the interior of $e_U$ is an open $(p+k)$-cell, and $e_U$ is a regular neighborhood of $e \cap N$ in $e$. Thanks to the constructions of $T$ and $U$, it is equipped with the canonical projection $\pi: e_U \to e \cap N$ whose fiber is a $k$-dimensional $\mathsf{PL}$-ball.  
Then, for any relative $p$-cocycle $\a$ that vanishes on the boundary $\d(e \cap N)$, we get $(u \cup \pi^\ast(\a))(e_U) = \a(e \cap N)$. Thus we conclude that the restrictions $(u \cup \pi^\ast(\a))|_{e_U} = \mathsf{trunc}^\ast(\a)|_{e_U}$. \smallskip

Next, we represent the fundamental class $[M]$ as a $(n+k)$-chain $\mathcal M = \sum_{\{e:\; \dim(e) = n+k\}} e$, where each $(n+k)$-cell $e$ is viewed, with the help of $T$, as a singular $(n+k)$-chain, a sum of singular $(n+k)$-simplicies.

Then, using the basic properties of cap and cup products ($\cap$ and $\cup$) (see \cite{Sp}), we get $$\mathcal M \cap \mathsf{trunc}^\ast(\a) =  
\sum_{\{e:\; \dim(e) = n+k\}} e \cap (\pi^\ast(\a) \cup u)  $$ 
$$= \sum_{\{e:\; \dim(e) = n+k\}}  (e \cap u) \cap \pi^\ast(\a)= \sum_{\{e:\; \dim(e) = n+k\}}  e_u \cap \pi^\ast(\a),$$
where the $n$-chain $e_u :=  e \cap u$. Applying $\pi_\ast$ to the last expression, we get
$$\pi_\ast\big(\sum_e  e_u \cap \pi^\ast(\a)\big) = \sum_e  \pi_\ast(e_u \cap \pi^\ast(\a)) = \sum_e  (\pi_\ast(e_u) \cap \a) =  \big(\sum_e  \pi_\ast(e_u)\big) \cap \a =  [N] \cap \a.
$$

Since $\pi: U \to N$ is a homotopy and thus homology equivalence, the cycles $[M] \cap \mathsf{trunc}^\ast(\a)$ and $\e_\ast([N] \cap \a)$ are homologous in $H_{n - j}(M \setminus K;\, \A^t)$. This completes the proof of the lemma.
\end{proof}
\smallskip




The commutativity of diagram \ref{eq.DIAGRAM} implies instantly the following fact.

\begin{corollary}\label{cor.epsilon_iso}  Under the hypotheses and notations of \ref{lem.diagram}, the transfer \hfill \break $H^{j }(N, L;\, \e^\ast(\A))  \stackrel{(\mathsf{trunc})^\ast}{\longrightarrow}  H^{j + k}(M, K;\, \A)$ is an isomorphism if and only if  the natural homomorphism $H_{n - j}(N \setminus L;\, \e^\ast(\A^t))   \stackrel{\e_\ast}{\longrightarrow} H_{n - j }(M \setminus K;\, \A^t)$ is. \hfill $\diamondsuit$
\end{corollary}


\medskip

Now we are in position to consider stabilization in cohomology for the spaces $\cB_d^{\bfc \Theta}$.

\begin{theorem}[{\bf Dual short stabilization} $\mathbf{d \Rightarrow d+2}$, projective case]\label{simple_stabilizationc} 
Let $\Theta \subseteq \Om^\infty$ be a closed equal parity 
poset.

Then for $d$ of parity $\Theta$ the embedding $\e:\, \cB_{d}^{\mathbf c\Theta} \subset \cB_{d+2}^{\mathbf c\Theta}$, induces an isomorphism 
\begin{eqnarray}\label{eq.epsilon is iso} 
\e_\ast: H_k(\cB_{d}^{\bfc\Theta};\, \e^\ast(\Z^t)) \cong H_k(\cB_{d+2}^{\bfc\Theta};\, \Z^t)
\end{eqnarray}
for all $k \leq d-1-\psi_\Theta(d+2)$. 
\smallskip

If $d \equiv 1 \mod 2$, then one can replace in \ref{eq.epsilon is iso} the local coefficient system $\Z^t$ with the trivial $\Z[\pi_1(\cB_{d+2}^{\mathbf c\Theta})]$-module $\Z$.
\end{theorem}

\begin{proof} 
We apply \ref{lem.diagram} and \ref{cor.epsilon_iso} to the spaces of $\cB_d$ and $\cB_d^\Theta$, where $N = \cB_d$, $M = \cB_{d+2}$, $K = \cB_{d+2}^\Theta$, and $L = \cB_d^\Theta$.  In order to apply \ref{lem.diagram}, we use the fundamental fact 
that stratified algebraic sets admit triangulations \cite{Jo}, \cite{Har}, \cite{Har1}, consistent with their stratifications.

Let $\mathbb A = \Z$ in \ref{cor.epsilon_iso}. We then get that $$\e_\ast: H_{d-j}(\cB_{d}^{\mathbf c\Theta}; \e^\ast(\Z^t)) \rightarrow H_{d-j}(\cB_{d+2}^{\mathbf c\Theta}; \Z^t)$$ is an isomorphism if and only if 
	$$\mathsf{trunc}^\ast : H^{j}(\cB_{d}, \cB_{d}^{\Theta};\, \Z) \rightarrow H^{j + 2}(\cB_{d+2}, \cB_{d+2}^{\Theta};\, \Z)$$ is an isomorphism.  Dualizing the short exact sequence of \ref{lem13.8x}(i) (all modules are free), and using that the kernel complex has no
cochains in degrees above $\psi_\Theta(d+2)$ (see \ref{lem.dim_K}), we see that
$\mathsf{trunc}^\ast: H^m(\cB_d^\Theta;\Z) \to H^{m+2}(\cB_{d+2}^\Theta;\Z)$ is an isomorphism for $m \geq \psi_\Theta(d+2)$.
Since the space of forms of degree $d+2$ with only real roots is homeomorphic to $\Sym^{d+2}(\R\P^1)\simeq S^1$,
$\mathsf{trunc}^\ast: H^m(\cB_d;\Z) \to H^{m+2}(\cB_{d+2};\Z)$ is an isomorphism for $m \geq 1$.
The Five Lemma, applied to the cohomology sequences of the pairs $(\cB_d,\cB_d^\Theta)$ and $(\cB_{d+2},\cB_{d+2}^\Theta)$,
shows that $\mathsf{trunc}^\ast : H^{j}(\cB_{d}, \cB_{d}^{\Theta};\, \Z) \to H^{j + 2}(\cB_{d+2}, \cB_{d+2}^{\Theta};\, \Z)$
is an isomorphism for all $j \geq \psi_\Theta(d+2)+1$. Therefore, by \ref{cor.epsilon_iso},
$\e_\ast: H_k(\cB_{d}^{\mathbf c\Theta}; \e^\ast(\Z^t)) \cong H_k(\cB_{d+2}^{\mathbf c\Theta}; \Z^t)$ is an isomorphism
for all $k = d-j$ with $j \geq \psi_\Theta(d+2)+1$, that is, for all $k \leq d-1-\psi_\Theta(d+2)$.
\smallskip

 For $d > 1$, $\e_\ast: \pi_1(\R\P^d) \to \pi_1(\R\P^{d+2})$ is an isomorphism between two copies of $\Z_2$. Note that the Stiefel-Whitney homomorphism $w: \pi_1(\cB_{d}^{\mathbf c\Theta}) \to \Z_2$ factors through
$\tilde w = w_1(\R\P^d): \pi_1(\R\P^d) \to  \Z_2$ (an isomorphism for $d$ even, and zero for $d$ odd), and the Stiefel-Whitney
homomorphism $w^\#: \pi_1(\cB_{d+2}^{\mathbf c\Theta}) \to \Z_2$ factors through $\tilde w^\# = w_1(\R\P^{d+2})$. Since
$w_1(\R\P^{d+2})$ restricts to $w_1(\R\P^d)$, we get $\e^\ast(\Z^t) \cong \Z^t$ as $\Z[\pi_1(\cB_{d}^{\mathbf c\Theta})]$-modules.
If $d \equiv 1 \mod 2$, then $\R\P^d$ and $\R\P^{d+2}$ are orientable, and one can replace in \ref{eq.epsilon is iso}
the local coefficient system $\Z^t$ with the trivial $\Z[\pi_1(\cB_{d+2}^{\mathbf c\Theta})]$-module $\Z$.
\end{proof}

\begin{corollary}[{\bf Dual long stabilization} $\mathbf{d \Rightarrow \infty}${\bf, the projective case}] Let $\Theta \subseteq \Om^\infty$ be a profinite (see \ref{def.profinite})
	closed equal parity poset. \smallskip 

Then, for each $j >0$, there is a sufficiently big number $d$ of parity $\Theta$ such that the homomorphism 
$(\e_{d, d'})_\ast: H_j(\cB_{d}^{\mathbf c\Theta}; \e^\ast(\Z^t)) \cong H_j(\cB_{d'}^{\mathbf c\Theta}; \Z^t)$ is an isomorphism for 
all  $d' \geq d$ of parity $\Theta$.
\end{corollary}
\begin{proof}
	The assertion follows immediately from \ref{lem:profinite} and \ref{simple_stabilizationc}.
\end{proof}

Next, we derive the analogues of these stabilizations for the spaces of polynomials.

\begin{theorem}[{\bf Dual short stabilization}, $\mathbf{d \Rightarrow d+2}${\bf, the polynomial case}]
	\label{pol_simple_stabilizationc} 
Let $\Theta$ be a closed equal parity subposet of $\mathbf\Om$. 
Then for $d$ of parity $\Theta$ the embedding $\e:\, \cP_{d}^{\mathbf c\Theta} \subset \cP_{d+2}^{\mathbf c\Theta}$, induces an isomorphism 
\begin{eqnarray*}
\e_\ast: H_k(\cP_{d}^{\bfc\Theta};\, \Z) \cong H_k(\cP_{d+2}^{\bfc\Theta};\, \Z)
\end{eqnarray*}
for all $k \leq d-1-\psi_\Theta(d+2)$. 

\end{theorem}

\begin{proof} 
As in the proof of \ref{simple_stabilizationc}, \ref{lem.diagram} and \ref{cor.epsilon_iso} show that
$\e_\ast: H_{d-j}(\cP_{d}^{\bfc\Theta};\Z) \to H_{d-j}(\cP_{d+2}^{\bfc\Theta};\Z)$ is an isomorphism if and only if
$\mathsf{trunc}^\ast: H^j(\bar{\cP}_d, \bar{\cP}_d^\Theta;\Z) \to H^{j+2}(\bar{\cP}_{d+2}, \bar{\cP}_{d+2}^\Theta;\Z)$ is.
By the dual of \ref{lem13.8x}(i), $\mathsf{trunc}^\ast$ is an isomorphism on $\tilde H^m(\bar{\cP}^\Theta;\Z)$ for
$m \geq \psi_\Theta(d+2)$; on $\tilde H^m(\bar{\cP}_d;\Z) = \tilde H^m(S^d;\Z)$ it is an isomorphism for all $m$, since the
kernel is the cochain complex of the one-point compactification of the closed cone of polynomials with only real roots, a disk.
The Five Lemma, applied to the long exact cohomology sequences of the pairs, gives the claim for $j \geq \psi_\Theta(d+2)+1$,
that is, for $k = d-j \leq d-1-\psi_\Theta(d+2)$. 
\end{proof}

\begin{corollary}[{\bf Dual long stabilization}
	$\mathbf{d \Rightarrow \infty}${\bf, the polynomial case}]
	\label{cor.stable_Theta_polys} 
For any \emph{profinite} closed $\Theta$ and for each $j > 0$, there is a sufficiently big number $d$ of parity $\Theta$ such that the homomorphism 
$(\e_{d, d'})_\ast: H_j(\cP_{d}^{\mathbf c\Theta}; \Z) \to H_j(\cP_{d'}^{\mathbf c\Theta}; \Z)$ is an isomorphism for all $d' \geq d$ of parity $\Theta$. 
\end{corollary}

\begin{proof}
	The assertion follows immediately from \ref{lem:profinite} and \ref{pol_simple_stabilizationc}.
\end{proof}


We conclude by presenting a special case where we can describe the homotopy type of $\cP_d^{\bfc \Theta}$. 

\begin{proposition}\label{th.bouquet} For $k > 2$ and a fixed parity let $\Theta$ be the closed poset of compositions from 
  $\Om_{|\sim|' \geq k}$ of that parity. 
	Then $\mathcal P_{d}^{\mathbf c\Theta}$ has the homotopy type of a bouquet of $(k-1)$-spheres.
	The number of spheres in the bouquet equals the absolute value of the reduced Euler characteristic of $\bar{\cP}_d^\Theta$.

\end{proposition}

\begin{proof} 
Since by its definition $\bar{\mathcal P}_{d}^{\Theta}$ is the $(d-k)$-skeleton of $\bar{\mathcal P}_d \cong S^d$, the cohomology of $\bar{\mathcal P}_{d}^{\Theta}$ is torsion-free and is concentrated in a single dimension $d-k$. By the Alexander duality,  the homology of  $\mathcal P_{d}^{\mathbf c\Theta}$ is torsion-free and is concentrated in a single dimension $k-1$. The space $\mathcal P_{d}^{\mathbf c\Theta}$ is simply-connected for $k > 2$. By \cite[Proposition 4C.1]{Ha}, this implies that $\mathcal P_{d}^{\mathbf c\Theta}$
has the homotopy type of a bouquet of $(k-1)$-spheres.  
By Alexander duality the number of spheres in the bouquet 
	equals the absolute value of the reduced Euler characteristic of $\bar{\cP}_d^\Theta$.
 \end{proof}
 
%
%


\section{Computational results}\label{sec:computer}
\label{sec:compute}

In conclusion, let us state somewhat surprising results of one computer-assisted computation.  

For a given $\om \in \mathbf\Om$, we denote by $\langle \om \rangle$ the minimal closed poset that contains $\om$.
Below, for $d \le 13$, we list all compositions $\om$ for which the space $\bar{\cP}_d^{\langle \om \rangle}$ is \emph{homologically nontrivial}. 
In fact, for all $d \le 13$, every homologically nontrivial $\bar{\cP}_d^{\langle \om \rangle}$ is a \emph{homology} sphere! Moreover, at least for all $\om$'s with $|\om|' > 2$, every space $\cP_d^{\mathbf c \langle \om \rangle}$ is either contractible or homotopy equivalent to a sphere.
Unfortunately, the reason for such phenomena is unknown to us... \smallskip

\begin{figure}
\begin{center}\footnotesize\renewcommand{\arraystretch}{0.95}
\begin{tabular}{c|c||c|c||c|c}
$d$ & Codimension & $\om$ & $i$ with $\widetilde{H}_i = \Z$ & $\om$ & $i$ with $\widetilde{H}_i = \Z$ \\
  &   &       &   & & \\
\hline 
\hline 
  &   &       &   & & \\
4 & 0 & $(1^2)$ & 3 & &\\
  & 1 & $(2)$ & 3 & & \\
  & 3 & $(4)$ & 1 & & \\
  &   &       &   & & \\
\hline 
  &   &       &   & & \\
5 & 0 & $(1^3)$ & 4 & & \\
  & 4 & $(5)$   & 1 & & \\
  &   &       &   & & \\
\hline 
  &   &       &   & & \\
6 & 0 & $(1^4)$ & 5 & $(1^2)$ & 5 \\
  & 1 & $(2)$ & 5 & & \\ 
  & 5 & $(6)$ & 1 & & \\ 
  &   &       &   & & \\
\hline 
  &   &       &   & & \\
7 & 0 & $(1^5)$ & 6 & $(1^3)$ & 6 \\
  & 6 & $(7)$   & 1 &       & \\
  &   &       &   & & \\
\hline 
  &   &       &   & & \\
8 & 0 & $(1^6)$ & 7 & $(1^4)$ & 7 \\
  &   & $(1^2)$ & 7 & & \\
  & 1 & $(1,2,1)$ & 3 & $(2)$ & 7  \\
  & 2 & $(1,3)$ & 3 & $(3,1)$ & 3 \\
  & 3 & $(4)$ & 3 & & \\
  & 7 & $(8)$ & 1 & & \\
  &   &       &   & & \\
\hline 
  &   &       &   & & \\
9 & 0 & $(1^7)$&8 & $(1^5)$ & 8 \\
  &   & $(1^3)$   & 8 & & \\
  & 1 & $(1^2,2,1)$ & 4 & $(1,2,1^2)$ & 4 \\
  & 2 & $(1^2,3)$ & 4 & $(1,3,1)$ & 4 \\
  &   & $(3,1^2)$ & 4 & & \\
  & 8 & $(9)$ & 1 & & \\
  &   &       &   & & \\
\hline 
  &   &       &   & & \\
10 & 0 & $(1^8)$ & 9 & $(1^6)$ & 9 \\
   &   & $(1^4)$ & 9 & $(1^2)$ & 9 \\
   & 1 & $(1^3,2,1)$ & 5 & $(1^2,2,1^2)$ & 5 \\
   &   & $(1,2,1^3)$ & 5 & $(2)$ & 9 \\
   & 2 & $(1^3,3)$ & 5 & $(1^2,3,1)$ & 5 \\
   &   & $(1,3,1^2)$ & 5 & $(3,1^3)$ & 5 \\
   & 9 & $(10)$ & 1 & & \\
  &   &       &   & & \\
\hline 
\end{tabular}
\end{center}
\caption{\small{ The list of the $\om$'s and of the corresponding unique homological degrees $i = i(\om)$ for which $\widetilde{H}_i(\bar{\cP}^{\langle \omega \rangle}_d; \Z)= \ZZ$ in case  $4 \leq d \leq 10$. For  $\om$'s and $i$'s not listed above the homology vanishes.}} 
\end{figure}


\begin{figure}
\begin{center}\small
\begin{tabular}{c|c||c|c||c|c}
$d$ & Codimension & $\om$ & $i$ with $\widetilde{H}_i= \Z$ & $\om$ & $i$ with $\widetilde{H}_i = \Z$ \\
  &   &       &   & & \\
\hline 
\hline 
  &   &       &   & & \\
11 & 0 & $(1^9)$ & 10 & $(1^7)$ & 10 \\
   &   & $(1^5)$ & 10 & $(1^3)$ & 10 \\
   & 1 & $(1^4,2,1)$ & 6 & $(1^3,2,1^2)$ & 6 \\
   &   & $(1^2,2,1^3)$ & 6 & $(1,2,1^4)$ & 6 \\
   & 2 & $(1^4,3)$ & 6 & $(1^3,3,1)$ & 6 \\
   &   & $(1^2,3,1^2)$ & 6 & $(1,3,1^3)$ & 6 \\
   &   & $(3,1^4)$ & 6 & $(1,2,1,2,1)$ & 4 \\
   & 3 & $(1,2,1,3)$ & 4 & $(3,1,2,1)$ & 4 \\
   & 4 & $(3,1,3)$ & 4 & & \\
   & 10 & $(11)$ & 1 & & \\
   &   &       &   & & \\
\hline 
   &   &       &   & & \\
12 & 0 & $(1^{10})$ & 11 & $(1^8)$ & 11 \\
   &   & $(1^6)$ & 11 & $(1^4)$ & 11 \\
   &   & $(1^2)$ & 11 & & \\
   & 1 & $(1^5,2,1)$ & 7 & $(1^4,2,1^2)$ & 7 \\
   &   & $(1^3,2,1^3)$ & 7 & $(1^2,2,1^4)$ & 7 \\
   &   & $(1,2,1^5)$ & 7 & $(1,2,1)$ & 5 \\
   &   & $(2)$ & 11 & & \\
   & 2 & $(1^5,3)$ & 7 & $(1^4,3,1)$ & 7 \\
   &   & $(1^3,3,1^2)$ & 7 & $(1^2,3,1^3)$ & 7 \\
   &   & $(1,3,1^4)$ & 7 & $(3,1^5)$ & 7 \\
   &   & $(1,2,1^2,2,1)$ & 5 & $(1,2,1,2,1^2)$ & 5 \\
   &   & $(1^2,2,1,2,1)$ & 5 & & \\
   &   & $(1,2^2,1)$ & 3 & $(1,3)$ & 5 \\
   &   & $(3,1)$ & 5 & & \\
   & 3 & $(1^2,2,1,3)$ & 5 & $(1,2,1^2,3)$ & 5 \\
   &   & $(1,2,1,3,1)$ & 5 & $(1,3,1,2,1)$ & 5 \\
   &   & $(3,1^2,2,1)$ & 5 & $(3,1,2,1^2)$ & 5 \\
   &   & $(1,2,3)$ & 3 & $(3,2,1)$ & 3 \\
   &   & $(1,4,1)$ & 3 & $(4)$ & 5 \\
   & 4 & $(1,3,1,3)$ & 5 & $(3,1,3,1)$ & 5 \\
   &   & $(1,5)$ & 3 & $(5,1)$ & 3 \\
   &   & $(3^2)$ & 3 & & \\
   & 5 & $(6)$ & 3 & & \\
   & 11 & $(12)$ & 1 \\
   & & & &\\
\hline 
\end{tabular}
\end{center}
\caption{\small{ The list of the $\om$'s and of the corresponding unique homological degrees $i = i(\om)$ for which $\widetilde{H}_i(\bar{\cP}^{\langle \omega \rangle}_d; \Z)= \ZZ$ for $d = 11, 12$. For  $\om$'s and $i$'s not listed above, the homology vanishes.}} 
\end{figure}

\begin{figure}
\begin{center}
\begin{tabular}{c|c||c|c||c|c}
$d$ & Codimension & $\om$ & $i$ with $\widetilde{H}_i = \Z$ & $\om$ & $i$ with $\widetilde{H}_i = \Z$ \\
  &   &       &   & & \\
\hline
\hline 
  &   &       &   & & \\
13 &  0 & $(1^{11})$ & 12 & $(1^9)$ & 12 \\
   &    & $(1^7)$ & 12 & $(1^5)$ & 12 \\
   &    & $(1^3)$ & 12 & & \\
   &  1 & $(1^6,2,1)$ & 8 & $(1^5,2,1^2)$ & 8 \\
   &    & $(1^4,2,1^3)$ & 8 & $(1^3,2,1^4)$ & 8 \\
   &    & $(1^2,2,1^5)$ & 8 & $(1,2,1^6)$ & 8 \\
   &    & $(1^2,2,1)$ & 6 & $(1,2,1^2)$ & 6 \\
   &  2 & $(1^6,3)$ & 8 & $(1^5,3,1)$ & 8 \\
   &    & $(1^4,3,1^2)$ & 8 & $(1^3,3,1^3)$ & 8 \\
   &    & $(1^2,3,1^4)$ & 8 & $(1,3,1^5 )$ & 8 \\
   &    & $(3,1^6)$  & 8 & $(1^2,3)$ & 6 \\
   &    & $(1,3,1)$ & 6 & $(3,1^2)$ & 6 \\
   &    & $(1^3,2,1,2,1)$ & 6 & $(1^2,2,1^2,2,1)$ & 6 \\
   &    & $(1^2,2,1,2,1^2)$ & 6 & $(1,2,1^3,2,1)$ & 6 \\
   &    & $(1,2,1^2,2,1^2)$ & 6 & $(1,2,1,2,1^3)$ & 6 \\
   &    & $(1^2,2^2,1)$ & 4 & $(1,2^2,1^2)$ & 4 \\
   & 3  & $(1^3,2,1,3)$ & 6 & $(1^2,2,1^2,3)$ & 6 \\
   &    & $(1^2,2,1,3,1)$ & 6 & $(1^2,3,1,2,1)$ & 6 \\
   &    & $(1,2,1^2,3,1)$ & 6 & $(1,2,1,3,1^2)$ & 6 \\
   &    & $(1,3,1^2,2,1)$ & 6 & $(1,3,1,2,1^2)$ & 6 \\
   &    & $(3,1^2,2,1^2)$ & 6 & $(3,1,2,1^3)$ & 6 \\
   &    & $(1^2,2,3)$ & 4 & $(1,2,3,1)$ & 4 \\
   &    & $(1,3,2,1)$ & 4 & $(3,2,1^2)$ & 4 \\
   &    & $(1^2,4,1)$ & 4 & $(1, 4,1^2) $ & 4 \\
   & 4  & $(1^2,3,1,3)$ & 6 & $(1,3,1,3,1)$  & 6 \\
   &    & $(3,1,3,1^2)$ & 6 & $(1,3^2)$ & 4 \\
   &    & $(3^2,1)$ & 4 & $(1^2,5)$ & 4 \\
   &    & $(1,5,1)$ & 4 & $(5,1^2)$ & 4  \\
   & 12 & $(13)$ & 1\\ 
   & & &\\
\hline 
\end{tabular}
\end{center}
\caption{\small{The list of the $\om$'s and of the corresponding unique homological degrees $i = i(\om)$ for which $\widetilde{H}_i(\bar{\cP}^{\langle \omega \rangle}_d; \Z)= \ZZ$ for $d = 13$. For the rest of $\om$'s and $i$'s, the homology vanishes.}} 
\end{figure}
 
%
	%
%


\end{document}